\documentclass[11pt,twoside, leqno]{article}

\usepackage{amssymb}
\usepackage{amsmath}
\usepackage{mathrsfs}
\usepackage{amsthm}
\usepackage{txfonts}
\usepackage{color}

\usepackage{indentfirst}

\allowdisplaybreaks

\def\ls{\lesssim}
\def\gs{\gtrsim}
\def\fz{\infty}

\renewcommand{\r}{\right}
\newcommand{\lf}{\left}

\def\red{\color{red}}

\def\ls{\lesssim}
\def\gs{\gtrsim}

\def\paz{{\partial}}

\def\supp{{\mathop\mathrm{\,supp\,}}}

\def\rr{{\mathbb R}}

\def\rn{{{\rr}^n}}
\def\zz{{\mathbb Z}}
\def\nn{{\mathbb N}}
\def\cc{{\mathbb C}}

\newcommand{\wz}{\widetilde}

\newcommand{\ca}{{\mathcal A}}

\newcommand{\cd}{{\mathcal D}}

\newcommand{\cm}{{\mathcal M}}

\newcommand{\cs}{{\mathcal S}}

\def\az{\alpha}
\def\lz{\lambda}

\def\dz{\delta}

\def\epz{\epsilon}

\def\bz{\beta}

\def\fai{\varphi}
\def\gz{{\gamma}}
\def\bgz{{\Gamma}}

\def\tz{\theta}
\def\sz{\sigma}

\def\wz{\widetilde}

\def\ls{\lesssim}
\def\gs{\gtrsim}

\def\ol{\overline}

\def\boz{\Omega}

\def\uc{{\varepsilon}}

\def\divz{{{\mathop\mathrm {div}}}}

\def\bbmo{{{\mathop\mathrm {BMO}}}}

\def\hs{\hspace{0.3cm}}

\def\dfrac{\displaystyle\frac}

\newtheorem{theorem}{Theorem}[section]
\newtheorem{lemma}[theorem]{Lemma}
\newtheorem{corollary}[theorem]{Corollary}

\theoremstyle{definition}
\newtheorem{remark}[theorem]{Remark}
\newtheorem{definition}[theorem]{Definition}
\def\supp{{\mathop\mathrm{\,supp\,}}}
\def\diam{{\mathop\mathrm{diam\,}}}
\def\dist{{\mathop\mathrm{\,dist\,}}}
\def\loc{{\mathop\mathrm{loc}}}
\def\dive{{\mathop\mathrm{div\,}}}
\def\lfz{\lfloor}
\def\rfz{\rfloor}
\numberwithin{equation}{section}

\begin{document}

\title{\Large\bf Heat Kernels and Hardy Spaces on Non-Tangentially Accessible Domains
with Applications to Global Regularity of Inhomogeneous Dirichlet Problems
\footnotetext{\hspace{-0.35cm} 2020 {\it Mathematics Subject
Classification}. {Primary 42B35; Secondary 35J05, 35J25, 47B06.}
\endgraf{\it Key words and phrases}. NTA domain, heat kernel, Hardy space, divergence form elliptic operator,
global gradient estimate, Dirichlet problem.
\endgraf This project is partially supported by the National Natural Science Foundation
of China (Grant Nos. 11871254, 12071431, 11971058 and 12071197),
the National Key Research and Development Program of China
(Grant No.\ 2020YFA0712900) and the Fundamental Research Funds for the
Central Universities (Grant No. lzujbky-2021-ey18).}}
\author{Sibei Yang and Dachun Yang\,\footnote{Corresponding
author, E-mail: \texttt{dcyang@bnu.edu.cn}/{\red January 10, 2022}/Final version.}}
\date{ }
\maketitle

\vspace{-0.8cm}

\begin{center}
\begin{minipage}{13.5cm}\small
{{\bf Abstract.} Let $n\ge2$ and $\Omega$ be a bounded non-tangentially accessible domain
(for short, NTA domain) of $\mathbb{R}^n$. Assume that $L_D$ is a second-order divergence form elliptic operator
having real-valued, bounded, measurable coefficients on $L^2(\Omega)$ with the Dirichlet boundary condition.
The main aim of this article is threefold. First, the authors prove that the heat kernels
$\{K_t^{L_D}\}_{t>0}$ generated by $L_D$ are H\"older continuous. Second, for any $p\in(0,1]$,
the authors introduce the `geometrical' Hardy space $H^p_r(\Omega)$ by restricting any element of
the Hardy space $H^p(\mathbb{R}^n)$ to $\Omega$, and show that, when $p\in(\frac{n}{n+\delta_0},1]$, $H^p_r(\Omega)=H^p(\Omega)=H^p_{L_D}(\Omega)$ with equivalent quasi-norms, where $H^p(\Omega)$
and $H^p_{L_D}(\Omega)$ respectively denote the Hardy space on $\Omega$ and the Hardy space associated
with $L_D$, and $\delta_0\in(0,1]$ is the critical index of the H\"older continuity for the kernels
$\{K_t^{L_D}\}_{t>0}$. Third, as applications, the authors obtain the global gradient estimates in both $L^p(\Omega)$,
with $p\in(1,p_0)$, and $H^p_z(\Omega)$, with $p\in(\frac{n}{n+1},1]$, for the inhomogeneous
Dirichlet problem of second-order divergence form elliptic equations on bounded NTA domains,
where $p_0\in(2,\infty)$ is a constant depending only on $n$, $\Omega$, and the coefficient matrix of $L_D$.
Here, the `geometrical' Hardy space $H^p_z(\Omega)$ is defined by restricting any element of
the Hardy space $H^p(\mathbb{R}^n)$ supported in $\overline{\Omega}$ to $\Omega$, where $\overline{\Omega}$
denotes the closure of $\Omega$ in $\mathbb{R}^n$. It is worth pointing out that the range $p\in(1,p_0)$
for the global gradient estimate in the scale of Lebesgue spaces $L^p(\Omega)$ is sharp and the above
results are established without any additional assumptions on both the coefficient matrix of $L_D$,
and the domain $\Omega$.}
\end{minipage}
\end{center}

\vspace{0.1cm}

\section{Introduction\label{s1}}

 The study of elliptic value problems on non-smooth domains of $\rn$
has a long history (see, for instance, \cite{d20,jk95,k94} and the references therein).
In recent years, the research of the global regularity for elliptic equations
with rough coefficients on non-smooth domains of $\rn$ has aroused great interest
(see, for instance, \cite{bw04,d20,dk10,dk12,dk18,dl21,g12,sh18,sh05a}). The global regularity
estimates of elliptic equations with rough coefficients on the non-smooth domain $\boz$ of $\rn$
in the scale of Lebesgue spaces $L^p(\boz)$, with $p\in(1,\fz)$, have been extensively studied
in the existing literatures (see, for instance, the recent survey article \cite{d20},
the monograph \cite{sh18}, and the references therein). However, there exist very few
literatures on global regularity  estimates of elliptic equations with rough coefficients
on the non-smooth domain $\boz$ of $\rn$ in the scale of Hardy spaces $H^p(\boz)$ with $p\in(0,1]$.

Let $n\ge2$ and $\Omega$ be a bounded non-tangentially accessible domain (for short,
NTA domain) of $\mathbb{R}^n$. Assume that $L_D$ is a second-order divergence
form elliptic operator having real-valued, bounded, measurable coefficients on $L^2(\Omega)$ with the
Dirichlet boundary condition. The main aim of this article is threefold. First, we prove that the heat
kernels $\{K_t^{L_D}\}_{t>0}$ generated by $L_D$ are H\"older continuous. Second, for any $p\in(0,1]$,
we introduce the `geometrical' Hardy space $H^p_r(\Omega)$ by restricting any element
of the Hardy space $H^p(\mathbb{R}^n)$ to $\Omega$, and show that, when $p\in(\frac{n}{n+\delta_0},1]$,
$H^p_r(\Omega)=H^p(\Omega)=H^p_{L_D}(\Omega)$ with equivalent quasi-norms, where $H^p(\Omega)$ and $H^p_{L_D}(\Omega)$
respectively denote the Hardy space on $\Omega$ and the Hardy space associated with $L_D$,
and $\delta_0\in(0,1]$ is the critical index of the H\"older continuity for the kernels $\{K_t^{L_D}\}_{t>0}$.
Third, as applications, for the inhomogeneous Dirichlet boundary value problem
\begin{equation}\label{eq1.1}
\begin{cases}
-\dive(A\nabla u)=f\ \ &\text{in}\ \ \boz,\\
u=0 \ \ &\text{on}\ \ \partial\boz,
\end{cases}
\end{equation}
where the matrix $A$ is real-valued, bounded, and measurable, and satisfies the
uniform ellipticity condition [see \eqref{eq1.3} below for the details],
and $\partial\boz$ denotes the boundary of $\boz$, we obtain the global gradient estimates
of the weak solution $u$ in both Lebesgue spaces $L^p(\Omega)$,
with $p\in(1,p_0)$, and Hardy spaces $H^p_z(\Omega)$, with $p\in(\frac{n}{n+1},1]$,
where $p_0\in(2,\infty)$ is a constant depending only on $n$, $\Omega$, and the coefficient
matrix $A$. Here, the `geometrical' Hardy space $H^p_z(\Omega)$ is defined by restricting
any element of the Hardy space $H^p(\mathbb{R}^n)$ supported in $\overline{\Omega}$ to $\Omega$,
where $\overline{\Omega}$ denotes the closure of $\Omega$ in $\mathbb{R}^n$. Meanwhile,
it is worth pointing out that the range $p\in(1,p_0)$ of $p$ for the global gradient estimate in
the scale of the Lebesgue space $L^p(\Omega)$ is sharp [see Remark \ref{r1.4}(i) below for the details].

Compared with the global regularity estimate of elliptic equations on the non-smooth
domain $\boz$ of $\rn$ in Lebesgue spaces $L^p(\boz)$ established in \cite{aq02,bw04,d96,dk10,sh05a},
we obtain the global regularity estimate for the Dirichlet problem \eqref{eq1.1} without
any additional assumptions on both the coefficient matrix $A$ and the domain $\boz$.
Recall that the global gradient estimate in $L^p(\boz)$ with any given $p\in(1,\fz)$ for
the Dirichlet problem \eqref{eq1.1}, with $f$ replaced by $\dive (\mathbf{f})$,
was established by Di Fazio \cite{d96}, under the assumptions that $A\in\mathrm{VMO}\,(\rn;\rr^{n^2})$
(see, for instance, \cite{s75}) and $\partial\boz\in C^{1,1}$, which was weakened to
$\partial\boz\in C^{1}$ by Auscher and Qafsaoui \cite{aq02}. Moreover, the global gradient
estimate in $L^p(\boz)$ with any given $p\in(1,\fz)$ for the problem \eqref{eq1.1},
with $f$ replaced by $\dive (\mathbf{f})$, was obtained by Byun and Wang \cite{bw04},
under the assumptions that $A$ satisfies the $(\dz,R)$-BMO condition (see, for instance,
\cite{bw04} or Definition \ref{d2.4} below for its definition)
for sufficiently small $\dz\in(0,\fz)$, and that $\boz$ is a bounded
Reifenberg flat domain of $\rn$ (see, for instance, \cite{r60,t97} or Remark \ref{r2.2}(i) below
for its definition). Furthermore, for the Dirichlet problem \eqref{eq1.1} with $f$ replaced by
$\dive (\mathbf{f})$, the global gradient estimate in $L^p(\boz)$ with any given
$p\in(1,\fz)$ was established by Dong and Kim \cite{dk10,dk12},
under the assumptions that $A$ has partial sufficiently small $\mathrm{BMO}$
coefficients and that $\boz\subset\rn$ is a bounded Lipschitz domain with small Lipschitz constant, or
a bounded Reifenberg flat domain. Meanwhile, for the problem \eqref{eq1.1} with $f$ replaced by
$\dive (\mathbf{f})$, the global gradient estimate in $L^p(\boz)$, with any given
$p\in(\frac32-\uc,3+\uc)$ when $n\ge3$, or $p\in(\frac43-\uc,4+\uc)$ when $n=2$, was obtained
by Shen \cite{sh05a}, under the assumptions that $A\in\mathrm{VMO}\,(\rn;\rr^{n^2})$ and
that $\boz\subset\rn$ is a bounded Lipschitz domain, where $\uc\in(0,\fz)$ is a constant
depending only on $n$ and $\boz$.

Moreover, NTA domains considered in this article were originally introduced by Jerison
and Kenig \cite{jk82} when studying the boundary behavior of harmonic functions.
We point out that NTA domains have a wide generality and contain Lipschitz domains,
$\mathrm{BMO}_1$ domains, Zygmund domains, quasi-spheres, and some Reifenberg flat domains
as special examples (see, for instance, \cite{jk82,kt97,t97}). Furthermore, NTA domains
are closely related to the theory of quasi-conformal mappings (see, for instance, \cite{jk82,j81}
and the references therein).

To describe the main results of this article, we first recall some necessary notions.
Let $\boz$ be a bounded NTA domain of $\rn$ as in Definition \ref{d2.1} below (see also \cite{jk82})
and $p\in(0,\fz)$. Recall that the \emph{Lebesgue space} $L^p(\Omega)$ is defined by setting
\begin{align*}
L^p(\Omega):=\lf\{f\ \text{is measurable on}\ \Omega: \
\|f\|_{L^p(\Omega)}:=\lf[\int_{\boz}
|f(x)|^p\,dx\r]^{\frac1p}<\fz\r\}.
\end{align*}
Moreover, for any given $m\in\nn$, let
\begin{equation}\label{eq1.2}
L^p(\boz;\rr^m):=\lf\{\mathbf{f}:=(f_1,\,\ldots,\,f_m):\ \text{for any}
\ i\in\{1,\,\ldots,\,m\},\ f_i\in L^p(\boz)\r\}
\end{equation}
and
$$\|\mathbf{f}\|_{L^p(\boz;\rr^m)}:=\sum_{i=1}^m\|f_i\|_{L^p(\boz)}.
$$
Denote by $W^{1,p}(\boz)$ the \emph{Sobolev space on $\boz$} equipped with the \emph{norm}
$$\|f\|_{W^{1,p}(\boz)}:=\|f\|_{L^p(\boz)}+\|\nabla f\|_{L^p(\boz;\rn)},$$
where $\nabla f$ is the \emph{distributional gradient} of $f$. Furthermore, $W^{1,p}_{0}(\boz)$
is defined to be the \emph{closure} of $C^{\fz}_{\mathrm{c}} (\boz)$ in $W^{1,p}(\boz)$, where
$C^{\fz}_{\mathrm{c}}(\boz)$ denotes the set of all \emph{infinitely differentiable functions on
$\boz$ with compact support contained in $\boz$}.

For any given $x\in\boz$, let $A(x):=\{a_{ij}(x)\}_{i,j=1}^n$ denote
an $n\times n$ matrix with real-valued, bounded, and measurable entries.
Then $A$ is said to satisfy the \emph{uniform ellipticity condition}
if there exists a positive constant $\mu_0\in(0,1]$ such that,
for any $x\in\boz$ and $\xi:=(\xi_1,\,\ldots,\,\xi_n)\in\rn$,
\begin{equation}\label{eq1.3}
\mu_0|\xi|^2\le\sum_{i,j=1}^na_{ij}(x)\xi_i\xi_j\le \mu_0^{-1}|\xi|^2.
\end{equation}
Denote by $L_D$ the \emph{maximal-accretive operator} (see, for instance, \cite[p.\,23, Definition 1.46]{o05}
for the definition) on $L^2 (\boz)$ with the largest domain $\cd(L_D)\subset W^{1,2}_{0}(\boz)$ such that,
for any $f\in \cd(L_D)$ and $g\in W^{1,2}_{0}(\boz)$,
\begin{equation*}
(L_Df,g)=\int_{\boz}A(x)\nabla f(x)\cdot\nabla g(x)\,dx,
\end{equation*}
where $(\cdot,\cdot)$ denotes the \emph{interior product} in $L^2(\boz)$.
In this sense, for any $f\in \cd(L_D)$, we write
\begin{equation}\label{eq1.4}
L_Df:=-\divz(A\nabla f).
\end{equation}

Let $\{K^{L_D}_{t}\}_{t>0}$ be the kernels of the semigroup $\{e^{-tL_D}\}_{t>0}$.
By \cite[Corollary 3.2.8]{d89} (see also \cite{at01a}), we find that there exist
positive constants $C$ and $c$ such that, for any $t\in(0,\fz)$ and $x,\,y\in\boz$,
\begin{equation}\label{eq1.5}
\lf|K^{L_D}_{t}(x,y)\r|\le\frac{C}{t^{n/2}}\exp\lf\{-\frac{|x-y|^2}{ct}\r\}.
\end{equation}
Furthermore, it is worth pointing out that the upper and the lower bound estimates,
and the H\"older continuity of the heat kernels play a key roles in the study of
the well-posedness of some parabolic partial differential equations, real-variable
characterizations of some function spaces, and some Sobolev-type inequalities
(see, for instance, \cite{d89,gs11,s10}).

Now, we state the main results of this article as follows; see Definitions \ref{d2.4} and
\ref{d2.2} below, respectively, for the definitions of both the $(\gz,R)$-$\mathrm{BMO}$ condition
and the $(\gz,\sz,R)$ quasi-convex domain.

\begin{theorem}\label{t1.1}
Let $n\ge2$, $\boz$ be a bounded {\rm NTA} domain of $\rn$, the real-valued, bounded, and
measurable matrix $A$ satisfy \eqref{eq1.3}, and $L_D$ be as in \eqref{eq1.4}. Denote by $\{K^{L_D}_t\}_{t>0}$
the heat kernels generated by $L_D$.
\begin{itemize}
\item[\rm(i)] Then there exists a constant $\dz_0\in(0,1]$, depending only on $n$, $A$,
and $\boz$, such that, for any given $\dz\in(0,\dz_0)$, there exist constants
$C,\,c\in(0,\fz)$ such that, for any $t\in(0,\fz)$ and $x,\,y_1,\,y_2\in\boz$ with $|y_1-y_2|\le\sqrt{t}/2$,
\begin{equation}\label{eq1.6}
\lf|K^{L_D}_{t}(x,y_1)-K^{L_D}_{t}(x,y_2)\r|\le\frac{C}{t^{n/2}}
\lf[\frac{|y_1-y_2|}{\sqrt{t}}\r]^{\dz}\exp\lf\{-\frac{|x-y_1|^2}{ct}\r\}.
\end{equation}
\vspace{-0.25cm}
\item[\rm(ii)] For any given $\dz_0\in(0,1]$, there exists a constant $\gz_0\in(0,\fz)$, depending
only on $\dz_0$, $n$, and $\boz$, such that, if $A$ satisfies the $(\gz,R)\mbox{-}\bbmo$ condition
and $\boz$ is a $(\gz,\sz,R)$ quasi-convex domain for some $\gz\in(0,\gz_0)$, $\sigma\in(0,1)$,
and $R\in(0,\fz)$, then, for any given $\dz\in(0,\dz_0)$, there exist constants $C,\,c\in(0,\fz)$
such that, for any $t\in(0,\fz)$ and $x,\,y_1,\,y_2\in\boz$ with $|y_1-y_2|\le\sqrt{t}/2$,
\eqref{eq1.6} holds true.
\end{itemize}
\end{theorem}

\begin{remark}\label{r1.1}
We point out that, when $\boz$ is a bounded Lipschitz domain of $\rn$,
the conclusion of Theorem \ref{t1.1}(i) is well known (see, for instance, \cite{at01a}).
Moreover, when $\boz$ is a bounded semi-convex domain of $\rn$ (see, for instance,
\cite{mmmy10,mmy10} or Remark \ref{r2.2}(ii) below for the details), and $L_D:=-\Delta_D$ with
$\Delta_D$ being the Laplace operator with the Dirichlet boundary condition on $\boz$,
Theorem \ref{t1.1}(ii) was obtained in \cite[Lemma 2.7]{dhmmy13}. Recall that the bounded
semi-convex domain $\boz$ is a $(\gz,\sz,R)$ quasi-convex domain for any $\gz\in(0,1)$,
some $\sigma\in(0,1)$, and some $R\in(0,\fz)$ (see, for instance, \cite{yyy20}).
Thus, Theorem \ref{t1.1}(ii) essentially improves \cite[Lemma 2.7]{dhmmy13} by weakening the assumptions
on both the matrix $A$ and the domain $\boz$.
\end{remark}

When $n\ge3$, we prove Theorem \ref{t1.1}(i) by using an upper estimate for the Green function associated with
$L_D$ in terms of distance functions (see, for instance, \cite[Remark 4.9]{hk07}), the Harnack inequality
(see, for instance, \cite[Theorem 8.22]{gt01}), and the functional calculus associated with $L_D$.
Precisely, using the upper estimate for the Green function of $L_D$ in terms of distance functions, and
the Harnack inequality, and borrowing some ideas from the proof of Gr\"{u}ter and Widman \cite[Theorem (1.9)]{gw82},
we prove the H\"older continuity of the Green function associated with $L_D$. Moreover, applying the
H\"older continuity of the Green function, and the functional calculus associated with $L_D$, and borrowing some
ideas from the proofs of Duong et al. \cite[Lemmas 2.6 and 2.7]{dhmmy13}, we prove Theorem \ref{t1.1}(i) in the case of $n\ge3$.
When $n=2$, we show Theorem \ref{t1.1}(i) via establishing the global gradient estimate for the Dirichlet problem
\eqref{eq1.1} in $L^p(\boz)$ with some $p\in(2,\fz)$ (see Lemma \ref{l3.6} below), and using the Sobolev embedding
theorem. Furthermore, we prove Theorem \ref{t1.1}(ii) by establishing the global gradient estimate for the
Dirichlet problem \eqref{eq1.1} in $L^p(\boz)$ for sufficiently large $p\in(n,\fz)$ (see Lemma \ref{l3.9} below),
and applying the Sobolev embedding theorem.

Next, we recall the definitions of the Hardy space $H^p(\rn)$,
the `geometrical' Hardy spaces $H^p_z(\boz)$ and $H^p(\boz)$, the Hardy space
$H^p(\boz)$, and the Hardy space $H^p_{L_D}(\boz)$ associated with $L_D$.

Denote by $\cs(\rn)$ the \emph{space of all Schwartz functions} equipped
with the well-known topology determined by a countable family of
norms, and by $\cs'(\rn)$ its \emph{dual space} (namely, the space of all
\emph{tempered distributions}) equipped with the weak-$\ast$ topology.
Let $\cd(\boz)$ denote the \emph{space of all
infinitely differentiable functions with compact support in $\boz$}
equipped with the inductive topology, and $\cd'(\boz)$ its
\emph{topological dual} equipped with the weak-$\ast$
topology, which is called the \emph{space of distributions on $\boz$}.

In what follows, for any $x\in\rn$ and $r\in(0,\fz)$, we always let $B(x,r):=\{y\in\rn:\ |y-x|<r\}$.

\begin{definition}\label{d1.1}
Let $p\in(0,1]$ and $\boz$ be a domain of $\rn$.
\begin{itemize}
\item[(i)] The \emph{Hardy space $H^p(\rn)$} is defined to be the set of
all $f\in\cs'(\rn)$ such that $f^+\in L^p(\rn)$
equipped with the \emph{quasi-norm} $\|f\|_{H^p(\rn)}:=\|f^+\|_{L^p(\rn)}$,
where the \emph{radial maximal function} $f^+$ of $f$ is defined by setting,
for any $x\in\rn$,
$$f^+(x):=\sup_{t\in(0,\fz)}\lf|e^{-t\Delta}(f)(x)\r|.
$$
Here, $\{e^{-t\Delta}\}_{t>0}$ denotes the heat semigroup generated by the \emph{Laplace
operator} $\Delta$ on $\rn$.
\vspace{-0.25cm}
\item[(ii)] The \emph{Hardy space $H^p_{z}(\boz)$} is defined by setting
\begin{equation*}
H^p_{z}(\boz):=\lf\{f\in H^p(\rn):\ f=0\ \text{on}\
\ol{\boz}^\complement\r\}/\{f\in H^p(\rn):\ f=0\ \text{on}\ \boz\}.
\end{equation*}
Here and thereafter, $\ol{\boz}$ and $\ol{\boz}^{\complement}$ denote, respectively,
the \emph{closure} of $\boz$ in $\rn$, and the \emph{complementary set} of $\overline{\boz}$
in $\rn$. Moreover, the \emph{quasi-norm} of the element in $H^p_{z}(\boz)$ is
defined to be the quotient norm, namely, for any $f\in H^p_{z}(\boz)$,
$$\|f\|_{H^p_{z}(\boz)}:=\inf\lf\{\|F\|_{H^p(\rn)}:\
F\in H^p(\rn), \ F=0\ \text{on}\
\ol{\boz}^\complement,\ \text{and}\ F|_{\boz}=f\r\},$$
where the infimum is taken over all $F\in H^p(\rn)$ satisfying
$F=0$ on $\ol{\boz}^\complement$, and $F=f$ on $\boz$.
\vspace{-0.25cm}
\item[(iii)] A distribution $f\in\cd'(\boz)$ is said to belong to the
\emph{Hardy space $H^p_{r}(\boz)$} if $f$ is the restriction to $\boz$ of
a distribution $F$ in $H^p(\rn)$, namely,
\begin{align*}
H^p_{r}(\boz):=\,&\{f\in \mathcal{D}'(\boz):\ \text{there exists an}\
F\in H^p(\rn)\ \text{such that}\ F|_{\boz}=f\}\\
=\,&H^p(\rn)/\{F\in H^p(\rn):\ F=0\ \text{on}\ \boz\}.
\end{align*}
Moreover, for any $f\in H^p_{r}(\boz)$, the \emph{quasi-norm}
$\|f\|_{H^p_{r}(\boz)}$ of $f$ in $H^p_{r}(\boz)$
is defined by setting
$$\|f\|_{H^p_{r}(\boz)}:=\inf\lf\{\|F\|_{H^p(\rn)}:\
F\in H^p(\rn)\ \text{and}\ F|_{\boz}=f\r\},$$
where the infimum is taken over all $F\in H^p(\rn)$ satisfying
$F=f$ on $\boz$.
\vspace{-0.25cm}
\item[(iv)] Let $\phi\in C^\fz_{\rm c}(\rn)$ be a nonnegative function such that
$$\supp(\phi):=\left\{x\in\rn:\ \phi(x)\neq0\right\}\subset B(\mathbf{0},1)$$
and $\int_{\rn}\phi(x)\,dx=1$, here and thereafter, $\mathbf{0}$ denotes the \emph{origin} of $\rn$.
For any $f\in\cd'(\Omega)$, the \emph{radial maximal function} $f^+_{\Omega}$
is defined by setting, for any $x\in\boz$,
\begin{equation*}
f^+_{\Omega}(x):=\sup_{t\in(0,\dz(x)/2)}\lf|\phi_t\ast f(x)\r|,
\end{equation*}
where, for any $x\in\boz$, $\dz(x):=\dist(x,\boz^\complement)$ and, for any $t\in(0,\fz)$,
$\phi_t(\cdot):=\frac{1}{t^n}\phi(\frac{\cdot}{t})$.
Then the \emph{Hardy space} $H^{p}(\Omega)$ is defined by setting
$$H^{p}(\Omega):=\lf\{f\in\cd'(\Omega):\ \|f\|_{H^{p}(\Omega)}:=
\lf\|f^+_{\Omega}\r\|_{L^{p}(\Omega)}<\infty\r\}.$$
\end{itemize}
\end{definition}

Let $p\in(0,1]$. From the definitions of both $H^p_r(\boz)$ and $H^p_z(\boz)$, it follows that
$H^p_z(\boz)\subset H^p_r(\boz)$. Moreover, by \cite{cks93}, we find that $H^p_z(\boz)\subsetneqq H^p_r(\boz)$.

\begin{definition}\label{d1.2}
Let $n\ge2$, $\boz$ be a bounded NTA domain of $\rn$,
$p\in(0,1]$, and $L_D$ be as in \eqref{eq1.4}. For any $f\in L^2(\boz)$,
the \emph{Lusin area function, $S_{L_D}(f)$, associated with $L_D$,} is defined by setting,
for any $x\in\boz$,
\begin{equation*}
S_{L_D}(f)(x):=\lf[\int_{\bgz(x)}\lf|t^2 L_De^{-t^2L_D}(f)(y)\r|^2
\frac{dy\,dt}{|B_\boz(x,t)|t}\r]^{1/2},
\end{equation*}
where
$$\bgz(x):=\{(y,t)\in\boz\times(0,\fz):\ |x-y|<t\}$$
and $B_\boz(x,t):=B(x,t)\cap\boz$.

A function $f\in L^2(\boz)$ is said to be in the \emph{set} $\mathbb{H}^p_{L_D}(\boz)$
if $S_{L_D}(f)\in L^p(\boz)$; moreover, for any $f\in\mathbb{H}^p_{L_D}(\boz)$, define
$\|f\|_{H^p_{L_D}(\boz)}:=\|S_{L_D}(f)\|_{L^p(\boz)}$.
Then the \emph{Hardy space} $H^p_{L_D}(\boz)$ is defined to be the completion
of $\mathbb{H}^p_{L_D}(\boz)$ with respect to the \emph{quasi-norm} $\|\cdot\|_{H^p_{L_D}(\boz)}$.
\end{definition}

It is well known that the real-variable theory of Hardy spaces on $\rn$,
initiated by Stein and Weiss \cite{sw60} and then systematically developed
by Fefferman and Stein \cite{fs72}, plays important roles in various fields of analysis
and partial differential equations. In recent years, the study on the real-variable
theory of Hardy spaces on $\rn$ or its domains, associated with different
differential operators, has aroused great interests (see, for instance, \cite{bdl18,dy05,hlmmy,hmm11,sy16}
for Hardy spaces on $\rn$, and \cite{ar03,bd18,bdl18a,ccyy13,dhmmy13,lw17,yy13,yy12,yy18}
for Hardy spaces on domains). Moreover, the Hardy space $H^p(\boz)$ on the domain $\boz$ of $\rn$
was introduced and studied by Miyachi \cite{m90}. Furthermore, when $\boz$ is a Lipschitz domain of $\rn$,
the `geometrical' Hardy spaces $H^p_{r}(\boz)$ and $H^p_{z}(\boz)$ on domains were introduced
by Chang et al. \cite{cks93,cks92} which naturally appear in the study of the regularity of the Dirichlet
and the Neumann boundary value problems of second-order elliptic equations (see,
for instance, \cite{ar03,cds99,cks93,dhmmy13}).

Then we have the following equivalence relation between $H^p_{r}(\boz)$, $H^p(\boz)$,
and $H^p_{L_D}(\boz)$.

\begin{theorem}\label{t1.2}
Let $n\ge2$, $\boz\subset\rn$ be a bounded {\rm NTA} domain, and $\dz_0\in(0,1]$ as in Theorem \ref{t1.1}.
Then, for any given $p\in(\frac{n}{n+1},1]$, $H^p_{r}(\boz)=H^p(\boz)$ with equivalent quasi-norms.
Moreover, for any given $p\in(\frac{n}{n+\dz_0},1]$, $H^p_{L_D}(\boz)=H^p_{r}(\boz)
=H^p(\boz)$ with equivalent quasi-norms.
\end{theorem}

\begin{remark}\label{r1.2}
Let $\dz_0\in(0,1]$ be as in Theorem \ref{t1.1}. When $\boz$ is a bounded Lipschitz domain of $\rn$,
the equivalence of the Hardy spaces $H^1_{L_D}(\boz)$ and $H^1_{r}(\boz)$ in Theorem \ref{t1.2} was
obtained by Auscher and Russ \cite[Theorem 1 and Proposition 5]{ar03}.
Moreover, when $\boz$ is a bounded Lipschitz domain of $\rn$, the operator $L_D$ is non-negative
and adjoint, and $p\in(\frac{n}{n+\dz_0},1]$, the equivalence of the spaces $H^p_{L_D}(\boz)$,
$H^p_{r}(\boz)$, and $H^p(\boz)$ in Theorem \ref{t1.2} was established by Bui et al.
\cite[Theorem 4.4 and Remark 4.5(c)]{bdl18a} (see also \cite[Corollary 4.5]{yy18}).
Thus, Theorem \ref{t1.2} improves the results of Auscher and Russ \cite[Theorem 1 and Proposition 5]{ar03}
and Bui et al. \cite[Theorem 4.4 and Remark 4.5(c)]{bdl18a} by weakening the assumptions
on both the domain $\boz$ and the operator $L_D$.

Furthermore, when $\boz$ is a bounded Lipschitz domain, the equivalence of $H^p_{r}(\boz)$ and
$H^p(\boz)$ for any given $p\in(\frac{n}{n+1},1]$ was essentially proved by Chang et al.
\cite[Theorem 2.7]{cks93}. Recall that NTA domains contain Lipschitz domains as special examples
(see, for instance, \cite{jk82,t97} or Remark \ref{r2.2}(iii) below).
Therefore, the equivalence of $H^p_{r}(\boz)$ and $H^p(\boz)$ for any given $p\in(\frac{n}{n+1},1]$
obtained in Theorem \ref{t1.2} improves \cite[Theorem 2.7]{cks93} via weakening the assumption on
the domain $\boz$ under consideration.
\end{remark}

By subtly using some geometrical properties of NTA domains, obtained in Lemma \ref{l2.1} below,
the reflection technology related to NTA domains, the atomic characterizations of both $H^p(\rn)$
and $H^p(\boz)$, and the H\"older continuity of the heat kernels $\{K^{L_D}_t\}_{t>0}$
given in Theorem \ref{t1.1}, we show Theorem \ref{t1.2}.

Let $n\ge2$, $p\in(1,\fz)$, and $\boz\subset\rn$ be a bounded NTA domain. Assume that
\begin{align}\label{eq1.7}
p_\ast:=
\begin{cases}
\dfrac{np}{n+p}\ \ &\text{when}\,p\in\lf(\dfrac{n}{n-1},\fz\r),\\
1+\epz\ \ &\text{when}\,p\in\lf(1,\dfrac{n}{n-1}\r],
\end{cases}
\end{align}
where $\epz\in(0,\fz)$ is an arbitrary given constant.
Let $p\in(1,\fz)$, $f\in L^{p_\ast}(\boz)$, and the real-valued, bounded, and measurable matrix $A$
satisfy \eqref{eq1.3}. Then a function $u$ is called a \emph{weak solution} of the
\emph{Dirichlet boundary value problem} \eqref{eq1.1}
if $u\in W^{1,p}_{0}(\boz)$ and, for any $\varphi\in C^\fz_\mathrm{c}(\boz)$,
\begin{equation}\label{eq1.8}
\int_{\boz}A(x)\nabla u(x)\cdot\nabla\varphi(x)\,dx=\int_\boz f(x)\varphi(x)\,dx.
\end{equation}
Moreover, the Dirichlet problem \eqref{eq1.1} is said to be \emph{uniquely solvable}
if, for any $f\in L^{p_\ast}(\boz)$, there exists a unique $u\in W^{1,p}_0(\boz)$
such that \eqref{eq1.8} holds true for any $\varphi\in C^\fz_{\rm c}(\boz)$.

\begin{remark}\label{r1.3}
Assume that $f\in L^{2_\ast}(\boz)$, where $2_\ast$ is as in \eqref{eq1.7} with $p$ replaced by $2$.
Then, by the Sobolev inequality (see, for instance, \cite[Theorem 1.1]{bk95}) and the Lax--Milgram theorem
(see, for instance, \cite[Theorem 5.8]{gt01}), we find that the Dirichlet problem \eqref{eq1.1} is uniquely solvable
and the weak solution $u$ satisfies
\begin{align}\label{eq1.9}
\|\nabla u\|_{L^2(\boz;\rn)}\le C\|f\|_{L^{2_\ast}(\boz)},
\end{align}
where $C$ is a positive constant independent of $u$ and $f$.
\end{remark}

For the Dirichlet problem \eqref{eq1.1} on bounded NTA domains, we have the following global
regularity estimates in both Lebesgue spaces $L^q(\boz)$ with $q\in(1,p_0)$, and Hardy spaces $H^q_z(\boz)$
with $q\in(\frac{n}{n+1},1]$, where $p_0\in(2,\fz)$ is a constant depending only on $n$, the domain $\boz$,
and the matrix $A$.

\begin{theorem}\label{t1.3}
Let $n\ge2$, $\boz$ be a bounded {\rm NTA} domain of $\rn$, and the real-valued, bounded, and
measurable matrix $A$ satisfy \eqref{eq1.3}.
\begin{itemize}
\item[{\rm(i)}] Then there exists a $p_0\in(2,\fz)$, depending only on $n$, $A$, and $\boz$, such that,
for any given $q\in(\frac{n}{n-1},p_0)$ and $p\in(1,n)$ satisfying $\frac{1}{p}-\frac{1}{q}=\frac{1}{n}$,
and any $f\in L^p(\boz)\cap L^2(\boz)$, the weak solution $u$ of the problem
\eqref{eq1.1} uniquely exists and satisfies $\nabla u\in L^q(\boz;\rn)$ and
$$\|\nabla u\|_{L^q(\boz;\rn)}\le C\|f\|_{L^p(\boz)},$$
where $C$ is a positive constant independent of $u$ and $f$.
\vspace{-0.25cm}
\item[{\rm(ii)}] Let $q\in(1,\frac{n}{n-1}]$ and $p\in(\frac{n}{n+1},1]$ satisfy
$\frac{1}{p}-\frac{1}{q}=\frac{1}{n}$, and $f\in H^p_{L_D}(\boz)\cap L^2(\boz)$.
Then the weak solution $u$ of the problem \eqref{eq1.1} uniquely exists and
satisfies $\nabla u\in L^q(\boz;\rn)$ and
$$\|\nabla u\|_{L^q(\boz;\rn)}\le C\|f\|_{H^p_{L_D}(\boz)},$$
where $C$ is a positive constant independent of $u$ and $f$.
\vspace{-0.25cm}
\item[{\rm(iii)}] Let $q\in(\frac{n}{n+1},1]$ and $p\in(\frac{n}{n+2},\frac{n}{n+1}]$
satisfy $\frac{1}{p}-\frac{1}{q}=\frac{1}{n}$, and $f\in H^p_{L_D}(\boz)\cap L^2(\boz)$.
Then the weak solution $u$ of the problem \eqref{eq1.1} uniquely exists and satisfies
$\nabla u\in H^q_{z}(\boz;\rn)$ and
$$\|\nabla u\|_{H^q_{z}(\boz;\rn)}\le C\|f\|_{H^p_{L_D}(\boz)},$$
where $C$ is a positive constant independent of $u$ and $f$. Here and thereafter,
the space $H^q_{z}(\boz;\rn)$ is defined as $L^p(\boz;\rn)$ via $L^p(\boz)$ replaced by $H^q_{z}(\boz)$
[see \eqref{eq1.2}].
\end{itemize}
\end{theorem}

We prove Theorem \ref{t1.3} by the following strategy. We first obtain the global
gradient estimate for the problem \eqref{eq1.1} in $L^q(\boz)$ with $q\in(2,p_0)$, by
using \eqref{eq1.9}, a reverse H\"older inequality for the gradient of the weak solution
of some local Dirichlet boundary value problems (see Lemma \ref{l3.8} below for the details),
and a real-variable argument for $L^q(\boz)$ estimates, essentially established by Shen
\cite[Theorem 3.4]{sh07} (see also \cite[Theorem 4.2.6]{sh18} and \cite[Theorem 3.3]{sh05a}).
Then we show (i) by the global gradient estimate in $L^q(\boz)$ with $q\in(2,p_0)$,
the conclusion of (ii), and the complex interpolation theory of Hardy spaces (see, for instance,
\cite[Theorem 8.1 and (9.3)]{kmm07}).  Moreover, we prove (ii) via establishing
some fine estimate for the kernels $\{K_t^{L_D}\}_{t>0}$ (see Lemma \ref{l5.1} below),
and using the global gradient estimate in $L^q(\boz)$ with $q\in(2,p_0)$, and the molecular
characterization of $H^p_{L_D}(\boz)$. Finally, we show (iii) via using the global gradient
estimate in $L^q(\boz)$ with $q\in(2,p_0)$, and the molecular characterization of
both $H^p_{L_D}(\boz)$ and $H^p(\rn)$.

As a corollary of Theorems \ref{t1.2} and \ref{t1.3}, we have the following conclusion.

\begin{corollary}\label{c1.1}
Let $n\ge2$, $\boz$ be a bounded {\rm NTA} domain of $\rn$, and $\dz_0\in(0,1]$
as in Theorem \ref{t1.1}. Assume that $q\in(1,\frac{n}{n-1}]$ and $p\in(\frac{n}{n+\dz_0},1]$
satisfy $\frac{1}{p}-\frac{1}{q}=\frac{1}{n}$. Let $f\in H^p_{r}(\boz)\cap L^2(\boz)$.
Then the weak solution $u$ of the problem \eqref{eq1.1} uniquely exists and satisfies
$\nabla u\in L^q(\boz;\rn)$ and
$$\|\nabla u\|_{L^q(\boz;\rn)}\le C\|f\|_{H^p_{r}(\boz)},$$
where $C$ is a positive constant independent of $u$ and $f$.
\end{corollary}

\begin{remark}\label{r1.4}
\begin{itemize}
\item[\rm(i)] By an example given in \cite[p.\,120]{at98} (which is essentially due to C. E. Kenig,
as was pointed out in \cite[p.\,119, Theorem 7]{at98}), we find that, even when $\boz\subset\rr^2$ is a
bounded smooth domain, there exists a real-valued, bounded, and measurable matrix $A$ satisfying
\eqref{eq1.3} such that, for any given $p\in(2,\fz)$, the weak solution $u$ of the problem \eqref{eq1.1}
with some $f\in L^2(\boz)$ satisfies $|\nabla u|\not\in L^p(\boz)$. Therefore,
the range $q\in(\frac{n}{n-1},p_0)$ of $q$ obtained in Theorem \ref{t1.3}(i) is sharp.
\vspace{-0.25cm}
\item[\rm(ii)] When $A:=I$ (the identity matrix) and $\boz$ is a bounded Lipschitz domain
of $\rn$, Dahlberg \cite[Theorem 1]{d79} obtained Theorem \ref{t1.3}(i) with $p_0:=3+\epz_0$
when $n\ge3$, and with $p_0:=4+\epz_0$ when $n=2$, where $\epz_0\in(0,\fz)$ is a constant depending
only on $n$ and $\boz$. Meanwhile, Dahlberg \cite{d79} also showed that the range
$q\in(\frac{n}{n-1},p_0)$ of $q$ is sharp in this case. Thus, Theorem \ref{t1.3}(i) extends
the results of Dahlberg \cite[Theorem 1]{d79} via weakening the assumptions on both the
coefficient matrix $A$ and the domain $\boz$.

Moreover, to the best of our knowledge, the global gradient estimates obtained in
(ii) and (iii) of Theorem \ref{t1.3} and Corollary \ref{c1.1} are new even when $\boz\subset\rn$
is a bounded Lipschitz domain.
\end{itemize}
\end{remark}

The organization of the remainder of this article is as follows.

In Section \ref{s2}, we present the definitions of both NTA domains and quasi-convex domains,
some geometrical properties of NTA domains, the definition of the $(\gz,R)$-$\mathrm{BMO}$ condition,
and the atomic and the molecular characterizations of the Hardy spaces $H^p(\rn)$, $H^p(\boz)$,
and $H^p_{L_D}(\boz)$ associated with $L_D$.

Section \ref{s3} is devoted to the proof of Theorem \ref{t1.1}. The proofs of Theorems
\ref{t1.2} and \ref{t1.3} are presented, respectively, in Sections \ref{s4} and \ref{s5}.

Finally, we make some conventions on notation. Throughout the whole article, we always denote by $C$ or $c$ a
\emph{positive constant} which is independent of the main parameters, but it may vary from line to
line. We also use $C_{(\gz,\,\bz,\,\ldots)}$ to denote a  \emph{positive
constant} depending on the indicated parameters $\gz,$ $\bz$, $\ldots$. The \emph{symbol} $f\ls g$ means
that $f\le Cg$. If $f\ls g$ and $g\ls f$, then we write $f\sim g$. If $f\le Cg$ and $g=h$ or $g\le h$,
we then write $f\ls g\sim h$ or $f\ls g\ls h$, \emph{rather than} $f\ls g=h$ or $f\ls g\le h$.
For each ball $B:=B(x_B,r_B)$ in $\rn$, with $x_B\in\rn$ and $r_B\in (0,\fz)$, and $\az\in(0,\fz)$,
let $\az B:=B(x_B,\az r_B)$; furthermore, denote by $B_\boz$ the ball $B\cap\boz$ in $\boz$ with $B$ being
a ball of $\rn$. For any subset $E$ of $\rn$, we denote by $E^\complement$ the \emph{set} $\rn\setminus E$,
and by $\mathbf{1}_{E}$ its \emph{characteristic function}. We also let $\nn:=\{1,\, 2,\, \ldots\}$
and $\zz_+:=\nn\cup\{0\}$. The \emph{symbol} $\lfloor s \rfloor$ for any $s\in\rr$ denotes the largest
integer not greater than $s$. For any multi-index $\alpha:=(\alpha_1,\dots,\alpha_n)
\in\zz^n_+:=(\zz_+)^n$, let $|\alpha|:= \alpha_1+\cdots+\alpha_n.$ For any ball $B$ of $\rn$
or of $\boz$, let $S_j(B):=(2^{j+1}B)\setminus(2^{j}B)$ for any $j\in\nn$, and $S_0(B):=2B$.
Moreover, for any $q\in[1,\fz]$, we denote by $q'$ its \emph{conjugate exponent},
namely, $1/q + 1/q'=1$. Finally, for any measurable set $E\subset\rn$ with $|E|<\fz$, and
any $f\in L^1(E)$, we let
$$
\fint_Ef(x)\,dx:=\frac{1}{|E|}\int_{E}f(x)\,dx.
$$

\section{Preliminaries}\label{s2}

 In this section, we first recall the definitions of NTA domains, quasi-convex domains,
and the $(\gz,R)$-$\mathrm{BMO}$ condition, and then give some geometrical properties of NTA domains.
Moreover, we present the atomic and the molecular characterizations of the Hardy spaces $H^p(\rn)$, $H^p(\boz)$,
and $H^p_{L_D}(\boz)$ associated with $L_D$.

\subsection{NTA Domains}\label{s2.1}
 In this subsection, we first recall the definitions of NTA domains introduced by
Jerison and Kenig \cite{jk82} (see also \cite{kt97,t97}) and quasi-convex domains
introduced by Jia et al. \cite{jlw10}, and then state some geometrical properties of NTA domains.
We begin with recalling several notions.
For any given $x\in\rn$ and measurable subset $E\subset\rn$,
let $\dist(x,E):=\inf\{|x-y|:\ y\in E\}$. Meanwhile, for any measurable subsets $E,\ F\subset\rn$,
let $\dist(E,F):=\inf\{|x-y|:\ x\in E,\,y\in F\}$ and $\diam(E):=\sup\{|x-y|:\ x,\,y\in E\}$.

\begin{definition}\label{d2.1}
Let $n\ge2$, $\boz\subset\rn$ be a \emph{domain}
which means that $\boz$ is a connected open set, and $\boz^\complement:=\rn\backslash\boz$.
\begin{itemize}
\item[{\rm(i)}] The domain $\boz$ is said to satisfy the \emph{interior}
[resp., \emph{exterior}] \emph{corkscrew condition}
if there exist constants $R\in(0,\fz)$ and $\sz\in(0,1)$ such that,
for any $x\in\partial\boz$ and $r\in(0,R)$, there exists a point $x_0\in\boz$
[resp., $x_0\in\boz^\complement$], depending on $x$,
such that $B(x_0,\sz r)\subset\boz\cap B(x,r)$ [resp., $B(x_0,\sz r)\subset\boz^\complement\cap B(x,r)$].
\vspace{-0.25cm}
\item[{\rm(ii)}] The domain $\boz$ is said to satisfy the \emph{Harnack chain condition} if there exist
a constant $m_1\in(1,\fz)$ and a constant $m_2\in(0,\fz)$ such that, for any $x_1,\,x_2\in\boz$ satisfying
$$M:=\frac{|x_1-x_2|}{\min\{\dist(x_1,\partial\boz),\dist(x_2,\partial\boz)\}}>1,
$$
there exists a chain $\{B_i\}_{i=1}^N$ of open Harnack balls, with
$B_i\subset\boz$ for any $i\in\{1,\,\ldots,\,N\}$, that connects $x_1$ to $x_2$; namely,
$x_1\in B_1$, $x_2\in B_N$, $B_i\cap B_{i+1}\neq\emptyset$ for any $i\in\{1,\,\ldots,\,N-1\}$, and, for any
$i\in\{1,\,\ldots,\,N\}$,
$$m_1^{-1}\diam(B_i)\le\dist(B_i,\partial\boz)\le m_1\diam(B_i),
$$
where the integer $N$ satisfies $N\le m_2\log_2 M$.
\vspace{-0.25cm}
\item[{\rm(iii)}] The domain $\boz$ is called a \emph{non-tangentially accessible domain} (for short,
NTA \emph{domain}) if $\boz$ satisfies the interior
and the exterior corkscrew conditions, and the Harnack chain condition.
\end{itemize}
\end{definition}

We point out that NTA domains include Lipschitz domains, $\mathrm{BMO}_1$ domains,
Zygmund domains, quasi-spheres, and some Reifenberg flat domains as special examples
(see, for instance, \cite{jk82,kt97,t97}). Moreover, it is well known that NTA domains
are $W^{1,p}$-extension domains with $p\in[1,\fz)$ (see, for instance, \cite{hkt08,j81}).
Meanwhile, NTA domains have the following geometrical properties.

\begin{lemma}\label{l2.1}
Let $n\ge2$ and $\boz$ be a bounded {\rm NTA} domain of $\rn$.
\begin{itemize}
\item[\rm(i)] Then there exists a constant $C\in(0,1]$, depending only on $n$ and $\boz$,
such that, for any ball $B\subset\rn$ with the center $x_B\in\overline{\boz}$ and
the radius $r_B\in(0,\diam(\boz))$,
$$|B\cap\boz|\ge C|B|.$$
\vspace{-0.85cm}
\item[\rm(ii)] Let $B:=B(x_B,r_B)\subset\rn$ be a ball satisfying that $2B\subset\boz$ and $4B\cap\boz^\complement\neq\emptyset$.
Then there exists a ball $\wz{B}\subset\boz^\complement$ such that $r_{\wz{B}}\sim r_B$ and $\dist(B,\wz{B})\sim r_B$,
where the positive equivalence constants are independent of both $B$ and $\wz{B}$.
\vspace{-0.25cm}
\item[\rm(iii)] For any ball $B\subset\rn$ with the center $x_B\in\partial\boz$ and
the radius $r_B\in(0,\diam(\boz))$, there exists a ball $\wz{B}\subset\boz^\complement\cap B$
such that $r_{\wz{B}}\sim r_B$, where the positive equivalence constants are independent of both $B$ and $\wz{B}$.
In particular, there exists a constant $C\in(0,1]$, depending only on $n$ and $\boz$,
such that, for any ball $B\subset\rn$ with the center $x_B\in\partial\boz$ and
the radius $r_B\in(0,\diam(\boz))$,
$$|B\cap\boz^\complement|\ge C|B|.$$
\end{itemize}
\end{lemma}

\begin{proof}
We first show (i). By the fact that $\boz$ is a $W^{1,p}$-extension domain with $p\in[1,\fz)$,
and \cite[Theorem 2]{hkt08}, we conclude that there exist constants $R_0\in(0,\infty)$ and
$C_1\in(0,1]$ such that, for any given ball $B:=B(x_B,r_B)$ of $\rn$ with $x_B\in\overline{\boz}$
and $r_B\in(0,R_0]$,
\begin{equation}\label{eq2.1}
|B\cap\boz|\ge C_1|B|.
\end{equation}

If $\diam(\boz)\le R_0$, then, from \eqref{eq2.1}, it follows that (i) holds true in this case.
If $\diam(\boz)>R_0$, then, for any ball $B:=B(x_B,r_B)$ of $\rn$ with $x_B\in\overline{\boz}$
and $r_B\in(R_0,\diam(\boz))$,
\begin{align*}
|B\cap\boz|&\ge|B(x_B,R_0)\cap\boz|\ge C_1|B(x_B,R_0)|\\
&=C_1\lf[\frac{R_0}{\diam(\boz)}\r]^n|B(x_B,\diam(\boz))|
\ge C_1\lf[\frac{R_0}{\diam(\boz)}\r]^n|B|,
\end{align*}
which, together with \eqref{eq2.1}, further implies that (i) holds true in this case.
This finishes the proof of (i).

Now, we prove (ii). Let $y_B\in\partial\boz$ be such that $|x_B-y_B|=\dist(x_B,\partial\boz)$.
Define $\ell:=|x_B-y_B|$. Then $\ell\in(2r_{B},4r_{B})$.
By the exterior corkscrew condition of $\boz$, we find that there exists a ball
$\wz{B}:=B(x_0,\sigma r)\subset[\boz^\complement\cap B(y_B,r)]$, where
$r\in(\frac{1}{2}\min\{R,\ell\},\min\{R,\ell\})$, and $\sz$ and $R$ are as in Definition \ref{d2.1}.
From this, we deduce that
$$
\dfrac{r_{B}}{r_{\wz{B}}}\le\max\lf\{\frac{\diam(\boz)}{2\sigma R}, \frac{1}{\sigma}\r\}
$$
and
$$
\dfrac{r_{B}}{r_{\wz{B}}}\ge\frac{1}{4\sigma},
$$
which implies that $r_{B}\sim r_{\wz{B}}$. Moreover, it is easy to see that $\dist(B,\wz{B})\ge r_B$ and
$$
\dist\lf(B,\wz{B}\r)\le |x_B-y_B|+r=\ell+\frac{1}{\sz}r_{\wz{B}}\ls r_B.
$$
Therefore, $\dist(B,\wz{B})\sim r_B$. This finishes the proof of (ii).

Finally, we show (iii). Let $B:=B(x_B,r_B)$ be a ball of $\rn$ with the center $x_B\in\partial\boz$ and
the radius $r_B\in(0,\diam(\boz))$. Then, when $r_B\in(0,\min\{R,\diam(\boz)\})$,
where $R\in(0,\fz)$ is as in Definition \ref{d2.1}, from the exterior corkscrew
condition of $\boz$, we deduce that there exists a ball $\wz{B}:=B(x_0,\sigma r_B)
\subset[\boz^\complement\cap B(x_B,r_B)]$, which further implies that $r_{\wz{B}}\sim r_B$ and
$$
\lf|B(x_B,r_B)\cap\boz^\complement\r|\ge|B(x_0,\sigma r_B)|\gs|B(x_B,r_B)|.
$$
Thus, (iii) holds true in the case that $r_B\in(0,\min\{R,\diam(\boz)\})$.

Moreover, when $r_B\in(\min\{R,\diam(\boz)\},\diam(\boz))$, let $\ell:=\min\{R,\diam(\boz)\}/2$.
By the exterior corkscrew condition of $\boz$ again, we conclude that there exists
a ball $\wz{B}:=B(x_0,\sigma \ell)\subset[\boz^\complement\cap B(x_B,\ell)]$, which
implies that $r_{\wz{B}}\sim r_B$ and
\begin{align*}
\lf|B(x_B,r_B)\cap\boz^\complement\r|&\ge\lf|B\lf(x_B,\ell\r)\cap\boz^\complement\r|
\gs\lf|B\lf(x_B,\ell\r)\r|\\
&\sim|B\lf(x_B,\diam(\boz)\r)|\gs|B(x_B,r_B)|.
\end{align*}
From this, it follows that (iii) holds true in the case that $r_B\in(\min\{R,\diam(\boz)\},
\diam(\boz))$. This finishes the proof of (iii) and hence of Lemma \ref{l2.1}.
\end{proof}

\begin{remark}\label{r2.1}
Let $n\ge2$ and $\boz$ be a bounded NTA domain of $\rn$.
By Lemma \ref{l2.1}(i), we find that $(\boz, |\cdot|, dx)$ is a space of homogeneous type in the
sense of Coifman and Weiss \cite{cw71}, where $|\cdot|$ denotes the usual norm on $\rn$, and $dx$ the Lebesgue
measure on $\rn$. Moreover, as a space of homogeneous type, the collection of
all balls of $\boz$ is given by the set
$$\lf\{B_\boz:=B\cap\boz:\  \text{any ball}\ B\subset\rn \ \text{satisfying}\  x_B\in\overline{\boz}\
\text{and}\ r_B\in(0,\diam(\boz))\r\},$$
where $x_B$ denotes the center of $B$ and $r_B$ the radius of $B$.
\end{remark}

Now, we recall the definition of the $(\gz,\sigma,R)$ quasi-convex domain as follows.

\begin{definition}\label{d2.2}
Let $n\ge2$, $\boz\subset\rn$ be a domain, $\gz$, $\sigma\in(0,1)$, and $R\in(0,\fz)$.
Then $\boz$ is called a $(\gz,\sigma,R)$ \emph{quasi\mbox{-}convex domain}
if, for any $x\in \partial\boz$ and $r\in (0,R]$,
\begin{itemize}
\item[\rm(a)] there exists an $x_0\in\boz$, depending on $x$, such that $B(x_0,
\sigma r)\subset[\boz\cap B(x,r)];$
\vspace{-0.25cm}
\item[\rm(b)] there exists a convex domain $V:=V(x,r)$, depending on $x$ and $r$,
such that $[B(x,r)\cap\boz]\subset V$ and
$$d_H(\partial(B(x,r)\cap\boz), \partial V)\le \gz r,
$$
where $d_H(\cdot,\cdot)$ denotes the \emph{Hausdorff distance} which is defined
by setting, for any non\mbox{-}empty measurable subsets $E_1$ and $E_2$ of $\rn$,
$$d_H(E_1,E_2):=\max\lf\{\sup_{x\in E_1}\inf_{y\in E_2}|x-y|,\,
\sup_{x\in E_2}\inf_{y\in E_1}|x-y|\r\}.
$$
\end{itemize}
\end{definition}

The concept of quasi-convex domains was introduced by Jia et al. \cite{jlw10}
to study the global regularity of second-order elliptic equations.
Roughly speaking, a quasi-convex domain is a domain satisfying that the local boundary
is close to be convex at small scales. It is easy to find that, if $\boz$ is a convex domain,
then $\boz$ is a $(\gz,\sz,R)$ quasi-convex domain
for any $\gz\in(0,1)$, some $\sz\in(0,1)$, and some $R\in(0,\fz)$.

\begin{remark}\label{r2.2}
\begin{itemize}
\item [{\rm (i)}] Let $n\ge2$, $\gz\in(0,1)$, and $R\in(0,\fz)$. A domain $\boz\subset\rn$ is called a
$(\gz,R)$-\emph{Reifenberg flat domain} if, for any $x_0\in\partial\boz$ and $r\in(0,R_0]$,
there exists a system $\{y_1,\,\ldots,\,y_n\}$ of coordinates, which may depend on $x_0$ and $r$,
such that, in this coordinate system, $x_0=\mathbf{0}$ and
\begin{equation*}
\lf[B(\mathbf{0},r)\cap\{y\in\rn:\ y_n>\gz r\}\r]\subset \lf[B(\mathbf{0},r)\cap\boz\r]\subset
\lf[B(\mathbf{0},r)\cap\{y\in\rn:\ y_n>-\gz r\}\r].
\end{equation*}

The Reifenberg flat domain was introduced by Reifenberg \cite{r60},
which naturally appears in the theory of both minimal surfaces and free boundary problems.
A typical example of Reifenberg flat domains is the well-known Van Koch snowflake
(see, for instance, \cite{t97}). Moreover, it was shown by Kenig and Toro \cite[Theorem 3.1]{kt97}
that there exists a $\gz_0(n)\in(0,1)$, depending only on $n$, such that, if $\boz$ is a
$(\gz,R)$-Reifenberg flat domain for some $\gz\in(0,\gz_0(n)]$ and $R\in(0,\fz)$, then $\boz$ is an
NTA domain. Meanwhile, it is easy to see that a $(\gz,R)$-Reifenberg flat domain is also a
$(\gz,\sz,R)$ quasi-convex domain for some $\sz\in(0,1)$. In recent years, boundary value problems
of elliptic or parabolic equations on Reifenberg flat domains have been widely concerned and studied
(see, for instance, \cite{bd17,bw04,dk12,dk18,dl21}).
\vspace{-0.25cm}
\item [{\rm (ii)}] It is known that, for any open set $\boz\subset\rn$ with compact boundary,
$\boz$ is a semi-convex domain of $\rn$ if and only if $\boz$ is a Lipschitz domain satisfying
a uniform exterior ball condition (see, for instance, \cite{mmmy10,mmy10,ycyy20} for the
definitions of both the semi-convex domain and the uniform exterior ball condition). It is worth pointing out that
convex domains of $\rn$ are semi-convex domains and (semi-)convex domains are special cases
of Lipschitz domains (see, for instance, \cite{mmmy10,mmy10,ycyy20}). Furthermore, a (semi-)convex
domain is also a $(\gz,\sz,R)$ quasi-convex domain for any $\gz\in(0,1)$, some $\sz\in(0,1)$,
and some $R\in(0,\fz)$ (see, for instance, \cite{yyy20}).
\vspace{-0.25cm}
\item[{\rm(iii)}] On NTA domains, Lipschitz domains, quasi-convex domains, Reifenberg flat domains,
$C^1$ domains, and (semi-)convex domains, we have the following relations (see, for instance,
\cite{jk82,jlw10,kt97,t97,ycyy20,yyy20}).
\begin{enumerate}
\item[\rm(a)]$\text{class of $C^1$ domains}\\
~~~~~~~~~\subsetneqq\ \text{class of Lipschitz domains
with small Lipschitz constants}\\
~~~~~~~~~\subsetneqq\ \text{class of Lipschitz domains}\\
~~~~~~~~~\subsetneqq\ \text{class of NTA domains};$
\vspace{-0.15cm}
\item[\rm(b)]$\text{class of $C^1$ domains}\\
~~~~~~~~~\subsetneqq\ \text{class of Lipschitz domains with small Lipschitz constants}\\
~~~~~~~~~\subsetneqq\ \text{class of Reifenberg flat domains}\\
~~~~~~~~~\subsetneqq\ \text{class of quasi-convex domains};$
\vspace{-0.15cm}
\item[\rm(c)]$\text{class of (semi-)convex domains}\ \subsetneqq\ \text{class of quasi-convex domains}.$
\end{enumerate}
\end{itemize}
\end{remark}

Next, we recall the definition of the $(\gz,R)$-BMO condition (see, for instance, \cite{bw04}).
For any given domain $\boz\subset\rn$, denote by $L^1_{\loc}(\boz)$ the \emph{set of all
locally integrable functions on $\boz$}.

\begin{definition}\label{d2.4}
Let $n\ge2$, $\boz\subset\rn$ be a domain, and $\gz,\,R\in(0,\fz)$.
A function $f\in L^1_{\loc}(\boz)$ is said to satisfy
the \emph{$(\gz,R)$-$\mathrm{BMO}$ condition} if
\begin{equation*}
\|f\|_{\ast,\,R}:=\sup_{B(x,r)\subset\boz,\,r\in(0,R)}\fint_{B(x,r)}|f(y)-f_{B(x,r)}|\,dy\le\gz,
\end{equation*}
where the supremum is taken over all balls $B(x,r)\subset\boz$ with $r\in(0,R)$,
and $f_{B(x,r)}:=\fint_{B(x,r)}f(z)\,dz$.

Moreover, a matrix $A:=\{a_{ij}\}_{i,j=1}^n$ is said to satisfy the \emph{$(\gz,R)$-$\mathrm{BMO}$ condition}
if, for any $i,\,j\in\{1,\,\ldots,\,n\}$, $a_{ij}$ satisfies the $(\gz,R)$-$\mathrm{BMO}$ condition.
\end{definition}

\subsection{Hardy Spaces}\label{s2.2}

 In this subsection, we recall the atomic and the molecular characterizations of
the Hardy spaces $H^p(\rn)$, $H^p(\boz)$, and $H^p_{L_D}(\boz)$ associated with $L_D$.

\begin{definition}\label{d2.5}
Let $p\in(0,1]$, $q\in(1,\infty]$, and $s\in\zz_+$ with $s\ge \lfz n(\frac{1}{p}-1)\rfz$.
\begin{itemize}
\item[(i)] A function $a\in L^q(\rn)$ is called a \emph{$(p,\,q,\,s)$-atom}
if there exists a ball $B$ of $\rn$ such that
\vspace{-0.25cm}
\begin{enumerate}
\item[$\mathrm{(i)_1}$] $\supp (a)\subset B$;
\item[$\mathrm{(i)_2}$] $\|a\|_{L^q(\rn)}\le|B|^{1/q-1/p}$;
\item[$\mathrm{(i)_3}$] $\int_{\rn}a(x)x^{\az}\,dx=0$ for any
$\az\in\zz_+^n$ with $|\az|\le s$.
\end{enumerate}
\vspace{-0.25cm}
\item[(ii)] The \emph{atomic Hardy space} $H^{p,\,q,\,s}_{\mathrm{at}}(\rn)$ is defined to be the set of all
$f\in\cs'(\rn)$ satisfying that $f=\sum_{j=1}^\infty\lz_ja_j$ in $\cs'(\rn)$, where $\{a_j\}_{j=1}^\infty$
is a sequence of $(p,\,q,\,s)$-atoms, and $\{\lz_j\}^\infty_{j=1}\subset\mathbb{C}$ satisfies
$\sum_{j=1}^\fz|\lz_j|^p<\infty.$ Moreover, for any $f\in H^{p,\,q,\,s}_{\mathrm{at}}(\rn)$,
the \emph{quasi-norm} $\|f\|_{H^{p,\,q,\,s}_{\mathrm{at}}(\rn)}$ of $f$ is defined by setting
$$\|f\|_{H^{p,\,q,\,s}_{\mathrm{at}}(\rn)}:=\inf\lf\{\lf(\sum_{j=1}^\fz|\lz_j|^p\r)^{1/p}\r\},$$
where the infimum is taken over all decompositions of $f$ as above.
\end{itemize}
\end{definition}

\begin{definition}\label{d2.6}
Let $p\in(0,1]$, $q\in(1,\infty]$, $s\in\zz_+$ with $s\ge \lfz n(\frac{1}{p}-1)\rfz$, and
$\varepsilon\in(0,\fz)$.
\begin{itemize}
\item[(i)] A function $\az\in L^q(\rn)$ is called a \emph{$(p,\,q,\,s,\,\uc)$-molecule}
associated with the ball $B$ of $\rn$ if
\vspace{-0.25cm}
\begin{enumerate}
\item[$\mathrm{(i)_1}$]  for any $j\in\zz_+$, $\|\az\|_{L^q(S_j(B))}\le
2^{-j\uc}|2^j B|^{1/q}|B|^{-1/p}$.
\item[$\mathrm{(i)_2}$] $\int_{\rn}\az(x)x^{\bz}\,dx=0$ for any $\bz\in\zz_+^n$ with $|\bz|\le s$.
\end{enumerate}
\vspace{-0.25cm}
\item[(ii)] The \emph{molecular Hardy space} $H^{p,\,q,\,s,\,\uc}_{\mathrm{mol}}(\rn)$ is defined to be
the set of all $f\in\cs'(\rn)$ satisfying that $f=\sum_{j=1}^\infty\lz_j\az_j$ in $\cs'(\rn)$,
where $\{\az_j\}_{j=1}^\infty$ is a sequence of $(p,\,q,\,s,\,\uc)$-molecules, and $\{\lz_j\}_{j=1}^\infty\subset\cc$
satisfies $\sum_{j=1}^\fz|\lz_j|^p<\infty.$ Moreover, define
\begin{eqnarray*}
\|f\|_{H^{p,\,q,\,s,\,\uc}_{\mathrm{mol}}(\rn)}:=\inf\lf\{\lf(\sum_{j=1}^\fz|\lz_j|^p\r)^{1/p}\r\},
\end{eqnarray*}
where the infimum is taken over all decompositions of $f$ as above.
\end{itemize}
\end{definition}

Then we have the following atomic and molecular characterizations of the Hardy space
$H^p(\rn)$, respectively (see, for instance, \cite{l78,st93,tw80}).

\begin{lemma}\label{l2.2}
Let $p\in(0,1]$, $q\in(1,\infty]$, and $s\in\zz_+$ with $s\ge \lfz n(\frac{1}{p}-1)\rfz$,
and $\varepsilon\in(\max\{n+s,n/p\},\fz)$. Then the spaces $H^p(\rn)=H^{p,\,q,\,s}_{\mathrm{at}}(\rn)=
H^{p,\,q,\,s,\,\uc}_{\mathrm{mol}}(\rn)$ with equivalent quasi-norms.
\end{lemma}

\begin{definition}\label{d2.7}
Let $n\ge2$, $\Omega$ be a domain of $\rn$, $p\in(0,1]$, $q\in(1,\infty]$, and $s\in\zz_+$ with $s\ge \lfz n(\frac{1}{p}-1)\rfz$.
\begin{itemize}
\item[{\rm(i)}] A ball $B\subset\boz$ is called a \emph{type $(a)$ ball} of $\boz$ if
$4B\subset\Omega$, and a ball $B\subset\boz$ is called a \emph{type $(b)$ ball} of $\boz$
if $2B\cap\Omega^\complement=\emptyset$ but $4B\cap\Omega^\complement\neq\emptyset$.
\vspace{-0.25cm}
\item[(ii)] A function $a\in L^q(\boz)$ is called a \emph{type $(a)$}
\emph{$(p,\,q,\,s)_\boz$-atom} if there exists a type $(a)$ ball $B\subset\boz$
such that $\supp(a)\subset B$ and $a$ is a $(p,\,q,\,s)$-atom.

Moreover, a function $b\in L^q(\boz)$ is called a \emph{type $(b)$}
\emph{$(p,\,q)_\boz$-atom} if there exists a type $(b)$ ball $B\subset\boz$
such that $\supp(b)\subset B$ and $\|b\|_{L^q(\boz)}\le|B|^{1/q-1/p}$.
\vspace{-0.25cm}
\item[(iii)] The \emph{atomic Hardy space} $H^{p,\,q,\,s}_{\mathrm{at}}(\Omega)$
is defined to be the set of all distributions $f\in\cd'(\Omega)$
satisfying that there exist two sequences $\{\lz_j\}^\infty_{j=1}\subset \mathbb{C}$ and
$\{\kappa_j\}^\infty_{j=1}\subset \mathbb{C}$, a sequence $\{a_{j}\}_{j=1}^\fz$
of type $(a)$ $(p,\,q,\,s)_\boz$-atoms, and a sequence $\{b_{j}\}_{j=1}^\fz$
of type $(b)$ $(p,\,q)_\boz$-atoms such that
\begin{align}\label{eq2.2}
f=\sum_{j=1}^\fz\lambda_{j}a_{j} +\sum_{j=1}^\fz\kappa_{j}b_{j}
\end{align}
in $\cd'(\Omega)$, and
$$\sum_{j=1}^\fz|\lz_j|^p+\sum_{j=1}^\fz|\kappa_j|^p<\infty.$$
Furthermore, for any given $f\in H^{p,\,q,\,s}_{\mathrm{at}}(\Omega)$, let
$$\|f\|_{H^{p,\,q,\,s}_{\mathrm{at}}(\Omega)}:=\inf\lf\{\lf(\sum_{j=1}^\fz|\lz_j|^p\r)^{1/p}
+\lf(\sum_{j=1}^\fz|\kappa_j|^p\r)^{1/p}\r\},$$
where the infimum is taken over all decompositions of $f$ as in \eqref{eq2.2}.
\end{itemize}
\end{definition}

Then we have the following atomic characterization of the Hardy space $H^p(\boz)$
on the domain $\boz$ (see, for instance, \cite[Theorem 1]{m90}).

\begin{lemma}\label{l2.3}
Let $n\ge2$ and $\boz\subset\rn$ be a proper domain, $p\in(0,1]$, $q\in(1,\fz]$,
and $s\in\zz_+$ with $s\ge \lfz n(\frac{1}{p}-1)\rfz$. Then $H^{p}(\boz)=H^{p,\,q,\,s}_{\mathrm{at}}(\boz)$
with equivalent quasi-norms.
\end{lemma}

\begin{definition}\label{d2.8}
Let $n\ge2$, $\boz\subset\rn$ be a bounded NTA domain, $L_D$ as in \eqref{eq1.4}, and $p\in(0,1]$.
Assume that $q\in(1,\fz]$, $M\in\nn$, $\epz\in(0,\fz)$, $B:=B(x_B,r_B)$ with $x_B\in\boz$
and $r_B\in(0,\diam\,(\boz))$, and $B_\boz:=B\cap\boz$.
\begin{itemize}
\item[{\rm(i)}] A function $\az\in L^q(\boz)$ is called a \emph{$(p,\,q,\,M,\,\epz)_{L_D}$-molecule}
associated with the ball $B_\boz$ if, for any $k\in\{0,\,\ldots,\,M\}$ and $j\in\zz_+$, it holds true that
$$\lf\|\lf(r_B^{-2} L_D^{-1}\r)^k\az\r\|_{L^q(S_j(B_\boz))}\le
2^{-j\epz}|2^jB_\boz|^{1/q}|B_\boz|^{-1/p}.$$
\item[{\rm(ii)}] For any $f\in L^2(\boz)$,
\begin{align}\label{eq2.3}
f=\sum_{j=1}^\fz\lz_j\az_j
\end{align}
is called a \emph{molecular} $(p,\,q,\,M,\,\epz)_{L_D}$-\emph{representation} of the function $f$ if, for any $j\in\nn$,
$\az_j$ is a $(p,\,q,\,M,\,\epz)_{L_D}$-molecule associated with the ball $B_{\boz,\,j}\subset\boz$,
the summation \eqref{eq2.3} converges in $L^2(\boz)$, and
$\{\lz_j\}_{j=1}^\fz\subset\cc$ satisfies $\sum_{j=1}^\fz|\lz_j|^p<\fz$.

Let
$$\mathbb{H}^{p,\,q,\,M,\,\epz}_{L_D,\,\mathrm{mol}}(\boz):=\lf\{f\in L^2(\boz):\ f\
\text{has a molecular}\ (p,\,q,\,M,\,\epz)_L\text{-representation}\r\}
$$
equipped with the \emph{quasi-norm}
\begin{align*}
\|f\|_{H^{p,\,q,\,M,\,\epz}_{L,\,\mathrm{mol}}(\boz)}
:=\inf\lf\{\lf(\sum_{j=1}^\fz|\lz_j|^p\r)^{1/p}:\ \sum_{j=1}^\fz\lz_j\az_j\
\text{is a}\ (p,\,q,\,M,\,\epz)_{L_D}\text{-representation of}\ f\r\},
\end{align*}
where the infimum is taken over all molecular $(p,\,q,\,M,\,\epz)_{L_D}$-representations of $f$ as above.
Then the \emph{molecular Hardy space} $H^{p,\,q,\,M,\,\epz}_{L_D,\,\mathrm{mol}}(\boz)$
is then defined as the completion of the set $\mathbb{H}^{p,\,q,\,M,\,\epz}_{L_D,\,\mathrm{mol}}(\boz)$
with respect to the \emph{quasi-norm} $\|\cdot\|_{H^{p,\,q,\,M,\,\epz}_{L_D,\,\mathrm{mol}}(\boz)}$.
\end{itemize}
\end{definition}

By Remark \ref{r2.1}, we find that bounded {\rm NTA} domains of $\rn$ are spaces of homogeneous type.
Thus, for the Hardy space $H^p_{L_D}(\boz)$, we have the following molecular characterization
(see, for instance, \cite{bckyy13b}).

\begin{lemma}\label{l2.4}
Let $n\ge2$, $\boz\subset\rn$ be a bounded {\rm NTA} domain, $L_D$ as in \eqref{eq1.4}, and $p\in(0,1]$.
Then, for any $q\in(1,\fz)$, $M\in\nn\cap(\frac{n}{2p},\fz)$, and $\epz\in(\frac{n}{p},\fz)$,
$H^p_{L_D}(\boz)=H^{p,\,q,\,M,\,\epz}_{L_D,\,\mathrm{mol}}(\boz)$ with equivalent quasi-norms.
\end{lemma}

\section{Proof of Theorem \ref{t1.1}}\label{s3}

In this section, we show Theorem \ref{t1.1}.
Let $L_D$ be as in \eqref{eq1.4}. It was shown by Gr\"{u}ter and Widman \cite{gw82} that,
when $n\ge3$, the Green function associated with the operator $L_D$ exists; in this case,
denote by $G(\cdot,\cdot)$ the Green function associated with $L_D$.
Then we have the following estimates for $G(\cdot,\cdot)$ which are known
(see, for instance, \cite[Remark 4.9]{hk07}).

\begin{lemma}\label{l3.1}
Let $n\ge3$ and $\Omega$ be a bounded {\rm NTA} domain of $\rn$. Then there exist
constants $C\in(0,\fz)$ and $\bz\in(0,1]$, depending only on $n$, the matrix $A$ appeared
in the definition of $L_D$, and $\boz$, such that, for any $x,\,y\in\Omega$,
\begin{equation}\label{eq3.1}
G(x,y)\le C[\dz(y)]^\bz |x-y|^{2-n-\bz},
\end{equation}
here and thereafter, $\dz(y):=\dist(y,\partial\boz)$ with $\partial\boz$ being the boundary of $\boz$.
\end{lemma}

Now, we prove the following H\"older continuity of the Green function $G$
via using Lemma \ref{l3.1} and the Harnack inequality (see, for instance, \cite[Theorem 8.22]{gt01}).

\begin{lemma}\label{l3.2}
Let $n\ge3$ and $\Omega$ be a bounded {\rm NTA} domain of $\rn$.
Then there exist constants $C\in(0,\fz)$ and $\bz\in(0,1]$, depending only on $n$, the matrix $A$ appeared
in the definition of $L_D$, and $\boz$, such that, for any $x,\,y_1,\,y_2\in\Omega$,
$$\lf|G(x,y_1)-G(x,y_2)\r|\le C\lf|y_1-y_2\r|^\bz \lf[|x-y_1|^{2-n-\bz}+|x-y_2|^{2-n-\bz}\r].$$
\end{lemma}

\begin{proof}
We show the present lemma via borrowing some ideas from the proof of \cite[Theorem (1.9)]{gw82}.
Fix $x,\,y_1,\,y_2\in\Omega$. Without loss of generality, we may assume that $G(x,y_1)\ge G(x,y_2)$.
Now, we prove the present lemma by considering the following three cases on $|y_1-y_2|$.

\emph{Case 1)} $|y_1-y_2|\ge|x-y_1|/2$. In this case, by the upper estimate that $G(x,z)\ls |x-z|^{2-n}$ for any
$x,\,z\in\Omega$ (see, for instance, \cite[(1.8)]{gw82}) and $1\ls|y_1-y_2||x-y_1|^{-1}$, we conclude that,
for any given $\bz\in(0,1]$,
\begin{align}\label{eq3.2}
\lf|G(x,y_1)-G(x,y_2)\r| &=G(x,y_1)-G(x,y_2)\le G(x,y_1)\ls |x-y_1|^{2-n}\\ \nonumber
&\ls\lf|y_1-y_2\r|^\bz|x-y_1|^{2-n-\bz}\\ \nonumber
&\ls\lf|y_1-y_2\r|^\bz\lf[|x-y_1|^{2-n-\bz}+|x-y_2|^{2-n-\bz}\r].
\end{align}

\emph{Case 2)} $|y_1-y_2|<|x-y_1|/2$ and $\dz(y_1)\le|y_1-y_2|$. In this case, from Lemma \ref{l3.1}
and $\dz(y_1)\le|y_1-y_2|$, it follows that there exists a $\bz\in(0,1]$ such that
\begin{align}\label{eq3.3}
\lf|G(x,y_1)-G(x,y_2)\r|&\le G(x,y_1)\ls[\dz(y_1)]^\bz|x-y_1|^{2-n-\bz}\\ \nonumber
&\ls\lf|y_1-y_2\r|^\bz\lf[|x-y_1|^{2-n-\bz}+|x-y_2|^{2-n-\bz}\r].
\end{align}

\emph{Case 3)} $|y_1-y_2|<|x-y_1|/2$ and $\dz(y_1)>|y_1-y_2|$. In this case,
let $R:=\min\{\dz(y_1),|x-y_1|/2\}$. It is easy to see that $L_D^\ast G(x,\cdot)=0$ in $B(y_1,R)$.
Applying the Harnack inequality (see, for instance, \cite[Theorem 8.22]{gt01}),
we find that there exists a $\bz\in(0,1]$ such that, for any $0<r_1\le r_2\le R$,
\begin{align*}
\max_{B(y_1,r_1)}G(x,\cdot)-\min_{B(y_1,r_1)}G(x,\cdot)
\ls \lf(\frac{r_1}{r_2}\r)^\bz\lf[\max_{B(y_1,r_2)}G(x,\cdot)\r],
\end{align*}
which further implies that
\begin{align}\label{eq3.4}
\lf|G(x,y_1)-G(x,y_2)\r|\ls \lf|y_1-y_2\r|^\bz R^{-\bz}\max_{B(y_1,R)}G(x,\cdot).
\end{align}
If $R=|x_1-y|/2$, then, by \eqref{eq3.4} and the upper estimate $G(x,y_1)\ls
|x-y_1|^{2-n}$, we conclude that
\begin{align}\label{eq3.5}
\lf|G(x,y_1)-G(x,y_2)\r|&\ls|y_1-y_2|^\bz|x-y_1|^{2-n-\bz}\\ \nonumber
&\ls\lf|y_1-y_2\r|^\bz\lf[|x-y_1|^{2-n-\bz}+|x-y_2|^{2-n-\bz}\r].
\end{align}
If $R=\dz(y_1)$, then $\dz(y_1)\le|x-y_1|/2$. From this and Lemma \ref{l3.1}, we deduce that, for any $z\in B(y_1,R)$,
\begin{align*}
G(x,z)&\ls[\dz(z)]^\bz|x-z|^{2-n-\bz}\ls\lf[|z-y_1|+\dz(y_1)\r]^\bz\lf[|x-y_1|-|y_1-z|\r]^{2-n-\bz}\\
&\ls [\dz(y_1)]^\bz|x-y_1|^{2-n-\bz},
\end{align*}
which, combined with \eqref{eq3.4}, further implies that
$$\lf|G(x,y_1)-G(x,y_2)\r|\ls\lf|y_1-y_2\r|^\bz\lf[|x-y_1|^{2-n-\bz}+|x-y_2|^{2-n-\bz}\r].$$
This, together with \eqref{eq3.2}, \eqref{eq3.3}, and \eqref{eq3.5}, then finishes the proof of Lemma \ref{l3.2}.
\end{proof}

In what follows, for any given $z\in\cc$, denote by $\mathfrak{R}(z)$ and $\arg z$,
respectively, the real part of $z$ and the argument in $(-\pi,\pi]$ of $z$.
Recall that $L_D$ is a $\omega_0$-accretive operator on $L^2(\boz)$ with some $\omega_0\in[0,\pi/2)$
(see, for instance, \cite{at98} for the definition of the $\omega_0$-accretive operator).
For any given $\tz\in[0,\pi)$, let $\Gamma_\tz:=\{z\in\cc\backslash\{0\},\ |\arg z|<\tz\}$.
It is well known that $-L_D$ generates a holomorphic semigroup
$\{e^{-zL_D}\}_{z\in\mathbb{C},\,|\arg z|<\frac{\pi}{2}-\omega_0}$ in $\Gamma_{\frac{\pi}{2}-\omega_0}$
(see, for instance, \cite{at98,o05}).

Moreover, it is known that, for any $\lambda\in \Gamma_{\pi-\omega_0}$,
the operator $L_D+\lambda I$ is bounded on $L^2(\boz)$, and has the integral kernel $G_\lambda$,
where $I$ denotes the identity operator. Namely, for any given $f\in L^2(\Omega)$ and $x\in\Omega$,
$$(L_D+\lambda I)^{-1}f(x)=\int_\Omega G_\lambda(x,y)f(y)\,dy.$$

\begin{lemma}\label{l3.3}
Assume that $n\ge3$ and $\Omega$ is a bounded {\rm NTA} domain of $\rn$.
Let $\mu\in(\omega_0,\pi/2)$. Then there exists a $\dz\in(0,1]$ such that,
for any $\lambda\in\Gamma_{\pi-\mu}$ and $x,\,y_1,\,y_2\in\boz$ with
$x\neq y_i$ for $i\in\{1,\,2\}$,
$$\lf|G_\lambda(x,y_1)-G_\lambda(x,y_2)\r|\le C\lf[\max_{i\in\{1,2\}}e^{-\gz\sqrt{|\lambda|}|x-y_i|}\r]
|y_1-y_2|^{\dz}\lf[\frac{1}{|x-y_1|^{n-2+\dz}}+\frac{|\lambda|^{\dz/2}}{|x-y_2|^{n-2}}\r],$$
where $C$ is a positive constant depending only on $\mu$, $n$, and $\boz$, and $\gz$
is a positive constant depending only on $\mu$ and $\dz$.
\end{lemma}

\begin{proof}
Let $\lambda\in\Gamma_{\pi-\mu}$. By the resolvent identity (see, for instance,
\cite[p.\,36, (9.1)]{p83}), we conclude that, for any $y,\,z\in\Omega$ with $y\neq z$,
\begin{align}\label{eq3.6}
G_\lambda(y,z)=G(y,z)+\lambda\int_\Omega G(t,z)G_\lambda(y,t)\,dt.
\end{align}
Moreover, since $L_D$ generates a holomorphic semigroup $\{e^{-zL_D}\}_{z\in\Gamma_{\frac{\pi}{2}-\omega_0}}$
(see, for instance, \cite[Theorem 1.52]{o05}), and the kernels $\{K^{L_D}_{t}\}_{t>0}$ of the semigroup
$\{e^{-tL_D}\}_{t>0}$ satisfy the Gaussian upper bound estimate \eqref{eq1.5}, similarly to the proof
of \cite[Lemma 2.3]{dhmmy13}, it follows that there exists a positive constant $c$ such that,
for any $y,\,z\in\Omega$ with $y\neq z$,
\begin{align}\label{eq3.7}
\lf|G_\lambda(y,z)\r|\ls e^{-c\sqrt{|\lambda|}\,|y-z|}\frac1{|y-z|^{n-2}}.
\end{align}
Thus, by \eqref{eq3.6}, \eqref{eq3.7}, and Lemma \ref{l3.2}, we find that there
exists a $\bz\in(0,1]$ such that, for any given $x,\,y_1,\,y_2\in\boz$
with $x\neq y_i$ for any $i\in\{1,\,2\}$,
\begin{align}\label{eq3.8}
&\lf|G_\lambda(x,y_1)-G_\lambda(x,y_2)\r|\\ \nonumber
&\quad\le\lf|G(x,y_1)-G(x,y_2)\r|+
|\lambda|\int_\Omega \lf|G(z,y_1)-G(z,y_2)\r||G_\lambda(x,z)|\,dz \\ \nonumber
&\quad\ls\lf|y_1-y_2\r|^\bz\lf[|x-y_1|^{2-n-\bz}+|x-y_2|^{2-n-\bz}\r]\\ \nonumber
&\quad\quad+|\lambda||y_1-y_2|^\bz\int_\Omega\frac{e^{-c\sqrt{|\lambda|}|x-z|}}{|x-z|^{n-2}}
\lf[|z-y_1|^{2-n-\bz}+|z-y_2|^{2-n-\bz}\r]\,dz.
\end{align}
For any $i\in\{1,\,2\}$, let $\boz_{1}:=\{z\in\Omega:\ |x-z|\le|z-y_i|\}$ and
$\boz_{2}:=\{z\in\Omega:\ |x-z|>|z-y_i|\}$.
Then, for any $i\in\{1,\,2\}$, we have
\begin{align}\label{eq3.9}
&\int_\Omega\frac{e^{-c\sqrt{|\lambda|}|x-z|}}{|x-z|^{n-2}} \frac1{|z-y_i|^{n-2+\bz}}\,dz\\ \nonumber
&\quad=\int_{\boz_1}\frac{e^{-c\sqrt{|\lambda|}|x-z|}}{|x-z|^{n-2}}
\frac1{|z-y_i|^{n-2+\bz}}\,dz+\int_{\boz_2}\cdots=:{\rm I}_1+{\rm I}_2.
\end{align}
For ${\rm I}_1$, from the fact that, for any $z\in\boz_1$, $2|z-y_i|\ge|x-z|+|z-y_j|\ge|x-y_i|$,
we deduce that
\begin{align}\label{eq3.10}
{\rm I}_1\le\frac2{|x-y_i|^{n-2+\bz}}\int_{\boz_1}\frac{e^{-c\sqrt{|\lambda|}|x-z|}}
{|x-z|^{n-2}}\,dz\ls\frac1{|x-y_i|^{n-2+\bz}} \frac1{|\lambda|}.
\end{align}
Moreover, it is easy to see that, for any $z\in\boz_2$, $2|x-z|\ge|x-z|+|z-y_i|\ge|x-y_i|$.
By this, we find that
\begin{align}\label{eq3.11}
{\rm I}_2&\le\frac2{|x-y_i|^{n-2}}\int_{\boz_2} \frac{e^{-c\sqrt{|\lambda|}
|z-y_i|}}{|z-y_i|^{n-2+\bz}}\,dz\ls\frac1{|x-y_i|^{n-2}}\frac1{|\lambda|^{1-\bz/2}}.
\end{align}
To finish the proof of the present lemma, we fix $x,\,y_1,\,y_2\in\boz$
with $x\neq y_i$ for any $i\in\{1,\,2\}$, and, without loss of generally, we may
assume that $|x-y_1|\le|x-y_2|$. Thus, from \eqref{eq3.8}, \eqref{eq3.9},
\eqref{eq3.10}, \eqref{eq3.11}, and $|x-y_1|\le|x-y_2|$, it follows that
\begin{align}\label{eq3.12}
&\lf|G_\lambda(x,y_1)-G_\lambda(x,y_2)\r|\\ \nonumber
&\quad\ls\lf|y_1-y_2\r|^\bz\lf[\frac1{|x-y_1|^{n-2+\bz}}
+\frac1{|x-y_2|^{n-2+\bz}}\r] \\ \nonumber
&\quad\quad+\lf|y_1-y_2\r|^\bz\lf[\frac{|\lambda|^{\bz/2}}{|x-y_1|^{n-2}}
+\frac{|\lambda|^{\bz/2}}{|x-y_2|^{n-2}}\r]\\ \nonumber
&\quad\ls\lf|y_1-y_2\r|^\bz\lf[\frac1{|x-y_1|^{n-2+\bz}}+\frac{|\lambda|^{\bz/2}}{|x-y_1|^{n-2}}\r].
\end{align}
Furthermore, by \eqref{eq3.7} and $|x-y_1|\le|x-y_2|$, we conclude that
\begin{align*}
\lf|G_\lambda(x,y_1)-G_\lambda(x,y_2)\r|
\ls\sum^2_{i=1}e^{-c\sqrt{|\lambda|}|x-y_i|}\frac1{|x-y_i|^{n-2}}
\ls e^{-c\sqrt{|\lambda|}|x-y_1|}\frac1{|x-y_1|^{n-2}},
\end{align*}
which, together with \eqref{eq3.12}, further implies that, for any $\dz\in(0,\bz)$,
\begin{align*}
&\lf|G_\lambda(x,y_1)-G_\lambda(x,y_2)\r| \\
&\quad\ls\lf\{\lf|y_1-y_2\r|^\bz\lf[\frac1{|x-y_1|^{n-2+\bz}}+
\frac{|\lambda|^{\bz/2}}{|x-y_1|^{n-2}}\r]\r\}^{\frac\dz{\bz}}
\lf[e^{-c\sqrt{|\lambda|}|x-y_1|}\frac1{|x-y_1|^{n-2}}\r]^{1-\frac\dz{\bz}}\\
&\quad\ls e^{-\gz\sqrt{|\lambda|}|x-y_1|} \lf|y_1-y_2\r|^\dz
\lf[\frac1{|x-y_1|^{n-2+\bz}}+\frac{|\lambda|^{\bz/2}}{|x-y_1|^{n-2}}\r]^{\frac\dz{\bz}}
\lf[\frac1{|x-y_1|^{n-2}}\r]^{1-\frac\dz{\bz}}\\
&\quad\sim e^{-\gz\sqrt{|\lambda|}|x-y_1|}\lf|y_1-y_2\r|^\dz
\lf[\frac{|x-y_1|^{-\bz}}{|x-y_1|^{n-2}}+\frac{|\lambda|^{\bz/2}}{|x-y_1|^{n-2}}\r]^{\frac\dz{\bz}}
\lf[\frac1{|x-y_1|^{n-2}}\r]^{1-\frac\dz{\bz}}\\
&\quad\sim e^{-\gz\sqrt{|\lambda|}|x-y_1|}\lf|y_1-y_2\r|^\dz
\frac1{|x-y_1|^{n-2}}\lf[\frac{1+|x-y_1|^{\bz}|\lambda|^{\bz/2}}{|x-y_1|^{\bz}}\r]^{\frac\dz{\bz}}\\
&\quad\ls e^{-\gz\sqrt{|\lambda|}|x-y_1|}\lf|y_1-y_2\r|^\dz
\frac1{|x-y_1|^{n-2+\dz}}\lf[{1+|x-y_1|^\dz|\lambda|^{\dz/2}}\r]\\
&\quad\sim e^{-\gz\sqrt{|\lambda|}|x-y_1|} \lf|y_1-y_2\r|^\dz
\lf[\frac1{|x-y_1|^{n-2+\dz}}+\frac{|\lambda|^{\dz/2}}{|x-y_1|^{n-2}}\r].
\end{align*}
This finishes the proof of Lemma \ref{l3.3}.
\end{proof}

\begin{lemma}\label{l3.4}
Let $n\ge3$, $\Omega$ be a bounded {\rm NTA} domain of $\rn$, $\mu\in(\omega_0,\pi/2)$,
and $\lambda\in\Gamma_{\pi-\mu}$. Then, for any $m\in\nn$ with $m>(n+2)/2$,
the operator $(L_D+\lambda I)^{-m}$ has a kernel $R_{\lambda,\,m}$ and there exists
a constant $\dz\in(0,1)$ such that, for any $x,\,y_1,\,y_2\in\Omega$ with $x\neq y_i$
for $i\in\{1,\,2\}$,
$$\lf|R_{\lambda,\,m}(x,y_1)-R_{\lambda,\,m}(x,y_2)\r|\le C|\lambda|^{-m+\frac{n}{2}
+\frac{\dz}{2}}|y_1-y_2|^{\dz}\lf[\max_{i\in\{1,2\}}e^{-c\sqrt{|\lambda|}|x-y_i|}\r],$$
where $C$ is a positive constant depending only on $\mu$, $\dz$, $n$, and $\boz$, and
$c$ is a positive constant depending only on $\mu$ and $\dz$.
\end{lemma}

\begin{proof}
We prove the present lemma via borrowing some ideas from the proof of \cite[Lemma 2.6]{dhmmy13}.
Similarly to the proof of \cite[Theorem 1]{do03}, using \eqref{eq1.5}, we find that,
for any $m\in\nn$ with $m>n/2$, $(L_D+\lambda I)^{-m}$ has a kernel $R_{\lambda,\,m}$
and there exists a positive constant $\kappa$ such that, for any $x,\,y\in\Omega$ with $x\neq y$,
\begin{align}\label{eq3.13}
\lf|R_{\lambda,\,m}(x,y)\r|\ls|\lambda|^{-m+\frac{n}{2}}e^{-\kappa\sqrt{|\lambda|}|x-y|}.
\end{align}
Let $m\in\nn$ and $m>(n+2)/2$. Then, from $(L_D+\lambda I)^{-(m+1)}=(L_D+\lambda I)^{-m}(L_D+\lambda I)^{-1}$,
we deduce that, for any $x,\,y\in\Omega$ with $x\neq y$,
\begin{align}\label{eq3.14}
R_{\lambda,\,m+1}(x,y)=\int_\Omega R_{\lambda,\,m}(x,z)G_\lambda(z,y)\,dz.
\end{align}
Therefore, by \eqref{eq3.13}, \eqref{eq3.14}, and Lemma \ref{l3.3}, we find that
there exists a $\dz\in(0,1)$ such that,
for any $x,\,y_1,\,y_2\in\Omega$ with $x\neq y_i$ for any $i\in\{1,2\}$,
\begin{align}\label{eq3.15}
&\lf|R_{\lambda,\,m+1}(x,y_1)-R_{\lambda,\,m+1}(x,y_2)\r| \\ \nonumber
&\quad\le\int_\Omega \lf|R_{\lambda,\,m}(x,z)\r|\lf|G_\lambda(z,y_1)-G_\lambda(z,y_2)\r|\,dz\\ \nonumber
&\quad\ls|\lambda|^{-m+\frac{n}{2}}|y_1-y_2|^\dz\int_\Omega e^{-\kappa\sqrt{|\lambda|}|x-z|}
e^{-\gz\sqrt{|\lambda|}|z-y_1|}\frac1{|z-y_1|^{n-2+\dz}}\,dz\\ \nonumber
&\quad\quad+|\lambda|^{-m+\frac{n}{2}}|y_1-y_2|^\dz\int_\Omega e^{-\kappa\sqrt{|\lambda|}|x-z|}
e^{-\gz\sqrt{|\lambda|}|z-y_1|}\frac{|\lambda|^{\dz/2}}{|x-y_1|^{n-2}}\,dz\\ \nonumber
&\quad\quad+|\lambda|^{-m+\frac{n}{2}} |y_1-y_2|^\dz\int_\Omega e^{-\kappa\sqrt{|\lambda|}|x-z|}
e^{-\gz\sqrt{|\lambda|}|z-y_2|}\frac1{|z-y_2|^{n-2+\dz}}\,dz\\ \nonumber
&\quad\quad+|\lambda|^{-m+\frac{n}{2}} |y_1-y_2|^\dz\int_\Omega e^{-\kappa\sqrt{|\lambda|}|x-z|}
e^{-\gz\sqrt{|\lambda|}|z-y_2|}\frac{|\lambda|^{\dz/2}}{|x-y_2|^{n-2}}\,dz\\ \nonumber
&\quad=:II_1+II_2+II_3+II_4.
\end{align}

Let $c:=\min\{\kappa,\gz/2\}$, where $\kappa$ is as in \eqref{eq3.13}. From the fact that, for any $z\in\boz$,
$|x-z|+|z-y_1|\ge|x-y_1|$, it follows that
\begin{align*}
&\int_\Omega e^{-\kappa\sqrt{|\lambda|}|x-z|}
e^{-\gz\sqrt{|\lambda|}|z-y_1|}\frac1{|z-y_1|^{n-2+\dz}}\,dz\\
&\quad\le\int_\Omega e^{-c\sqrt{|\lambda|}(|x-z|+|z-y_1|)}
e^{-\frac{\gz}{2}\sqrt{|\lambda|}|z-y_1|}\frac1{|z-y_1|^{n-2+\dz}}\,dz\\
&\quad\le e^{-c\sqrt{|\lambda|}|x-y_1|}\int_\Omega e^{-\frac{\gz}{2}\sqrt{|\lambda|}|z-y_1|}
\frac1{|z-y_1|^{n-2+\dz}}\,dz\\
&\quad\ls e^{-c\sqrt{|\lambda|}|x-y_1|}|\lz|^{-1+\frac{\dz}{2}}\int_\Omega e^{-\frac{\gz}{2}|z-y_1|}
\frac1{|z-y_1|^{n-2+\dz}}\,dz\\
&\quad\ls e^{-c\sqrt{|\lambda|}|x-y_1|}|\lz|^{-1+\frac{\dz}{2}},
\end{align*}
which further implies that
\begin{align}\label{eq3.16}
II_1\ls|\lambda|^{-(m+1)+\frac{n}{2}+\frac{\dz}{2}}|y_1-y_2|^{\dz}
e^{-c\sqrt{|\lambda|}|x-y_1|}.
\end{align}
Similarly to the estimation of \eqref{eq3.16}, for any $j\in\{2,\,3,\,4\}$, we also have
$$
II_j\ls|\lambda|^{-(m+1)+\frac{n}{2}+\frac{\dz}{2}}|y_1-y_2|^{\dz}
e^{-c\sqrt{|\lambda|}|x-y_1|},
$$
which, combined with \eqref{eq3.16} and \eqref{eq3.15}, implies that
$$
\lf|R_{\lambda,\,m+1}(x,y_1)-R_{\lambda,\,m+1}(x,y_2)\r|\ls
|\lambda|^{-(m+1)+\frac{n}{2}+\frac{\dz}{2}}|y_1-y_2|^{\dz}
e^{-c\sqrt{|\lambda|}|x-y_1|}.
$$
This finishes the proof of Lemma \ref{l3.4}.
\end{proof}

\begin{lemma}\label{l3.5}
Let $n\ge3$, $\boz$ be a bounded {\rm NTA} domain of $\rn$, the real-valued, bounded, and measurable matrix $A$
satisfy \eqref{eq1.3}, and $L_D$ be as in \eqref{eq1.4}. Then there exists a constant $\dz_0\in(0,1]$,
depending only on $n$, $A$, and $\boz$, such that, for any given $\dz\in(0,\dz_0)$, there exist constants
$C,\,c\in(0,\fz)$ such that, for any $t\in(0,\fz)$ and $x,\,y_1,\,y_2\in\boz$ with $|y_1-y_2|\le\sqrt{t}/2$,
\begin{equation*}
\lf|K^{L_D}_{t}(x,y_1)-K^{L_D}_{t}(x,y_2)\r|\le\frac{C}{t^{n/2}}
\lf[\frac{|y_1-y_2|}{\sqrt{t}}\r]^{\dz}\exp\lf\{-\frac{|x-y_1|^2}{ct}\r\}.
\end{equation*}
\end{lemma}

\begin{proof}
For any given $\mu\in(\pi/2,\pi-\omega_0)$ and $R\in(0,\fz)$, define
$$
\Gamma_1:=\{re^{-i\mu}:\ r\ge R\}, \quad \Gamma_2:=\{Re^{-i\phi}:\ |\phi|\le\mu\},
$$
$$
\Gamma_3:=\{re^{i\mu}:\ r\ge R\},\quad\text{and}\quad\Gamma_R:=\Gamma_1\cup\Gamma_2\cup\Gamma_3.
$$
Let $m\in\nn$ and $m\ge\frac{n+3}{2}$. Using the inverse Laplace transform (see, for instance,
\cite[p.\,30, Theorem 7.7]{p83}), we find that, for any given $t\in(0,\fz)$ and any $x,\,y\in\boz$ with $x\neq y$,
\begin{align}\label{eq3.17}
K^{L_D}_t(x,y)=(-1)^{m}\frac{(m-1)!}{2\pi it^{m-1}}\int_{\Gamma_R}e^{\lz t}R_{\lz,\,m}(x,y)\,d\lz,
\end{align}
where $R\in[R(x,y,t),\fz)$ and
\begin{align}\label{eq3.18}
R(x,y,t):=\max\lf\{\frac{1}{t},\frac{|x-y|^2}{t^2}\r\}.
\end{align}
Fix $x,\,y_1,\,y_2\in\boz$ with $x\neq y_i$ for any $i\in\{1,2\}$, a $t\in(0,\fz)$,
and an
$$R\in\lf[\max\lf\{R(x,y_1,t),R(x,y_2,t)\r\},\fz\r),$$
where $R(x,y_1,t)$ and $R(x,y_2,t)$ are as in \eqref{eq3.18}.
Then, by \eqref{eq3.17} and Lemma \ref{l3.4}, we conclude that
\begin{align}\label{eq3.19}
&\lf|K^{L_D}_t(x,y_1)-K^{L_D}_t(x,y_2)\r|\\ \nonumber
&\quad\ls\sum_{i=1}^2\frac{1}{t^{m-1}}|y_1-y_2|^\dz\int_{\Gamma_R}e^{\mathfrak{R}(\lz t)}
|\lambda|^{-m+\frac{n}{2}+\frac{\dz}{2}}e^{-c\sqrt{|\lambda|}|x-y_i|}d|\lz|\\ \nonumber
&\quad=:{\rm E}+{\rm F},
\end{align}
where $\dz\in(0,1)$ is as in Lemma \ref{l3.4}. From the assumption that $R\ge\max\{\frac{1}{t},
\frac{|x-y_1|^2}{t^2}\}$, it follows that there exists a positive constant $\wz{c}$ such that
\begin{align}\label{eq3.20}
&\frac{1}{t^{m-1}}\int_{\Gamma_1\cup\Gamma_2}e^{\mathfrak{R}(\lz t)}|\lambda|^{-m+\frac{n}{2}
+\frac{\dz}{2}}e^{-c\sqrt{|\lambda|}|x-y_1|}d|\lz|\\ \nonumber
&\quad\ls t^{-\frac{n+\dz}{2}}\int_R^\fz e^{-\wz{c}|\lz|t}
(|\lz|t)^{-m+\frac{n}{2}+\frac{\dz}{2}+1}e^{-c\sqrt{|\lambda|}|x-y_1|}\frac{d|\lz|}{|\lz|}\\ \nonumber
&\quad\ls t^{-\frac{n+\dz}{2}}e^{-c\sqrt{R}|x-y_1|}e^{-\frac{\wz{c}}{2}Rt}\int_1^\fz
e^{-\frac{\wz{c}}{2}s}s^{-m+\frac{n}{2}+\frac{\dz}{2}}ds\\ \nonumber
&\quad\ls t^{-\frac{n+\dz}{2}}e^{-c\sqrt{R}|x-y_1|}e^{-\frac{\wz{c}}{2}Rt}\ls
t^{-\frac{n+\dz}{2}}e^{-c_1\frac{|x-y_1|^2}{t}},
\end{align}
where $c$ is as in Lemma \ref{l3.4}, and $c_1:=c+\frac{\wz{c}}{2}$. Moreover, by the assumptions
that $-m+\frac{n}{2}+\frac{\dz}{2}+1<0$ and $R\ge\max\{\frac{1}{t},\frac{|x-y_1|^2}{t^2}\}$ again,
we find that
\begin{align*}
&\frac{1}{t^{m-1}}\int_{\Gamma_3}e^{\mathfrak{R}(\lz t)}|\lambda|^{-m+\frac{n}{2}+\frac{\dz}{2}}
e^{-c\sqrt{|\lambda|}|x-y_1|}d|\lz|\\ \nonumber
&\quad\ls t^{-\frac{n+\dz}{2}}\int_{\{\lz\in\cc:\ |\lz|=R\}}e^{-\wz{c}Rt}
(Rt)^{-m+\frac{n}{2}+\frac{\dz}{2}+1}e^{-c\sqrt{R}|x-y_1|}\frac{d|\lz|}{|\lz|}\\ \nonumber
&\quad\ls t^{-\frac{n+\dz}{2}}(Rt)^{-m+\frac{n}{2}+\frac{\dz}{2}+1}e^{-c_1\frac{|x-y_1|^2}{t}}
\ls t^{-\frac{n+\dz}{2}}e^{-c_1\frac{|x-y_1|^2}{t}},
\end{align*}
which, together with \eqref{eq3.20}, further implies that
\begin{align}\label{eq3.21}
{\rm E}\ls t^{-\frac{n}{2}}e^{-c_1\frac{|x-y_1|^2}{t}}\lf[\frac{|y_1-y_2|}{t^{1/2}}\r]^\dz.
\end{align}
Similarly to the estimation of \eqref{eq3.21}, we also have
\begin{align*}
{\rm F}\ls t^{-\frac{n}{2}}e^{-c_1\frac{|x-y_2|^2}{t}}\lf[\frac{|y_1-y_2|}{t^{1/2}}\r]^\dz.
\end{align*}
From this, \eqref{eq3.21}, and \eqref{eq3.19}, we deduce that, for any $t\in(0,\fz)$
and $x,\,y_1,\,y_2\in\boz$ satisfying $|y_1-y_2|\le \sqrt{t}/2$,
\begin{align*}
\lf|K^{L_D}_t(x,y_1)-K^{L_D}_t(x,y_2)\r|\ls
t^{-\frac{n}{2}}\lf(\max_{i\in\{1,\,2\}}e^{-c_1\frac{|x-y_i|^2}{t}}\r)\lf[\frac{|y_1-y_2|}{t^{1/2}}\r]^\dz
\ls t^{-\frac{n}{2}}e^{-c_1\frac{|x-y_1|^2}{t}}\lf[\frac{|y_1-y_2|}{t^{1/2}}\r]^\dz.
\end{align*}
This finishes the proof of Lemma \ref{l3.5}.
\end{proof}

To show Theorem \ref{t1.1}, we also need the following global regularity
estimate for the Dirichlet problem \eqref{eq1.1}.

\begin{lemma}\label{l3.6}
Let $n\ge2$, $\boz\subset\rn$ be a bounded {\rm NTA} domain, and the real-valued, bounded,
and measurable matrix $A$ satisfy \eqref{eq1.3}. Then there exists a positive constant $p_0\in(2,\fz)$,
depending only on $n$, $A$, and $\boz$, such that, for any given $p\in[2,p_0)$, the Dirichlet
problem \eqref{eq1.1}, with $f\in L^{p_\ast}(\boz)$, is uniquely solvable and, moreover,
for any weak solution $u$ of the problem \eqref{eq1.1}, $u\in W^{1,p}_0(\boz)$ and there
exists a positive constant $C$, depending only on $n$, $p$, and $\boz$, such that
\begin{equation}\label{eq3.22}
\|\nabla u\|_{L^p(\boz;\rn)}\le C\|f\|_{L^{p_\ast}(\boz)}.
\end{equation}
\end{lemma}

To prove Lemma \ref{l3.6}, we need the global regularity estimate for the
following Dirichlet problem \eqref{eq3.23}.

Let $p\in(1,\fz)$ and $\mathbf{f}\in L^p(\boz;\rn)$. Then a function $u$
is called a \emph{weak solution} of the following \emph{Dirichlet boundary value problem}
\begin{equation}\label{eq3.23}
\begin{cases}
-\dive (A\nabla u)=\dive (\mathbf{f})\ \ & \text{in}\ \ \boz,\\
u=0 \ \ & \text{on}\ \ \partial\boz
\end{cases}
\end{equation}
if $u\in W^{1,p}_{0}(\boz)$ and, for any $\varphi\in C^\fz_\mathrm{c}(\boz)$,
\begin{equation}\label{eq3.24}
\int_{\boz}A(x)\nabla u(x)\cdot\nabla\varphi(x)\,dx=-\int_\boz\mathbf{f}(x)\cdot\nabla\varphi(x)\,dx.
\end{equation}
Moreover, the Dirichlet problem \eqref{eq3.23} is said to be \emph{uniquely solvable} if,
for any $\mathbf{f}\in L^p(\boz;\rn)$, there exists a unique $u\in W^{1,p}_0(\boz)$
such that \eqref{eq3.24} holds true for any $\varphi\in C^\fz_\mathrm{c}(\boz)$.

By the Lax--Milgram theorem, we conclude that, when $p=2$, the Dirichlet problem \eqref{eq3.23},
with $\mathbf{f}\in L^2(\boz;\rn)$, is uniquely solvable and the weak solution $u$ satisfies
\begin{equation*}
\|\nabla u\|_{L^2(\boz;\rn)}\le \mu_0^{-1}\|\mathbf{f}\|_{L^{2}(\boz;\rn)},
\end{equation*}
where $\mu_0$ is as in \eqref{eq1.3}.

\begin{lemma}\label{l3.7}
Let $n\ge2$, $\boz\subset\rn$ be a bounded {\rm NTA} domain, and the real-valued, bounded, and measurable
matrix $A$ satisfy \eqref{eq1.3}. Then there exists a positive constant $p_0\in(2,\fz)$,
depending only on $n$, $A$, and $\boz$, such that, for any given $p\in(p_0',p_0)$ with $1/p_0'+1/p_0=1$,
the Dirichlet problem \eqref{eq3.23}, with $\mathbf{f}\in L^{p}(\boz;\rn)$, is uniquely solvable and, moreover,
for any weak solution $u$ of the problem \eqref{eq3.23}, $u\in W^{1,p}_0(\boz)$ and there
exists a positive constant $C$, depending only on $n$, $p$, and $\boz$, such that
\begin{equation*}
\|\nabla u\|_{L^p(\boz;\rn)}\le C\|\mathbf{f}\|_{L^{p}(\boz;\rn)}.
\end{equation*}
\end{lemma}

Via replacing \cite[Lemma 4.3]{yyy20} by the following Lemma \ref{l3.8}, using a real-variable argument
for $L^p(\boz)$ estimates, which was essentially established by Shen \cite[Theorem 3.4]{sh07}
(see also \cite[Theorem 4.2.6]{sh18} and \cite[Theorem 3.3]{sh05a}), and then
repeating the proof of \cite[Theorem 1.10]{yyy20}, we can prove Lemma \ref{l3.7};
we omit the details here.

\begin{lemma}\label{l3.8}
Let $n\ge2$, $\boz\subset\rn$ be a bounded $\mathrm{NTA}$ domain, and $B(x_0,r)$ a ball such
that $r\in(0,r_0/4)$ and either $x_0\in\partial\boz$ or $B(x_0,2r)\subset\boz$, where
$r_0\in(0,\diam(\boz))$ is a constant. Assume that the real-valued, bounded,
and measurable matrix $A$ satisfies \eqref{eq1.3}, and $u\in W^{1,2}(B_\boz(x_0,2r))$ is a weak
solution of the following Dirichlet problem
\begin{equation*}
\begin{cases}
\dive (A\nabla u)=0\ \ & \text{in}\ \ B_\boz(x_0,2r),\\
u=0 \ \ & \text{on}\ \ B(x_0,2r)\cap\partial\boz.
\end{cases}
\end{equation*}
Then there exists a constant $p_0\in(2,\fz)$, depending on $\boz$, $n$, and $\mu_0$ in \eqref{eq1.3},
such that
$$\lf[\fint_{B_\boz(x_0,r)}|\nabla u(x)|^{p_0}\,dx\r]^{1/p_0}\le C
\lf[\fint_{B_\boz(x_0,2r)}|\nabla u(x)|^2\,dx\r]^{1/2},
$$
where $C$ is a positive constant depending only on $n$, $\boz$, and $p_0$.
\end{lemma}

Lemma \ref{l3.8} was established in \cite[Lemma 3.2 and Corollary 4.1]{lp19}.

Now, we show Lemma \ref{l3.6} by using Lemma \ref{l3.7}.

\begin{proof}[Proof of Lemma \ref{l3.6}]
Let $p_0\in(2,\fz)$ be as in Lemma \ref{l3.8}, $p\in[2,p_0)$, and $f\in L^{p_\ast}(\boz)$.
Consider the Dirichlet problem
\begin{equation}\label{eq3.25}
\begin{cases}
-\dive (A^\ast\nabla v)=\dive (\mathbf{g})\ \ & \text{in}\ \ \boz,\\
v=0 \ \ & \text{on}\ \ \partial\boz
\end{cases}
\end{equation}
with $\mathbf{g}\in L^{p'}(\boz;\rn)$, where $\frac1p+\frac1{p'}=1$. Then, by Lemma \ref{l3.7} and
$\mathbf{g}\in L^{p'}(\boz;\rn)$ with $p'\in(p_0',p_0)$, we conclude that the Dirichlet problem
\eqref{eq3.25} is uniquely solvable and the weak solution $v$ satisfies
\begin{align}\label{eq3.26}
\|\nabla v\|_{L^{p'}(\boz;\rn)}\ls\|\mathbf{g}\|_{L^{p'}(\boz;\rn)}.
\end{align}
Moreover, from the assumptions that $p\ge2$ and $f\in L^{p_\ast}(\boz)$,
and Remark \ref{r1.3}, we deduce that there exists a weak solution $u\in W^{1,2}_0(\boz)$
for the Dirichlet problem \eqref{eq1.1} with $f\in L^{p_\ast}(\boz)$. Moreover, we have
\begin{align*}
\int_\boz \nabla u(x)\cdot \mathbf{g}(x)\,dx&=-\int_\boz A^\ast(x)\nabla v(x)\cdot\nabla u(x)\,dx\\
&=-\int_\boz A(x)\nabla u(x)\cdot\nabla v(x)\,dx=-\int_\boz f(x)v(x)\,dx,
\end{align*}
which, together with the H\"older inequality, the Sobolev embedding theorem
on NTA domains (see, for instance, \cite[Theorem 1.1]{bk95}), and \eqref{eq3.26},
further implies that
\begin{align*}
\|\nabla u\|_{L^{p}(\boz;\rn)}&=\sup_{\|\mathbf{g}\|_{L^{p'}(\boz;\rn)}\le 1}
\lf|\int_\boz \nabla u(x)\cdot\mathbf{g}(x)\,dx\r|=\sup_{\|\mathbf{g}\|_{L^{p'}(\boz;\rn)}\le 1}
\lf|\int_\boz f(x) v(x)\,dx\r|\\
&\le\sup_{\|\mathbf{g}\|_{L^{p'}(\boz;\rn)}\le 1}
\|f\|_{L^{p_\ast}(\boz)}\|v\|_{L^{(p_\ast)'}(\boz)}\\
&\ls\sup_{\|\mathbf{g}\|_{L^{p'}(\boz;\rn)}\le 1}
\|f\|_{L^{p_\ast}(\boz)}\|\nabla v\|_{L^{p'}(\boz;\rn)}\\
&\ls\sup_{\|\mathbf{g}\|_{L^{p'}(\boz;\rn)}\le 1}
\|f\|_{L^{p_\ast}(\boz)}\|\mathbf{g}\|_{L^{p'}(\boz;\rn)}\ls\|f\|_{L^{p_\ast}(\boz)}.
\end{align*}
By this and Remark \ref{r1.3}, we find that the Dirichlet problem \eqref{eq1.1}
with $f\in L^{p_\ast}(\boz)$ is uniquely solvable and \eqref{eq3.22} holds true.
This finishes the proof of Lemma \ref{l3.6}.
\end{proof}

Furthermore, to prove Theorem \ref{t1.1}, we also need the following global
gradient estimate for the problem \eqref{eq1.1} on $(\gz,\sz,R)$ quasi-convex domains.

\begin{lemma}\label{l3.9}
Let $n\ge2$, $\boz\subset\rn$ be a bounded {\rm NTA} domain, and
the real-valued, bounded, and measurable matrix $A$ satisfy \eqref{eq1.3}.
Assume that $p\in(2,\fz)$. Then there exists a positive constant $\gz_0\in(0,1)$,
depending only on $n$, $p$, and $\boz$, such that, if $A$ satisfies the $(\gz,R)\mbox{-}\bbmo$ condition
and $\boz$ is a $(\gz,\sz,R)$ quasi-convex domain for some $\gz\in(0,\gz_0)$,
$\sigma\in(0,1)$, and $R\in(0,\fz)$, then the Dirichlet problem \eqref{eq1.1}, with $f\in L^{p_\ast}(\boz)$,
is uniquely solvable and, moreover, for any weak solution $u$ of the problem \eqref{eq1.1},
$u\in W^{1,p}_0(\boz)$ and there exists a positive constant $C$, depending only on $n$, $p$,
and $\boz$, such that
\begin{equation*}
\|\nabla u\|_{L^p(\boz;\rn)}\le C\|f\|_{L^{p_\ast}(\boz)}.
\end{equation*}
\end{lemma}

\begin{lemma}\label{l3.10}
Let $n\ge2$, $\boz\subset\rn$ be a bounded {\rm NTA} domain, and $r_0\in(0,\diam(\boz))$. Assume
that the matrix $A$ is as in Lemma \ref{l3.9}. Let $v\in W^{1,2}(B_\boz(x_0,4r))$ be a weak
solution of the equation $\dive (A\nabla v)=0$ in $B_\boz(x_0,4r)$ satisfying $v=0$ on
$B(x_0,4r)\cap\partial\boz$, where $r\in(0,r_0/4)$ and either $x_0\in\partial\boz$ or $B(x_0,4r)\subset\boz$.
Then, for any given $p\in(2,\fz)$, there exists a constant $\gz_0\in(0,1)$, depending only on
$n,\,p$, and $\boz$, such that, if $A$ satisfies the $(\gz,R)\mbox{-}\bbmo$ condition and $\boz$
is a $(\gz,\sz,R)$ quasi-convex domain for some $\gz\in(0,\gz_0)$, $\sigma\in(0,1)$, and
$R\in(0,\fz)$, then the weak reverse H\"older inequality
\begin{equation*}
\lf[\fint_{B_\boz(x_0,2r)}|\nabla v(x)|^p\,dx\r]^{\frac1p}\le C\lf[\fint_{B_\boz(x_0,4r)}
|\nabla v(x)|^2\,dx\r]^{\frac12}
\end{equation*}
holds true, where $C$ is a positive constant independent of $v$, $x_0$, and $r$.
\end{lemma}

Lemma \ref{l3.10} was established in \cite{yyy20}.

\begin{proof}[Proof of Lemma \ref{l3.9}]
Via replacing Lemma \ref{l3.8} by Lemma \ref{l3.10}, and repeating the proof of Lemma \ref{l3.6},
we prove the present lemma, which completes the proof of Lemma \ref{l3.9}.
\end{proof}

Finally, we prove Theorem \ref{t1.1} by using Lemmas \ref{l3.5}, \ref{l3.6},
and \ref{l3.9}.

\begin{proof}[Proof of Theorem \ref{t1.1}]
We first show (i). When $n\ge3$, by Lemma \ref{l3.5}, we find that the desired conclusion of (i)
holds true in this case.

Now, we assume that $n=2$. Denote by $L_D^\ast$ the adjoint operator of $L_D$.
Then $L_D^\ast=-\dive (A^\ast\nabla\cdot)$, where $A^\ast$ is the transpose of
the matrix $A$. Denote by $\{e^{-tL^\ast_D}\}_{t>0}$ the semigroup generated by $L_D^\ast$.
Then, for any $t\in(0,\fz)$, $(e^{-tL_D})^\ast=e^{-tL^\ast_D}$ (see, for instance,
\cite[p.\,41, Corollary 10.6]{p83}), where $(e^{-tL_D})^\ast$ denotes the adjoint operator $e^{-tL_D}$,
which implies that, for any $t\in(0,\fz)$ and $x,\,y\in\boz$, $K^{L_D}_t(x,y)=K^{L_D^\ast}_t(y,x)$.
Here and thereafter, $\{K^{L_D^\ast}_t\}_{t>0}$ denote the kernels of the semigroup $\{e^{-tL^\ast_D}\}_{t>0}$.
Thus, to prove (i) in the case of $n=2$, it suffices to show that there exists a $\dz_0\in(0,1]$ such that,
for any given $\dz\in(0,\dz_0)$, there exists a constant $c\in(0,\fz)$ such that, for any $t\in(0,\fz)$ and
$x,\,y_1,\,y_2\in\boz$ with $|y_1-y_2|\le\sqrt{t}/2$,
\begin{equation}\label{eq3.27}
\lf|K^{L_D^\ast}_{t}(y_1,x)-K^{L_D^\ast}_{t}(y_2,x)\r|\ls\frac{1}{t^{n/2}}
\lf[\frac{|y_1-y_2|}{\sqrt{t}}\r]^{\dz}\exp\lf\{-\frac{|x-y_1|^2}{ct}\r\}.
\end{equation}

For any given $t\in(0,\fz)$ and $x\in \boz$, let $u(\cdot):=K^{L^\ast_D}_t(\cdot,x)$
and $f(\cdot):=-\frac{d}{dt}K^{L^\ast_D}_t(\cdot,x)$. From the facts that, for any $t\in(0,\fz)$ and $x\in\boz$,
$K^{L^\ast_D}_t(\cdot,x),\,\frac{d}{dt}K^{L^\ast_D}_t(\cdot,x)\in\cd(L^\ast_D)\subset W^{1,2}_0(\boz)$,
and $L^\ast_De^{-tL_D^\ast}g=-\frac{d}{dt}e^{-tL_D^\ast}g$ for any $g\in L^2(\boz)$, we deduce that
\begin{equation}\label{eq3.28}
\begin{cases}
\dive (A^\ast\nabla u)=f\ \ &\text{in}\ \ \boz,\\
u=0 \ \ &\text{on}\ \ \partial\boz.
\end{cases}
\end{equation}
Let $p_0\in(2,\fz)$ be as in Lemma \ref{l3.6}, and $\dz_0:=1-2/p_0$.
For any given $\dz\in(0,\dz_0)$, take $\dz_1\in(\dz,\dz_0)$ sufficiently large
such that there exists a $p_1\in(2,p_0)$ satisfying $\dz_1:=1-2/p_1$.
By \eqref{eq3.28} and Lemma \ref{l3.6}, we conclude that
\begin{equation}\label{eq3.29}
\lf\|\nabla K^{L^\ast_D}_t(\cdot,x)\r\|_{L^{p_1}(\boz;\rn)}\ls
\lf\|\frac{d}{dt}K^{L^\ast_D}_t(\cdot,x)\r\|_{L^{(p_1)_\ast}(\boz)}.
\end{equation}
Then, from $n=2$, the Sobolev embedding theorem (see, for instance,
\cite[Theorem 7.26]{gt01}), the Poincar\'e inequality (see, for instance, \cite[Theorem 1.1]{bk95}), and
\eqref{eq3.29}, it follows that $K^{L^\ast_D}_t(\cdot,x)\in C^{0,\,\dz_1}(\overline{\boz})$ and,
for any $t\in(0,\fz)$ and $x\in\boz$,
\begin{align}\label{eq3.30}
\lf\|K^{L^\ast_D}_t(\cdot,x)\r\|_{C^{0,\,\dz_1}(\overline{\boz})}&\ls
\lf\|K^{L^\ast_D}_t(\cdot,x)\r\|_{W^{1,p_1}(\boz)}\ls\lf\|\nabla
K^{L^\ast_D}_t(\cdot,x)\r\|_{L^{p_1}(\boz;\rn)}\\ \nonumber
&\ls\lf\|\frac{d}{dt}K^{L^\ast_D}_t(\cdot,x)\r\|_{L^{(p_1)_\ast}(\boz)}.
\end{align}
Here, $C^{0,\,\dz_1}(\overline{\boz})$ denotes the \emph{H\"older space} on $\boz$,
which is defined by setting
$$C^{0,\,\dz_1}(\overline{\boz}):=\lf\{g\ \text{is bounded and continuous on}\ \boz:\
\|g\|_{C^{0,\,\dz_1}(\overline{\boz})}<\fz\r\}
$$
with
$$\|g\|_{C^{0,\,\dz_1}(\overline{\boz})}:=\sup_{x,\,y\in\boz,\,x\neq y}\frac{|g(x)-g(y)|}{|x-y|^{\dz_1}}.$$
Furthermore, by \cite[Theorem 6.17]{o05} and \eqref{eq1.5}, we find that,
for any $t\in(0,\fz)$ and $x,\,y\in\boz$,
\begin{align*}
\lf|\frac{d}{dt}K^{L^\ast_D}_t(x,y)\r|\ls\frac{1}{t^{1+n/2}}\exp\lf\{-\frac{|x-y|^2}{ct}\r\},
\end{align*}
which further implies that
\begin{align}\label{eq3.31}
\lf\|\frac{d}{dt}K^{L^\ast_D}_t(\cdot,x)\r\|_{L^{(p_1)_\ast}(\boz)}
\ls\lf[\int_\boz\frac{1}{t^{1+n/2}}e^{-\frac{(p_1)_\ast|x-y|^2}{ct}}\,dy\r]^{1/(p_1)_\ast}
\ls\frac{1}{t^{1+\frac{n}{2}[1-\frac{1}{(p_1)_\ast}]}}.
\end{align}
From \eqref{eq3.30}, \eqref{eq3.31}, and $\frac{1}{(p_1)_\ast}-\frac{1}{p_1}=\frac1n$,
we deduce that, for any $t\in(0,\fz)$ and $x\in\boz$,
\begin{align}\label{eq3.32}
\lf\|K^{L^\ast_D}_t(\cdot,x)\r\|_{C^{0,\,\dz_1}(\overline{\boz})}
\ls\frac{1}{t^{n/2}}\frac{1}{t^{\frac{1}{2}(1-\frac{n}{p_1})}}\sim\frac{1}{t^{(n+\dz_1)/2}}.
\end{align}
Thus, by \eqref{eq3.32}, we conclude that, for any $t\in(0,\fz)$ and $x,\,y_1,\,y_2\in\boz$,
\begin{align}\label{eq3.33}
\lf|K^{L^\ast_D}_t(y_1,x)-K^{L^\ast_D}_t(y_2,x)\r|
\ls\frac{1}{t^{n/2}}\lf[\frac{|y_1-y_2|}{t^{1/2}}\r]^{\dz_1}.
\end{align}
On the other hand, from \eqref{eq1.5}, it follows that, for any $t\in(0,\fz)$ and
$x,\,y_1,\,y_2\in\boz$ with $|y_1-y_2|\le\sqrt{t}/2$,
$$
\lf|K^{L^\ast_D}_t(y_1,x)-K^{L^\ast_D}_t(y_2,x)\r|\le
\lf|K^{L^\ast_D}_t(y_1,x)\r|+\lf|K^{L^\ast_D}_t(y_2,x)\r|
\ls\frac{1}{t^{n/2}}\exp\lf\{-\frac{|x-y_1|^2}{ct}\r\},
$$
which, together with \eqref{eq3.33}, further implies that
\begin{align*}
\lf|K^{L^\ast_D}_t(y_1,x)-K^{L^\ast_D}_t(y_2,x)\r|\ls\frac{1}{t^{n/2}}\lf[\frac{|y_1-y_2|}{t^{1/2}}\r]^{\dz}
\exp\lf\{-\frac{|x-y_1|^2}{ct}\r\}.
\end{align*}
This finishes the proof of \eqref{eq3.27}, and hence of (i).

Next, we show (ii). To prove this, similarly to the proof of (i) in
the case of $n=2$, it suffices to show that, for any given $\dz_0\in(0,1)$,
there exists a constant $\gz_0\in(0,\fz)$, depending on $\dz_0$, $n$, and $\boz$,
such that, if $A$ satisfies the $(\gz,R)\mbox{-}\bbmo$ condition and
$\boz$ is a $(\gz,\sz,R)$ quasi-convex domain for some $\gz\in(0,\gz_0)$, $\sigma\in(0,1)$,
and $R\in(0,\fz)$, then, for any given $\dz\in(0,\dz_0)$, there exists a positive constant
$c\in(0,\fz)$ such that, for any $t\in(0,\fz)$ and $x,\,y_1,\,y_2\in\boz$ with $|y_1-y_2|\le\sqrt{t}/2$,
\begin{equation}\label{eq3.34}
\lf|K^{L_D^\ast}_{t}(y_1,x)-K^{L_D^\ast}_{t}(y_2,x)\r|\ls\frac{1}{t^{n/2}}
\lf[\frac{|y_1-y_2|}{\sqrt{t}}\r]^{\dz}\exp\lf\{-\frac{|x-y_1|^2}{ct}\r\}.
\end{equation}
For any given $\dz_0\in(0,1)$, let $p\in(n,\fz)$ be such that
$\dz_0=1-n/p$. By Lemma \ref{l3.9}, we find that there exists a constant $\gz_0\in(0,1)$,
depending on $p$, $n$, and $\boz$, such that, if $A$ satisfies the $(\gz,R)\mbox{-}\bbmo$
condition and $\boz$ is a $(\gz,\sz,R)$ quasi-convex domain for some $\gz\in(0,\gz_0)$,
$\sigma\in(0,1)$, and $R\in(0,\fz)$, then
\begin{equation}\label{eq3.35}
\lf\|\nabla K^{L^\ast_D}_t(\cdot,x)\r\|_{L^p(\boz;\rn)}\ls
\lf\|\frac{d}{dt}K^{L^\ast_D}_t(\cdot,x)\r\|_{L^{p_\ast}(\boz)}.
\end{equation}
Replacing \eqref{eq3.29} by \eqref{eq3.35}, and repeating the estimation of
\eqref{eq3.27}, we conclude that \eqref{eq3.34} holds true. This finishes the proof of
(ii), and hence of Theorem \ref{t1.1}.
\end{proof}

\section{Proof of Theorem \ref{t1.2}}\label{s4}

In this section, we show Theorem \ref{t1.2}. We begin with recalling some concepts on the tent
space $T^{p}(\boz\times(0,\fz))$ (see, for instance, \cite{r07}).

\begin{definition}\label{d4.1}
Let $\Omega\subset\rn$ be a bounded {\rm NTA} domain and $p\in(0,\fz)$. Then the \emph{tent space} $T^{p}(\boz\times(0,\fz))$
is defined to be the set of all measurable functions $f$ on $\boz\times(0,\fz)$ such that $\|f\|_{T^{p}(\boz\times(0,\fz))}:
=\|\ca(f)\|_{L^{p}(\Omega)}<\infty$, where, for any $x\in\boz$,
$$\ca(f)(x):=\lf[\int^\infty_0\int_{\Gamma(x)}|f(y,t)|^2\frac{dy\,dt}{|B_\boz(x,t)|t}\r]^{1/2},$$
where $\Gamma(x):=\{y\in\Omega:\ |x-y|<t\}$.

Moreover, assume that $O$ is an open subset of $\Omega$. Then the \emph{tent} over $O$, denoted by $T(O)$,
is defined by setting
$$T(O):=\lf\{(x,t)\in\Omega\times(0,\infty):\mathrm{dist}\left(x,O^\complement\cap\boz\right)\ge t\r\}.$$
Let $p\in(0,1]$ and $q\in(1,\fz)$. A measurable function $a$ on $\boz\times(0,\fz)$ is called a \emph{$(p,\,q)$-atom}
if there exists a ball $B_\boz$ of $\boz$, which means $B_\boz:=B\cap\boz$ and $B:=B(x_B,r_B)$
is a ball of $\rn$ with $x_B\in\boz$ and $r_B\in(0,\diam(\boz))$, such that
\begin{enumerate}
\item[\rm{(i)}]  $\supp(a):=\lf\{(x,t)\in\Omega \times(0,\infty):\ a(x,t)\neq0\r\}\subset T(B_\boz)$;
\item[\rm{(ii)}] $\|a\|_{T^{q}(\Omega\times(0,\fz))}\le|B_\boz|^{1/q-1/p}$.
\end{enumerate}
\end{definition}

Then we have the following atomic decomposition of $T^{p}(\Omega\times(0,\fz))$, which is
a special case of \cite[Theorem 1.1]{r07}.

\begin{lemma}\label{l4.1}
Let $n\ge2$, $\boz\subset\rn$ be a bounded {\rm NTA} domain, $p\in(0,1]$, and $q\in(1,\fz)$.
Then, for any $f\in T^{p}(\Omega\times(0,\fz))$, there exist $\{\lambda_j\}_{j\in\nn}\subset\mathbb{C}$
and a sequence $\{a_j\}_{j\in\nn}$ of $(p,\,q)$-atoms such that, for almost every
$(x,t)\in \boz\times(0,\infty)$, $f(x,t)=\sum_{j\in\nn}\lambda_ja_j(x,t)$ and
$$C^{-1}\|f\|_{T^{p}(\boz\times(0,\fz))}\le \lf(\sum_{j=1}^\fz|\lz_j|^p\r)^{1/p}\le C\|f\|_{T^{p}(\boz\times(0,\fz))},$$
where $C$ is a positive constant independent of $f$.
Moreover, if $f\in T^{p}(\boz\times(0,\fz))\cap T^2(\boz\times(0,\fz))$, then $f(x,t)=\sum_{j\in\nn}\lambda_ja_j(x,t)$
holds true in both $T^{p}(\boz\times(0,\fz))$ and $T^2(\boz\times(0,\fz))$.
\end{lemma}

Next, we prove Theorem \ref{t1.2} by using Lemmas \ref{l2.1}, \ref{l2.2}, \ref{l2.3}, \ref{l2.4}, and \ref{l4.1}.

\begin{proof}[Proof of Theorem \ref{t1.2}]
Let $p\in(\frac{n}{n+1},1]$. We first prove that
\begin{equation}\label{eq4.1}
H^p_r(\boz)=H^p(\boz)
\end{equation}
with equivalent quasi-norms. Let $f\in H^p_r(\boz)$. Then there exists an $F\in H^p(\rn)$ such that $F|_\boz=f$ and
\begin{equation}\label{eq4.2}
\|f\|_{H^p_r(\boz)}\sim\|F\|_{H^p(\rn)}.
\end{equation}
Moreover, by Lemma \ref{l2.2}, we find that there exist a sequence $\{a_j\}_{j\in\nn}$
of $(p,\,\fz,\,0)$-atoms, and $\{\lz_j\}_{j\in\nn}\subset\cc$ satisfying $\sum_{j=1}^\fz|\lz_j|^p<\fz$ such that
$F=\sum_{j=1}^\fz\lz_ja_j$ in $\cs'(\rn)$, and
\begin{equation}\label{eq4.3}
\|F\|_{H^p(\rn)}\sim\lf(\sum_{j=1}^\fz|\lz_j|^p\r)^{1/p}.
\end{equation}
To show $f\in H^p(\boz)$ and $\|f\|_{H^p(\boz)}\ls\|f\|_{H^p_r(\boz)}$,
it suffices to prove that, for any $(p,\,\fz,\,0)$-atom $a$ supported in the ball $B:=B(x_B,r_B)\subset\rn$
with $x_B\in\rn$ and $r_B\in(0,\fz)$, $a\in H^p(\boz)$ and
\begin{equation}\label{eq4.4}
\|a\|_{H^p(\boz)}\ls1.
\end{equation}
Indeed, if \eqref{eq4.4} holds true, from \eqref{eq4.2}, \eqref{eq4.3}, and \eqref{eq4.4}, it follows that
$f\in H^p(\boz)$ and
\begin{equation*}
\|f\|^p_{H^p(\boz)}=\|F\|^p_{H^p(\boz)}\ls\sum_{j=1}^\fz|\lz_j|^p\|a_j\|^p_{H^p(\boz)}\ls
\sum_{j=1}^\fz|\lz_j|^p\sim\|f\|^p_{H^p_r(\boz)},
\end{equation*}
which further implies that $\|f\|_{H^p(\boz)}\ls\|f\|_{H^p_r(\boz)}$.

Now, we prove \eqref{eq4.4} by considering the following two cases on the ball $B$.

\emph{Case 1)} $4B\subset\boz$. In this case, by the definitions of $(p,\,\fz,\,0)$-atoms,
we conclude that $a$ is a type $(a)$ $(p,\,\fz,\,0)_\boz$-atom, which, combined with Lemma \ref{l2.3},
implies that $a\in H^p(\boz)$ and \eqref{eq4.4} holds true.

\emph{Case 2)} $4B\cap\partial\boz\neq\emptyset$. In this case, let $\phi$ be as in Definition \ref{d1.1}(iv).
We first claim that $\supp(a^+_\boz)\subset 8B_\boz$. We prove this claim via borrowing some ideas from
the proof of \cite[Theorem 3.1]{bd18}. Indeed, for any $x\in (8B_\boz)^\complement$,
\begin{align}\label{eq4.5}
a^+_\boz(x)&=\sup_{t\in(0,\dz(x)/2)}\lf|\int_{B_\boz}\frac{1}{t^n}\phi\lf(\frac{x-y}{t}\r)a(y)\,dy\r|\\ \nonumber
&\le\lf[\sup_{t\in(0,7r_B)}+
\sup_{t\in(7r_B,\dz(x)/2)}\r]\lf|\int_{B_\boz}\frac{1}{t^n}\phi\lf(\frac{x-y}{t}\r)a(y)\,dy\r|\\ \nonumber
&=:{\rm I}_1+{\rm I}_2,
\end{align}
where $\dz(x):=\dist(x,\partial\boz)$ with $\partial\boz$ being the boundary of $\boz$.
From the fact that $|x-y|>7r_B$ for any $y\in B_\boz$ and $x\in(8B_\boz)^\complement$, we deduce that,
for any $t\in(0,7r_B)$, $y\in B_\boz$, and $x\in(8B_\boz)^\complement$, $\phi(\frac{x-y}{t})=0$,
which implies that ${\rm I}_1=0$. Moreover, it is easy to see that, when $\dz(x)\le14r_B$,
${\rm I}_2=0$. We assume that $\dz(x)>14r_B$. In this case, we find that, for any $y\in B_\boz$
and $x\in(8B_\boz)^\complement$ with $\dz(x)>14r_B$,
\begin{align}\label{eq4.6}
|x-y|\ge\dz(x)-\dz(y).
\end{align}
Since $4B\cap\partial\boz\neq\emptyset$, it follows that, for any $y\in B_\boz$, $\dz(y)\le 4r_B$,
which, together with \eqref{eq4.6} implies that, for any $y\in B_\boz$ and $x\in(8B_\boz)^\complement$
with $\dz(x)>14r_B$,
$$
|x-y|\ge\dz(x)-\dz(y)\ge\dz(x)-4r_B>\dz(x)-\dz(x)/2=\dz(x)/2>t.
$$
By this, we conclude that, for any $t\in[7r_B,\dz(x)/2)$,
$y\in B_\boz$, and $x\in(8B_\boz)^\complement$, $\phi(\frac{x-y}{t})=0$,
which further implies that ${\rm I}_2=0$. This, combined with \eqref{eq4.5} and ${\rm I}_1=0$,
concludes that, for any $x\in (8B_\boz)^\complement$, $a^+_\boz(x)=0$. Thus, $\supp(a^+_\boz)\subset 8B_\boz$.

Recall that the \emph{Hardy--Littlewood maximal operator} $\cm$ on $\rn$ is defined by setting,
for any $f\in L^1_{\loc}(\rn)$ and $x\in\rn$,
$$\cm(f)(x):=\sup_{B\ni x}\fint_{B}|f(y)|\,dy,$$
where the supremum is taken over all balls $B\subset\rn$ containing $x$.
Then, from $\supp(a^+_\boz)\subset 8B_\boz$, the fact that $a^+_\boz\ls\cm(a)$, the H\"older inequality,
and the boundedness of $\cm$ on $L^q(\rn)$ with $q\in(1,\fz]$, we deduce that
\begin{align*}
\|a\|_{H^p(\boz)}&=\lf\|a^+_\boz\r\|_{L^p(\boz)}=\lf\|a^+_\boz\r\|_{L^p(8B_\boz)}
\ls\lf\|\cm(a)\r\|_{L^q(8B_\boz)}|8B_\boz|^{1/p-1/q}\\ \nonumber
&\ls\|a\|_{L^q(B_\boz)}|B_\boz|^{1/p-1/q}
\ls|B_\boz|^{1/q-1/p}|B_\boz|^{1/p-1/q}\sim1.
\end{align*}
Therefore, \eqref{eq4.4} holds true.

Next, we assume that $f\in H^p(\boz)$. By Lemma \ref{l2.3}, we find that
there exist two sequences $\{\lz_j\}^\infty_{j=1}\subset \mathbb{C}$ and
$\{\kappa_j\}^\infty_{j=1}\subset \mathbb{C}$,  a sequence $\{a_{j}\}_{j=1}^\fz$
of type $(a)$ $(p,\,\fz,\,0)_\boz$-atoms, and a sequence $\{b_{j}\}_{j=1}^\fz$
of type $(b)$ $(p,\,\fz)_\boz$-atoms such that
\begin{align*}
f=\sum_{j=1}^\fz\lambda_{j}a_{j} +\sum_{j=1}^\fz\kappa_{j}b_{j}
\end{align*}
in $\cd'(\boz)$, and
\begin{align}\label{eq4.7}
\|f\|_{H^p(\boz)}\sim\lf(\sum_{j=1}^\fz|\lambda_{j}|^p\r)^{1/p}
+\lf(\sum_{j=1}^\fz|\kappa_{j}|^p\r)^{1/p}.
\end{align}

Let $b$ be a type $(b)$ $(p,\,\fz)_\boz$-atom supported in the ball $B:=B(x_B,r_B)\subset\boz$
with $x_B\in\boz$ and $r_B\in(0,\fz)$. From Lemma \ref{l2.1}(ii), it follows that there exists a
ball $\wz{B}(x_{\wz{B}},r_{\wz{B}})\subset\boz^\complement$ with $x_{\wz{B}}\in\boz^\complement$
and $r_{\wz{B}}\in(0,\fz)$ such that $r_{\wz{B}}\sim r_B$ and $\dist(B,\wz{B})\sim r_B$.
Assume that $B_0(x_{B_0},r_{B_0})\subset\rn$ with $x_{B_0}\in\rn$ and $r_{B_0}\in(0,\fz)$ is a ball
such that $B\cup\wz{B}\subset B_0$ and $r_{B_0}\sim r_B$. Let
\begin{align}\label{eq4.8}
\wz{b}:=b-\lf[\frac{1}{|\wz{B}|}\int_B b(x)\,dx\r]\mathbf{1}_{\wz{B}}.
\end{align}
Then $\wz{b}|_\boz=b$, $\supp(\wz{b})\subset B_0$, and $\int_\rn\wz{b}(x)\,dx=0$.
Moreover, it is easy to see that
$$
\lf\|\wz{b}\r\|_{L^\fz(\rn)}\le\|b\|_{L^\fz(B)}\lf[1+\frac{|B|}{|\wz{B}|}\r]\ls
\|b\|_{L^\fz(B)}\ls|B|^{-1/p}\sim|B_0|^{-1/p}.
$$
Therefore, $\wz{b}$ is a harmless constant multiple of a $(p,\,\fz,\,0)$-atom
supported in the ball $B_0$.

For any $j\in\nn$, let $\wz{b}_j$ be as in \eqref{eq4.8} with $b$ replaced by $b_j$. Define $$\wz{f}:=\sum_{j=1}^\fz\lz_ja_j+\sum_{j=1}^\fz\kappa_j\wz{b}_j.$$
Then $\wz{f}|_\boz=f$, $\wz{f}\in H^p(\rn)$, and
$$
\lf\|\wz{f}\r\|_{H^p(\rn)}\ls\lf(\sum_{j=1}^\fz|\lambda_{j}|^p\r)^{1/p}
+\lf(\sum_{j=1}^\fz|\kappa_{j}|^p\r)^{1/p},
$$
which, together with \eqref{eq4.7}, further implies that $f\in H^p_r(\boz)$
and $\|f\|_{H^p_r(\boz)}\ls\|f\|_{H^p(\boz)}$. This finishes the proof of \eqref{eq4.1}.

Let $p\in(\frac{n}{n+\dz_0},1]$, where $\dz_0\in(0,1]$ is as in Theorem \ref{t1.1}.
Now, we prove that $H^p_r(\boz)=H^p_{L_D}(\boz)$ with equivalent quasi-norms.
To this end, we first show that
\begin{equation}\label{eq4.9}
\lf[H^p_{r}(\boz)\cap L^2 (\boz)\r]\subset\lf[H^p_{L_D}(\boz)\cap L^2(\boz)\r].
\end{equation}
Let $f\in [H^p_{r}(\boz)\cap L^2(\boz)]$. By the definition of $H^p_{r}(\boz)$, we conclude that
there exists an $\wz{f}\in H^p(\rn)$ such that $\wz{f}\big|_{\boz}=f$ and
\begin{equation}\label{eq4.10}
\lf\|\wz{f}\r\|_{H^p(\rn)}\ls\|f\|_{H^p_r(\boz)}.
\end{equation}
Then there exist a sequence $\{a_j\}_{j\in\nn}$ of $(p,\fz,0)$-atoms, and $\{\lz_j\}_{j\in\nn}
\subset\cc$ such that
\begin{equation}\label{eq4.11}
\wz{f}=\sum_{j\in\nn}\lz_j a_j
\end{equation}
in both $\cs'(\rn)$ and $(H^p(\rn))^\ast$, and
\begin{equation}\label{eq4.12}
\lf\|\wz{f}\r\|_{H^p(\rn)}\sim\lf(\sum_{j\in\nn}|\lz_j|^p\r)^{1/p},
\end{equation}
where $(H^p(\rn))^\ast$ denotes the dual space of $H^p(\rn)$. Denote by $\{H^{L_D}_t\}_{t>0}$ the kernels
of the family $\{tL_De^{-tL_D}\}_{t>0}$ of operators. From the fact that, for any $t\in(0,\fz)$,
$$tL_De^{-tL_D}=2\lf(\frac{t}{2}L_De^{-t/2L_D}\r)e^{-t/2L_D},$$
we deduce that, for any $t\in(0,\fz)$ and $x,\,y\in\boz$,
\begin{equation}\label{eq4.13}
H^{L_D}_t(x,y)=2\int_\boz H^{L_D}_{t/2}(x,z)K^{L_D}_{t/2}(z,y)\,dz.
\end{equation}
Moreover, by \cite[Theorem 6.17]{o05} and \eqref{eq1.5}, we find that,
for any $t\in(0,\fz)$ and $x,\,y\in\boz$,
\begin{align}\label{eq4.14}
\lf|H^{L_D}_t(x,y)\r|\ls\frac{1}{t^{n/2}}\exp\lf\{-\frac{|x-y|^2}{ct}\r\},
\end{align}
which, combined with \eqref{eq4.13} and Theorem \ref{t1.1}, implies that, for any given $\dz\in(0,\dz_0)$,
and any $t\in(0,\fz)$ and $x,\,y,\,z\in\boz$ with $|y-z|\le\sqrt{t}/2$,
\begin{equation}\label{eq4.15}
\lf|H^{L_D}_t(x,y)-H^{L_D}_{t}(x,z)\r|\ls\frac{1}{t^{n/2}}
\lf[\frac{|y-z|}{\sqrt{t}}\r]^{\dz}\exp\lf\{-\frac{|x-y|^2}{ct}\r\}.
\end{equation}
For any $t\in(0,\fz)$ and $x\in\boz$, denote by $\widetilde{H}^{L_D}_{t}(x,\cdot)$ the zero extension of
$H^{L_D}_{t}(x,\cdot)$ from $\Omega$ to $\mathbb{R}^n$.
Let $\gz\in(0,\fz)$.  Then the \emph{Campanato space} $\mathcal{L}^{\gz,\,1,\,0}(\rn)$ is defined by setting
$$
\mathcal{L}^{\gz,\,1,\,0}(\rn):=\lf\{g\in L^1_\loc(\rn):\ \|g\|_{\mathcal{L}^{\gz,\,1,\,0}(\rn)}<\fz\r\},
$$
where
$$\|g\|_{\mathcal{L}^{\gz,\,1,\,0}(\rn)}:=\sup_{B\subset\rn}\lf[|B|^{-\gz}\fint_B|g(x)-g_B|\,dx\r]
$$
with the supremum taken over all balls $B\subset\rn$. Then, similarly to the proof of
\cite[Lemma 3.9]{cfyy16}, via applying Theorem \ref{t1.1}
and Lemma \ref{l2.1},  we find that,
for any given $t\in(0,\fz)$, $x\in\boz$, and $\dz\in(0,\dz_0)$, $\widetilde{H}^{L_D}_{t}(x,\cdot)
\in\mathcal{L}^{\dz/n,\,1,\,0}(\rn)$. From this, the fact that $\mathcal{L}^{\frac{1}{p}-1,\,1,\,0}(\rn)$
is the dual space of $H^p(\rn)$ (see, for instance, \cite{st93}), $p\in(\frac{n}{n+\dz_0},1]$,
and \eqref{eq4.11}, it follows that, for any $t\in(0,\infty)$ and $x\in\Omega$,
\begin{align*}
\int_\Omega H^{L_D}_{t^2}(x,y)f(y)\,dy&=\int_\rn \widetilde{H}^{L_D}_{t^2}(x,y)\widetilde{f}(y)\,dy
=\sum_{j\in\nn}\lambda_j \int_\rn \widetilde{H}^{L_D}_{t^2}(x,y)a_j(y)\,dy \\
&=\sum_{j\in\nn}\lambda_j \int_\Omega H^{L_D}_{t^2}(x,y)a_j(y)\,dy,
\end{align*}
which, together with the boundedness of $S_{L_D}$ on $L^2(\Omega)$ (see, for instance,
\cite[Theorem 2.13]{bckyy13b}), further implies that, for almost every $x\in\Omega$,
\begin{equation}\label{eq4.16}
S_{L_D}(f)(x)\le \sum_{j\in\nn}|\lambda_j|S_{L_D}(a_j)(x).
\end{equation}
To prove \eqref{eq4.9}, we only need to show that, for any $(p,\,\fz,\,0)$-atom $a$ supported in the
ball $B:=B(x_B,r_B)$ with $x_B\in\rn$ and $r_B\in(0,\fz)$,
\begin{equation}\label{eq4.17}
\int_\boz\lf[S_{L_D}(a)(x)\r]^p\,dx\ls1.
\end{equation}
Indeed, if \eqref{eq4.17} holds true, then, by \eqref{eq4.16}, \eqref{eq4.17}, \eqref{eq4.10},
\eqref{eq4.11}, and \eqref{eq4.12}, we conclude that $f\in H^p_{L_D}(\boz)$ and
\begin{align*}
\|f\|_{H^p_{L_D}(\boz)}&\ls\lf\{\sum_{j=1}|\lz_j|^p\int_\boz\lf[S_{L_D}(a_j)(x)\r]^p\,dx\r\}^{1/p}
\ls\lf(\sum_{j\in\nn}|\lz_j|^p\r)^{1/p}\\ \nonumber
&\sim\lf\|\wz{f}\r\|_{H^p(\rn)}\ls\|f\|_{H^p_r(\boz)}.
\end{align*}
Thus, \eqref{eq4.9} holds true.

Next, we show \eqref{eq4.17} by considering the following three cases on the ball $B$.

\emph{Case 1)} $B\cap\Omega=\emptyset$. In this case, we find that, for any $x\in\Omega$,
$$t^2L_De^{-t^2L_D}(a)(x)=\int_{B\cap\Omega}H^{L_D}_{t^2}(x,y)a(y)dy=0,$$
which implies that $S_{L_D}(a)=0$. From this, we deduce that \eqref{eq4.17} holds true in this case.

\emph{Case 2)} $B\subset\Omega$. In this case, by the fact that $S_{L_\Omega}$ is bounded on $L^q(\Omega)$
with $q\in(1,\infty)$ (see, for instance, \cite[Theorem 2.13]{bckyy13b}), we conclude that
\begin{align}\label{eq4.18}
\int_{4B_\boz} \lf[S_{L_D}(a)(x)\r]^p\,dx\ls\|S_{L_D}(a)\|^p_{L^q(4B_\boz)}|4B_\boz|^{1-\frac{p}{q}}
\ls\|a\|^p_{L^q(4B_\boz)}|B|^{1-\frac{p}{q}}\ls1.
\end{align}
Furthermore, for any $x\in(4B)^\complement\cap\Omega$, we have
\begin{align}\label{eq4.19}
\lf[S_{L_D}(a)(x)\r]^2
&=\int^{r_{B}}_0\int_{|y-x|<t}
\lf|t^2L_D e^{-t^2L_D}(a)(y)\r|^2\frac{dy\,dt}{|B_\boz(x,t)|t}
+\int^\infty_{r_{B}}\int_{|y-x|<t}\cdots\\ \nonumber
&=:E+F.
\end{align}
It is easy to see that, for any $x\in(4B)^\complement\cap\Omega$, and any $t\in(0,\infty)$ and $y,\,z\in\Omega$
satisfying $|x-y|<t$ and $|z-x_{B}|<r_{B}$,
\begin{align}\label{eq4.20}
t+|y-z|\ge t+\lf|x-x_{B}\r|-|x-y|-\lf|x_{B}-z\r|>\lf|x-x_{B}\r|-r_{B}\ge \frac{3|x-x_{B}|}{4}.
\end{align}
Take $\dz_1\in(0,\dz)$ such that $p>\frac{n}{n+\dz_1}$.
From \eqref{eq4.14}, \eqref{eq4.20}, and the H\"{o}lder inequality, it follows that
\begin{align}\label{eq4.21}
E&=\int^{r_{B}}_0\int_{|y-x|<t}
\lf|\int_{B\cap\Omega}H^{L_D}_{t^2}(y,z)a(z)\,dz\r|^2\frac{dy\,dt}{|B_\boz(x,t)|t}\\ \nonumber
&\ls \int^{r_{B}}_0\int_{|y-x|<t}
\lf|\int_{B}\frac1{t^n}e^{-\frac{|y-z|^2}{ct^2}}a(z)\,dz\r|^2\frac{dy\,dt}{t^{n+1}}\\ \nonumber
&\ls \int^{r_{B}}_0\int_{|y-x|<t}
\lf[\int_{B}\frac{t^{\delta_1}}{|y-z|^{n+\delta_1}} \lf|a(z)\r|\,dz\r]^2\frac{dy\,dt}{t^{n+1}}\\ \nonumber
&\ls \frac1{|x-x_{B}|^{2(n+\delta_1)}}\lf[\int_{B}\lf|a(z)\r|\,dz\r]^2
\int^{r_{B}}_0 t^{2\delta_1-1}\,dt\\ \nonumber
&\ls \lf[\frac{r_{B}}{|x-x_{B}|}\r]^{2(n+\delta_1)}|B|^{-2/p}.
\end{align}
Next, we deal with the term $F$. By $\int_B a(x)\,dx=0$, \eqref{eq4.15}, \eqref{eq4.20}, $t>r_{B}$,
and the H\"{o}lder inequality, we find that, for any $y\in\Omega$,
\begin{align*}
&\lf|t^2L_D e^{-t^2L_D}(a)(y)\r|\\ \nonumber
&\quad\leq\int_{B} \lf|H^{L_D}_{t^2}(y,z)-K^{L_D}_{t^2}(y,x_{B})\r| \lf|a(z)\r|\,dz\\ \nonumber
&\quad\lesssim \int_{B} \frac1{t^n} \lf[\frac{|z-x_{B}|}{t}\r]^{\delta}
e^{-\frac{|y-z|^2}{ct^2}} \lf|a(z)\r|\,dz\\ \nonumber
&\quad\lesssim \int_{B} \lf[\frac{r_{B}}{t}\r]^{\delta}
\frac{t^{\delta_1}}{|y-z|^{n+\delta_1}} \lf|a(z)\r|\,dz\\ \nonumber
&\quad\ls\int_{B} \frac{r_{B}^{\delta}}
{t^{\delta-\dz_1} |x-x_{B}|^{n+\delta_1}} \lf|a(z)\r|\,dz
\lesssim \frac{r_{B}^{n+\delta}}{t^{\delta-\dz_1} |x-x_{B}|^{n+\delta_1}}
|B|^{-1/p}.
\end{align*}
From this, we deduce that
\begin{align*}
F&\lesssim \frac{r_{B}^{2(n+\delta)}|B|^{-2/p}}
{|x-x_{B}|^{2(n+\delta_1)}}
\int^\infty_{r_{B}}\int_{|y-x|<t}t^{-2(\dz-\delta_1)} \,\frac{dy\,dt}{|B_\boz(x,t)|t}\\
&\ls\frac{r_{B}^{2(n+\delta)}|B|^{-2/p}}
{|x-x_{B}|^{2(n+\delta_1)}}\int^\infty_{r_{B}}t^{-2(\dz-\delta_1)-1}\,dt
\ls \lf[\frac{r_{B}}{|x-x_{B}|}\r]^{2(n+\delta_1)}
|B|^{-2/p},
\end{align*}
which, combined with \eqref{eq4.19} and \eqref{eq4.21}, further implies that,
for any $x\in(4B)^\complement\cap\Omega$,
\begin{align*}
S_{L_D}(a)(x)\ls\lf[\frac{r_{B}}{|x-x_{B}|}\r]^{n+\delta_1}|B|^{-1/p}.
\end{align*}
By this and $p>\frac{n}{n+\dz_1}$, we obtain
\begin{align*}
\int_{\boz\setminus(4B_\boz)}\lf[S_{L_D}(a)(x)\r]^p\,dx&=\sum_{j=2}^\fz
\int_{S_j(B_\boz)}\lf[S_{L_D}(a)(x)\r]^p\,dx\\ \nonumber
&\ls\sum_{j=2}^\fz\int_{S_j(B_\boz)}2^{-(n+\dz_1)jp}|B|^{-1}\,dx
\ls\sum_{j=2}^\fz2^{-[(n+\dz_1)p-n]j}\ls1,
\end{align*}
which, together with \eqref{eq4.18}, further implies that \eqref{eq4.17} holds true in this case.

\emph{Case 3)} $B\cap\Omega\neq\emptyset$. In this case, take $y_B\in B\cap\partial\Omega$.
From the fact that, for any given $t\in(0,\fz)$ and $x\in\boz$, $H^{L_D}_{t}(x,\cdot)\in W^{1,2}_0(\boz)$,
it follows that $H^{L_\Omega}_{t}(x,y_B)=0$, which implies that, for any given $t\in(0,\fz)$ and any $x\in\boz$,
$$t^2L_De^{-t^2L_D}(a)(x)=\int_{B\cap\Omega}
\lf[H^{L_D}_{t^2}(x,y)-H^{L_D}_{t^2}(x,y_B)\r]a(y)dy.$$
The remaining estimations are similar to those of Case 2), we omit the details here.
This finishes the proof of \eqref{eq4.17}, and hence of \eqref{eq4.9}.

Now, we prove that
\begin{equation}\label{eq4.22}
\lf[H^p_{L_D}(\boz)\cap L^2(\boz)\r]\subset\lf[H^p_{r}(\boz)\cap L^2 (\boz)\r].
\end{equation}
Let $f\in[H^{p}_{L_D}(\Omega)\cap L^2(\Omega)]$. Then, by the $H^\infty$-functional
calculus associated with $L_\Omega$ (see, for instance, \cite[p.\,8]{at98}), we find that
\begin{equation}\label{eq4.23}
f=8\int^\infty_0 \lf(t^2L_De^{-t^2L_D}\r) \lf(t^2L_De^{-t^2L_D}\r) (f)\frac{dt}{t}
\end{equation}
in $L^2(\Omega)$. From the assumption $f\in[H^{p}_{L_D}(\Omega)\cap L^2(\Omega)]$, we deduce that
$S_{L_D}(f)\in L^{p}(\Omega)\cap L^2(\Omega)$, which implies that
\begin{equation}\label{eq4.24}
t^2L_De^{-t^2L_D}(f)\in T^{p}(\boz\times(0,\fz))\cap T^2(\boz\times(0,\fz))
\end{equation}
and
\begin{equation}\label{eq4.25}
\|f\|_{H^{p}_{L_D}(\Omega)}=\lf\|t^2L_De^{-t^2L_D}(f)\r\|_{T^{p}(\boz\times(0,\fz))}.
\end{equation}
By \eqref{eq4.24} and Lemma \ref{l4.1}, we conclude that there exist $\{\lambda_j\}_{j\in\nn}\subset\mathbb{C}$
and a sequence $\{a_j\}_{j\in\nn}$ of $(p,\,2)$-atoms associated, respectively, with the balls
$\{B_j\cap\boz\}_{j\in\nn}$ of $\boz$ such that, for almost every $(x,t)\in \Omega \times(0,\infty)$,
\begin{equation}\label{eq4.26}
t^2L_D e^{-t^2L_D}(f)= \sum_{j\in\nn}\lambda_ja_j
\end{equation}
and
\begin{equation}\label{eq4.27}
\lf\|t^2L_D e^{-t^2L_D}(f)\r\|_{T^{p}(\boz\times(0,\fz))}\sim \lf(\sum_{j=1}^\fz|\lz_j|^p\r)^{1/p}.
\end{equation}
For any $j\in\nn$, let $\alpha_j:=8\int^\infty_0 t^2L_\Omega e^{-t^2L_\Omega}(a_j(\cdot,t))\frac{dt}{t}$.
Then, from \eqref{eq4.23} and \eqref{eq4.26}, it follows that
\begin{equation}\label{eq4.28}
f=\sum_{j\in\nn}\lz_j \az_j
\end{equation}
in $L^2 (\boz)$.

For any $(p,\,2)$-atom $a$ associated with the ball $B_\boz$, let
\begin{align*}
\az:=8\int_0^{\fz}t^2 L_De^{-t^2 L_D}(a)\frac{dt}{t}.
\end{align*}
To show \eqref{eq4.22}, it suffices to prove that there exist a function $\wz{\az}$ on $\rn$ such that
\begin{equation}\label{eq4.29}
\wz{\az}|_{\boz}=\az,
\end{equation}
and a sequence $\{\kappa_i\}_i\subset\cc$ and a sequence $\{b_i\}_i$ of $(p,\,2,\,0)$-atoms such
that $\wz{\az}=\sum_i \kappa_ib_i$ in $L^2(\rn)$, and
\begin{equation}\label{eq4.30}
\sum_i |\kappa_i|^p\ls1.
\end{equation}
Indeed, if \eqref{eq4.29} and \eqref{eq4.30} hold true, then, by \eqref{eq4.28},
we find that, for any $j\in\nn$, there exists a function $\wz{\az}_j$ on $\rn$ such that
$\wz{\az}_j|_{\boz}=\az_j$. Let
$$\wz{f}:=\sum_{j\in\nn} \lz_j \wz{\az}_j.$$
Then $\wz{f}|_{\boz}=f$. Furthermore, from \eqref{eq4.30}, we deduce that there exist a sequence
$\{\kappa_{j,\,i}\}_{j\in\nn,\,i}\subset\cc$ and a sequence $\{b_{j,\,i}\}_{j\in\nn,\,i}$ of
$(p,\,2,\,0)$-atoms such that
$$\wz{f}=\sum_{j\in\nn} \sum_{i} \lz_j\kappa_{j,\,i}b_{j,\,i}$$
and
\begin{equation*}
\sum_{j\in\nn,\,i}\lf|\lz_j\kappa_{j,\,i}\r|^p\ls\sum_{j\in\nn}\lf|\lz_j\r|^p.
\end{equation*}
By this, Lemma \ref{l2.2}, \eqref{eq4.25}, and \eqref{eq4.27}, we conclude that $\wz{f}\in H^p(\rn)$ and
\begin{equation*}
\lf\|\wz{f}\r\|_{H^p(\rn)}\sim\lf\|\wz{f}\r\|_{H^{p,\,2,\,0}_{\mathrm{at}}(\rn)}
\ls\lf\|f\r\|_{H^p_{L_D}(\boz)}.
\end{equation*}
Thus, $f\in H^p_r(\boz)$ and
$$\|f\|_{H^p_r(\boz)}\ls\|f\|_{H^p_{L_D}(\boz)},$$
which, combined with the arbitrariness of $f\in H^p_{L_D}(\boz)\cap L^2(\boz)$,
implies that \eqref{eq4.22} holds true.

Next, we prove \eqref{eq4.29} and \eqref{eq4.30} by considering the following two cases on the ball
$B:=B(x_B,r_B)$, with $x_B\in\rn$ and $r_B\in(0,\fz)$, which appears in the support of $a$.

\emph{Case 1)} $8B\cap\boz^\complement\neq\emptyset$.
In this case, let $\wz{S}_0(B_\boz):=\cup_{j=0}^2S_j(B_\boz)$ and
$$J_{\boz}:=\lf\{k\in\nn:\ k\ge3,\,|S_k(B_\boz)|>0\r\}.$$
Moreover, assume that $\mathbf{1}_0:=\mathbf{1}_{\wz{S}_0(B_\boz)}$,
$m_0:=\int_{\wz{S}_0(B_\boz)}\az(x)\,dx$,
and, for any $k\in J_{\boz}$, $\mathbf{1}_k:=\mathbf{1}_{S_k
(B_\boz)}$ and $m_k:=\int_{S_k(B_\boz)}\az(x)\,dx.$
Then
\begin{equation*}
\az=\az\mathbf{1}_0+\sum_{k\in J_{\boz}}\az\mathbf{1}_k
\end{equation*}
holds true almost everywhere and also in $L^2(\boz)$. From the assumption $8B\cap\boz^\complement\neq\emptyset$,
it follows that there exists a $y_B\in\partial\boz$ such that $B(y_B,16r_B)\supset 8B$. By this and
Lemma \ref{l2.1}(iii), we find that there exists a ball $\wz{B}\subset\boz^\complement$ such that
$r_{\wz{B}}\sim r_B$ and $\dist(B_\boz,\wz{B})\sim r_B$. Therefore, there exists a ball $B^{\ast}_0\subset\rn$
such that $(B\cup\wz{B})\subset B^{\ast}_0$ and
\begin{equation}\label{eq4.31}
r_{B^{\ast}_0}\sim r_B.
\end{equation}

Let
$$b_0:= \az \mathbf{1}_0
-\lf[\frac{1}{|\wz{B}|}\int_{\wz{S}_{(B_\boz)}}
\az(x)\,dx\r]\mathbf{1}_{\wz{B}}.$$
Then $\int_{\rn}b_0(x)\,dx=0$ and $\supp(b_0)\subset B^{\ast}_0$.
Moreover, it is known that $\|\alpha\|_{L^2(\Omega)}
\ls\|a\|_{T^2(\Omega\times(0,\infty))}$ (see, for instance, \cite[Proposition 4.5]{bckyy13b}),
which further implies that
\begin{equation}\label{eq4.32}
\lf\|\alpha\r\|_{L^2(\Omega)}
\ls\lf\|a\r\|_{T^2(\Omega\times(0,\infty))}
\ls\lf|B_\Omega\r|^{1/2-1/p}.
\end{equation}
From this, $r_{\wz{B}}\sim r_B$, Lemma \ref{l2.1}(i), and \eqref{eq4.31}, we deduce that
\begin{align*}
\|b_0\|_{L^2(\rn)}&\le\lf\|\az \r\|_{L^2(\boz)}
+\lf\|\az \r\|_{L^2(\boz)}\lf|\wz{B}\r|^{-1/2}|B|^{1/2}\\ \nonumber
&\ls\lf\|\az\r\|_{L^2(\boz)}\ls|B_\boz|^{1/2-1/p}
\sim\lf|B^\ast_0\r|^{1/2-1/p}.
\end{align*}
Therefore, $b_0$ is a harmless constant multiple of a $(p,\,2,\,0)$-atom.

For any $k\in J_\boz$, by the definition of $J_\boz$, we find that $2^{k}r_B\le\diam(\boz)$ and
$2^kB\cap\partial\boz\neq\emptyset$. By this, we conclude that there exists a $y_B\in\partial\boz$
such that $B(y_B,2^{k+1}r_B)\supset 2^kB$, which, together with Lemma \ref{l2.1}(iii), further implies
that there exists a ball $\wz{B}_k\subset\boz^\complement$ such that $r_{\wz{B}_k}\sim2^kr_B$
and $\dist(S_k(B_\boz),\wz{B}_k)\ls 2^{k}r_B$. Let the ball $B^\ast_k\subset\rn$ satisfy that $\wz{B}_k\cup
S_k(B_\boz)\subset B^\ast_k$ and $r_{B^\ast_k}\sim r_{2^kB}$. Assume that, for any $k\in J_\boz$,
$$
b_k:=\az \mathbf{1}_k-\lf[\frac{1}{|\wz{B}_k|}\int_{S_k(B_\boz)}\az(x)\,dx\r]\mathbf{1}_{\wz{B}_k}.
$$
Then, for any $k\in J_\boz$, $\supp(b_k)\subset B^\ast_k$ and $\int_\rn b_k(x)\,dx=0$. Furthermore,
from the H\"older inequality and Lemma \ref{l2.1}(i), it follows that
\begin{align}\label{eq4.33}
\|b_k\|_{L^2(\rn)}&\le\lf\|\az\r\|_{L^2(S_k(B_\boz))}
+\lf\|\az \r\|_{L^2(S_k(B_\boz))}\lf|\wz{B}_k\r|^{-1/2}\lf|S_k(B_\boz)\r|^{1/2}\\ \nonumber
&\ls\lf\|\az\r\|_{L^2(S_k(B_\boz))}.
\end{align}
Moreover, we have, for any $k\in J_\boz$ with $k\ge3$, and any $x\in S_k(B_\boz)$,
\begin{align}\label{eq4.34}
\lf|\az(x)\r|&\ls\int_0^{r_B}\int_{B_\boz}\lf|H^{L_D}_{t^2}(x,y)\r|
|a(y,t)|\,\frac{dy\,dt}{t}\\ \nonumber
&\ls\int_0^{r_B}\int_{B_\boz}\frac{1}{t^n}e^{-\frac{|x-y|^2}{ct^2}}
|a(y,t)|\frac{dy\,dt}{t} \\ \nonumber
&\ls\|a\|_{T^2(\boz\times(0,\fz))} \lf\{\int_0^{r_B}\int_{B_\boz}\frac{t^{2}}{|x-y|
^{2(n+1)}}\frac{dy\,dt}{t}\r\}^{1/2}\\ \nonumber
&\ls|x-x_B|^{-(n+1)}r_B |B_\boz|^{1/2}\|a\|_{T^2(\boz\times(0,\fz))}\ls2^{-k(n+1)}|B_\boz|^{-1/p},
\end{align}
which, together with \eqref{eq4.33}, further implies that
\begin{align}\label{eq4.35}
\|b_k\|_{L^2(\rn)}&\ls2^{-(n/2+1)k}|B_\boz|^{1/2-1/p}
\ls2^{-s_0k}\lf|B^\ast_k\r|^{1/2-1/p},
\end{align}
where $s_0:=n(1+\frac{1}{n}-\frac1p)$.
By this, $\supp(b_k)\subset B^\ast_k$, and $\int_\rn b_k(x)\,dx=0$, we conclude that
$2^{s_0k}b_k$ is a harmless constant multiple of a $(p,\,2,\,0)$-atom.
Let
\begin{equation*}
\wz{\az}:=b_0+\sum_{k\in J_\boz}2^{-s_0k}(2^{s_0k}b_k).
\end{equation*}
Then $\wz{\az}\big|_{\boz}=\az$, $\wz{\az}\in H^p(\rn)$, and $\|\wz{\az}\|_{H^p(\rn)}\ls1$.

\emph{Case 2)} $8B\subset\boz$.
In this case, let $k_0\in\nn$ be such
that $2^{k_0} B\subset\boz$ but $(2^{k_0
+1}B)\cap\paz\boz\neq\emptyset$. Then $k_0\ge3$. Let
$$J_{\boz,\,k_0}:=\lf\{k\in\nn:\  k\ge k_0 +1,\,|S_k(B_\boz)|>0\r\}.$$
For any $k\in\zz_+$, let $\mathbf{1}_k:=\mathbf{1}_{S_k(B_\boz)}$,
$$m_k :=\int_{S_k(B_\boz)}\az(x)\,dx,$$
$M_k:=\az \mathbf{1}_k-
m_k|S_k(B_\boz)|^{-1}\mathbf{1}_k$, and $\wz{M}_k:=
\az \mathbf{1}_k$. Then
\begin{equation*}
\az =\sum_{k=0}^{k_0}M_k +\sum_{k\in
J_{\boz,k_0}}\wz{M}_k+\sum_{k=0}^{k_0}m_k|S_k(B_\boz)|^{-1}\mathbf{1}_k.
\end{equation*}
For any $k\in\{0,\,\ldots,\,k_0\}$, from the definition of $M_k$, we deduce that
$$\int_{\rn}M_k(x)\,dx=0\ \quad \text{and}\ \quad \supp(M_k)\subset2^{k+1}B.$$
Moreover, if $k=0$, by the H\"older inequality and \eqref{eq4.32}, we find that
\begin{equation}\label{eq4.36}
\|M_0\|_{L^2(\rn)}\ls\lf\|\az \r\|_{L^2(\boz)}\ls\|\az\|_{L^2(\boz)}\ls|B|^{1/2-1/p}
\end{equation}
and, if $k\in\{1,\,\ldots,\,k_0\}$, from \eqref{eq4.34}, it follows that
\begin{equation}\label{eq4.37}
\|M_k\|_{L^2(\rn)}\ls\lf\|\az\r\|_{L^2(S_k(B_\boz))}\ls2^{-s_0k}\lf|2^{k+1}B\r|^{1/2-1/p},
\end{equation}
where $s_0$ is as in \eqref{eq4.35}. Thus, for any $k\in\{0,\,\ldots,\,k_0\}$,
$2^{s_0k}M_k$ is a harmless constant multiple of a $(p,\,2,\,0)$-atom.

For any $k\in J_{\boz,\,k_0}$, by the definitions of both $k_0$ and $J_{\boz,\,k_0}$,
we have $2^{k}r_B\le\diam(\boz)$ and $2^kB\cap\partial\boz\neq\emptyset$. By this, we conclude that there
exists a $y_B\in\partial\boz$ such that $B(y_B,2^{k+1}r_B)\supset 2^kB$,
which, combined with Lemma \ref{l2.1}(iii), implies that there exists a ball
$\wz{B}_k\subset\boz^\complement$ such that $r_{\wz{B}_k}\sim 2^kr_{B}$
and $\dist(S_k(B_\boz),\wz{B}_k)\ls 2^kr_{B}$. Then there exists a ball $B^\ast_k$ such that
$\wz{B}_k\cup S_k(B_\boz)\subset B^\ast_k$ and $r_{B^\ast_k}\sim 2^kr_{B}$.
Let
$$a_k:=\az \mathbf{1}_{S_k(B_\boz)}-
\lf[\frac{1}{|\wz{B}_k|}\int_{S_k(B_\boz)}
\az(x)\,dx\r]\mathbf{1}_{\wz{B}_k}.
$$
Then $\supp(a_k)\subset B^\ast_k$ and $\int_\rn a_k(x)\,dx=0$.
Moreover, from \eqref{eq4.34} and Lemma \ref{l2.1}(i), we deduce that
\begin{equation}\label{eq4.38}
\|a_k\|_{L^2(\rn)}\ls\lf\|\az\r\|_{L^2(S_k(B_\boz))}\ls2^{-s_0k}\lf|B^\ast_k\r|^{1/2-1/p}.
\end{equation}
Thus, for any $k\in J_{\boz,\,k_0}$, $2^{s_0k}a_{k}$ is a harmless constant
multiple of a $(p,\,2,\,0)$-atom. For any $j\in\{0,\,\ldots,\,k_0\}$, let $N_j:=\sum_{k=j}^{k_0}m_k$.
It is easy to see that
\begin{equation*}
\sum_{k=0}^{k_0}m_k|S_k(B_\boz)|^{-1}\mathbf{1}_k=\sum_{k=1}^{k_0}\lf[|S_k(B_\boz)|^{-1}\mathbf{1}_k
-|S_{k-1}(B_\boz)|^{-1}\mathbf{1}_{k-1}\r]N_k +N_0|2B|^{-1}\mathbf{1}_0.
\end{equation*}
For any $k\in\{1,\,\ldots,\,k_0\}$, by
$$\lf|\lf|S_k(B_\boz)\r|^{-1}\mathbf{1}_k-\lf|S_{k-1}(B_\boz)\r|^{-1}\mathbf{1}_{k-1}\r|\ls\lf|2^k B\r|^{-1},$$
the H\"older inequality, and \eqref{eq4.34}, we find that
\begin{align}\label{eq4.39}
&\lf\|\lf[|S_k(B_\boz)|^{-1}\mathbf{1}_k-|S_{k-1}(B_\boz)|^{-1}\mathbf{1}_{k-1}\r]N_{k}
\r\|_{L^2(\rn)}\\ \nonumber
&\hs\ls|2^kB|^{-1/2}|N_k|\ls|2^kB|^{-1/2}\lf[\sum_{j=k}^{k_0}\lf\|\az \r\|_{L^2(S_j(B_\boz))}
|S_j(B_\boz)|^{1/2}\r]\\ \nonumber
&\hs\ls|2^kB|^{-1/2}\lf[\sum_{j=k}^{k_0}2^{-(n+1)j}|B|^{1/2-1/p}
\lf|2^jB\r|^{1/2}\r]\ls2^{-s_0k}\lf|2^kB\r|^{1/2-1/p},
\end{align}
where $s_0$ is as in \eqref{eq4.35}, which, together with
$$\int_{\rn}\lf[|S_k(B_\boz)|^{-1}\mathbf{1}_k(x)-|S_{k-1}(B_\boz)|^{-1}
\mathbf{1}_{k-1}(x)\r]\,dx=0$$
and $\supp(|S_k(B_\boz)|^{-1}\mathbf{1}_k-|S_{k-1}(B_\boz)|^{-1}\mathbf{1}_{k-1})
\subset2^{k}B$, implies that, for any $k\in\{1,\,\ldots,\,k_0\}$,
the function $2^{s_0k}[|S_k(B_\boz)|^{-1}\mathbf{1}_k-|S_{k-1}(B_\boz)|^{-1}\mathbf{1}_{k-1}]N_k$
is a harmless constant multiple of a $(p,\,2,\,0)$-atom.

Finally, we deal with $N_0|2B|^{-1}\mathbf{1}_0$. From $2^{k_0-1}r_0<\dist(x_0,\partial\boz)\le2^{k_0}r_0,$
it follows that there exist a positive integer $K$ and a sequence $\{B_{0,\,i}\}_{i=1}^{K}$ of balls such that
\vspace{-0.5em}
\begin{enumerate}
  \item[(i)] $K\sim2^{k_0}$;
  \vspace{-0.5em}
  \item[(ii)] for any $i\in\{1,\,\ldots,\,K\}$, $r_{B_{0,\,i}}=2r_0$ and
$B_{0,\,i}\subset\boz$;
\vspace{-0.5em}
  \item[(iii)] for any $i\in\{1,\,\ldots,\,K-1\}$, $B_{0,\,i}\cap
B_{0,\,i+1}\neq\emptyset$ and
$\dist(B_{0,\,i},\partial\boz)\ge\dist(B_{0,\,i+1},\partial\boz)$;
\vspace{-0.5em}
  \item[(iv)] $2B_{0,\,K}\cap\paz\boz\neq\emptyset$.
\end{enumerate}
\vspace{-0.5em}
\noindent By Lemma \ref{l2.1}(ii), we conclude that there exists a ball $B_{0,\,K+1}\subset\boz^\complement$
such that $r_{B_{0,\,K+1}}\sim r_0$ and $\dist(B_{0,\,K},B_{0,\,K+1})\sim r_0$. Let
$$a_{0,\,1}:=N_0|2B|^{-1}\mathbf{1}_0-N_0|B_{0,\,1}|^{-1}
\mathbf{1}_{B_{0,\,1}}$$
and
$$a_{0,\,i}:=N_0|B_{0,\,i-1}|^{-1}\mathbf{1}_{B_{0,\,i-1}}-
N_0|B_{0,\,i}|^{-1}\mathbf{1}_{B_{0,\,i}}$$
with $i\in\{2,\,\ldots,\,K+1\}$. Obviously, for any $i\in\{1,\,\ldots,\,K+1\}$, from the definition
of $a_{0,\,i}$, we deduce that $\int_{\rn}a_{0,\,i}(x)\,dx=0$ and there exists a ball
$B^{\ast}_{0,\,i}\subset\rn$ such that $\supp(a_{0,\,i})\subset B^{\ast}_{0,\,i}$ and
\begin{equation}\label{eq4.40}
r_{B^{\ast}_{0,\,i}}\sim r_B.
\end{equation}
Moreover, similarly to the estimation of \cite[(3.66)]{yy13}, we have
\begin{align}\label{eq4.41}
|N_0|\ls2^{-\frac{n+1}{n}k_0}|B|^{1-1/p}.
\end{align}

For any $i\in\{1,\,\ldots,\,K+1\}$, by the definition of
$a_{0,\,i}$, \eqref{eq4.40}, and \eqref{eq4.41}, we conclude that
\begin{equation}\label{eq4.42}
\lf\|a_{0,\,i}\r\|_{L^2(\rn)}\ls|N_0||B|^{-1/2}\ls2^{-\frac{n+1}{n}k_0}
|B|^{1/2-1/p}\sim2^{-\frac{n+1}{n}k_0}\lf|B^{\ast}_{0,\,i}\r|^{1/2-1/p},
\end{equation}
which, combined with the facts that $\int_{\rn}a_{0,\,i}(x)\,dx=0$
and $\supp(a_{0,\,i})\subset B^{\ast}_{0,\,i}$, further implies that
$2^{\frac{n+1}{n}k_0}a_{0,\,i}$ is a harmless constant multiple of a $(p,\,2,\,0)$-atom.
Let
\begin{equation*}
\wz{\az}:= \sum_{k=1}^{k_0}M_k +\sum_{k\in
J_{\boz,\,k_0}}a_k+\sum_{k=1}^{k_0}\lf[|S_k(B_\boz)|^{-1}\mathbf{1}_k
-|S_{k-1}(B_\boz)|^{-1}\mathbf{1}_{k-1}\r]N_k+\sum_{i=1}^{K+1}a_{0,\,i}.
\end{equation*}
It is easy to find that $\wz{\az}|_{\boz}=\az$.
Moreover, from \eqref{eq4.36}, \eqref{eq4.37}, \eqref{eq4.38}, \eqref{eq4.39}, and
\eqref{eq4.42}, it follows that $\wz{\az}$ has the following atomic decomposition
\begin{align*}
\wz{\az}&=\sum_{k=1}^{k_0}2^{-s_0k}\lf(2^{s_0k}M_k\r) +\sum_{k\in
J_{\boz,\,k_0}}2^{-s_0k}\lf(2^{s_0k}a_k\r)\\
&\quad+\sum_{k=1}^{k_0}2^{-s_0k}\lf\{2^{s_0k}\lf[|S_k(B_\boz)|^{-1}\mathbf{1}_k
-|S_{k-1}(B_\boz)|^{-1}\mathbf{1}_{k-1}\r]N_k\r\}\\
&\quad+\sum_{i=1}^{K+1}2^{-\frac{n+1}{n}k_0}\lf[2^{\frac{n+1}{n}k_0}a_{0,\,i}\r]
\end{align*}
and
\begin{align*}
&\sum_{k=1}^{k_0}2^{-s_0kp}+\sum_{k\in J_{\boz,\,k_0}}2^{-s_0kp}+\sum_{k=1}^{k_0}2^{-s_0kp}
+\sum_{i=1}^{K+1}2^{-\frac{n+1}{n}k_0p}\\ \nonumber
&\quad\ls\sum_{k=1}^{\fz}2^{-s_0kp}+2^{k_0}2^{-\frac{n+1}{n}k_0p}\ls1,
\end{align*}
which further implies that $\wz{\az}\in H^p(\rn)$ and $\|\wz{\az}\|_{H^p(\rn)}\ls1$.
This finishes the proofs of both \eqref{eq4.29} and \eqref{eq4.30}, and hence of \eqref{eq4.22}.

By \eqref{eq4.9} and \eqref{eq4.22}, we obtain
\begin{equation*}
\lf[H^p_{L_D}(\boz)\cap L^2(\boz)\r]=\lf[H^p_{r}(\boz)\cap L^2(\boz)\r],
\end{equation*}
which, together with the fact that $H^p_{L_D}(\boz)\cap L^2(\boz)$ and $H^p_{r}(\boz)\cap L^2(\boz)$ are dense,
respectively, in the spaces $H^p_{L_D}(\boz)$ and $H^p_{r}(\boz)$, and a density argument,
implies that $H^p_{L_D}(\boz)$ and $H^p_{r}(\boz)$ coincide with equivalent quasi-norms.
This finishes the proof of Theorem \ref{t1.2}.
\end{proof}

\section{Proof of Theorem \ref{t1.3}}\label{s5}

 In this section, we prove Theorem \ref{t1.3}. We begin with establishing the following
estimates for the kernels of the family $\{(tL_D)^ke^{-tL_D}\}_{t>0}$ of operators.

\begin{lemma}\label{l5.1}
Let $n\ge2$, $\boz\subset\rn$ be a bounded {\rm NTA} domain, the real-valued, bounded, and measurable matrix $A$
satisfy \eqref{eq1.3}, and $L_D$ be as in \eqref{eq1.4}. Assume that $p_0\in(2,\fz)$ is as in Lemma
\ref{l3.6} and $q\in(2,p_0)$. For any given $k\in\nn$, denote by $\{K_{t,\,k}^{L_D}\}_{t>0}$ the kernels
of the family $\{(tL_D)^ke^{-tL_D}\}_{t>0}$ of operators. Then there exist positive constants $C$
and $c$, depending on $n,\,p$, $k$, and $\boz$, such that, for any $r\in(0,\diam(\boz))$, $y\in\boz$,
and $t\in(0,\fz)$,
$$
\lf[\int_{\{x\in\boz:\ r\le|x-y|\le2r\}}\lf|\nabla_xK_{t,\,k}^{L_D}(x,y)\r|^q\,dx\r]^{1/q}\le
Cr^{\frac{n}{q}-1}t^{-\frac{n}{2}}e^{-\frac{r^2}{ct}}.
$$
\end{lemma}

\begin{proof}
Assume that $f\in L^2(\boz)$ and $u\in W^{1,2}_0(\boz)$ is a weak solution of
the Dirichlet boundary value problem \eqref{eq1.1}. Let $\eta\in C^\fz_{\rm c}(\rn)$. Then
\begin{equation}\label{eq5.1}
-\dive (A\nabla(u\eta))=-\dive (uA\nabla\eta)-A\nabla u\cdot\nabla\eta+f\eta
\end{equation}
in the sense of \eqref{eq1.8}. Indeed, for any $\fai\in C^\fz_{\rm c}(\boz)$,
\begin{align*}
&\int_\boz A(x)\nabla(u\eta)(x)\cdot\nabla\fai(x)\,dx\\
&\quad=\int_\boz A(x)\eta(x)\nabla u(x)\cdot\nabla\fai(x)\,dx
+\int_{\boz}A(x)u(x)\nabla\eta(x)\cdot\nabla\fai(x)\,dx\\ \nonumber
&\quad=\int_\boz A(x)\nabla u(x)\cdot\nabla(\eta\fai)(x)\,dx-
\int_\boz A(x)\nabla u(x)\cdot\nabla\eta(x)\fai(x)\,dx\\
&\quad\quad+\int_{\boz}A(x)u(x)\nabla\eta(x)\cdot\nabla\fai(x)\,dx\\ \nonumber
&\quad=\int_\boz f(x)\eta(x)\fai(x)\,dx-\int_\boz A(x)\nabla u(x)\cdot\nabla\eta(x)\fai(x)\,dx\\
&\quad\quad+\int_{\boz}A(x)u(x)\nabla\eta(x)\cdot\nabla\fai(x)\,dx,
\end{align*}
which implies that \eqref{eq5.1} holds true.

Let $t\in(0,\fz)$, $y\in\boz$, $r\in(0,\diam(\boz))$, $u_t:=K_{t,\,k}^{L_D}(\cdot,y)$,
and $f:=-\frac{d}{dt}K_{t,\,k}^{L_D}(\cdot,y)$. Then $L_D u_t=f$ in the sense of \eqref{eq1.8}.
Take $\eta\in C^\fz_{\rm c}(\rn)$ satisfying $\eta\equiv1$ on $\{x\in\rn:\ r\le|x-y|\le2r\}$,
$$\supp(\eta)\subset \lf\{x\in\rn:\ \frac{5r}{6}\le|x-y|\le\frac{13r}{6}\r\},$$
and $|\nabla\eta|\ls r^{-1}$. Assume that $q\in(2,p_0)$ and $p\in(1,n)$ satisfy
$\frac{1}{p}-\frac{1}{q}=\frac1n$. Then, by Lemmas \ref{l3.6} and \ref{l3.7}, and \eqref{eq5.1},
we conclude that
\begin{equation}\label{eq5.2}
\|\nabla(u_t \eta)\|_{L^q(\boz;\rn)}\ls\|u_tA\nabla\eta\|_{L^q(\boz;\rn)}+
\|A\nabla u_t\cdot\nabla\eta\|_{L^p(\boz)}+\|f\eta\|_{L^p(\boz)}.
\end{equation}
Let $S_1:=\{x\in\boz:\ r\le|x-y|\le2r\}$ and $S_2:=\{x\in\boz:\ \frac{5r}{6}\le|x-y|\le\frac{13r}{6}\}$.
Then, from \eqref{eq5.2}, we deduce that
\begin{equation}\label{eq5.3}
\|\nabla u_t\|_{L^q(S_1;\rn)}\ls r^{-1}\|u_t\|_{L^q(S_2)}+r^{-1}\|\nabla u_t\|_{L^p(S_2;\rn)}
+\lf\|\frac{du_t}{dt}\r\|_{L^p(S_2)}.
\end{equation}
We first assume that $p\le 2$. Similarly to the proof of \cite[Proposition 16]{at01a}, we have
$$
\|\nabla u_t\|_{L^2(S_2;\rn)}\ls t^{-\frac{1}{2}-\frac{n}{4}}\lf(\frac{r}
{t^{1/2}}\r)^{\frac{n-2}{2}}e^{-\frac{r^2}{ct}},
$$
which, together with the H\"older inequality, implies that
\begin{equation}\label{eq5.4}
r^{-1}\|\nabla u_t\|_{L^p(S_2;\rn)}\ls r^{-1}\|\nabla u_t\|_{L^2(S_2;\rn)}
|S_2|^{\frac{1}{p}-\frac{1}{2}}\ls r^{\frac{n}{q}-1}t^{-\frac{n}{2}}e^{-\frac{r^2}{ct}}.
\end{equation}
Furthermore, by \cite[Theorem 6.17]{o05} and \eqref{eq1.5}, we find that, for any $x\in\boz$,
\begin{equation}\label{eq5.5}
|u_t(x)|+t\lf|\frac{du_t(x)}{dt}\r|\ls t^{-\frac{n}{2}}e^{-\frac{|x-y|^2}{ct}},
\end{equation}
which further implies that
\begin{align*}
r^{-1}\|u_t\|_{L^q(S_2)}+\lf\|\frac{du_t}{dt}\r\|_{L^p(S_2)}
&\ls r^{-1}r^{\frac{n}{q}}t^{-\frac{n}{2}}e^{-\frac{r^2}{ct}}
+t^{-1}r^{\frac{n}{p}}t^{-\frac{n}{2}}e^{-\frac{r^2}{ct}}\\ \nonumber
&\ls r^{\frac{n}{q}-1}t^{-\frac{n}{2}}e^{-\frac{r^2}{ct}}.
\end{align*}
From this, \eqref{eq5.3}, and \eqref{eq5.4}, it follows that
\begin{equation}\label{eq5.6}
\|\nabla u_t\|_{L^q(S_1;\rn)}\ls r^{\frac{n}{q}-1}t^{-\frac{n}{2}}e^{-\frac{r^2}{ct}}.
\end{equation}
This finishes the proof of the present lemma in the case that $p\le2$.

When $p>2$, take $i_0\in\nn$ and $1<p_{i_0}\le2<p_{i_0-1}<\cdots <p_1:=p<q$
such that $\frac{1}{p_{i+1}}-\frac{1}{p_i}=\frac{1}{n}$ for any $i\in\{1,\,\ldots,\,i_0-1\}$.
Then, using a simple iteration argument, \eqref{eq5.3}, \eqref{eq5.5}, and \eqref{eq5.6},
we conclude that \eqref{eq5.6} also holds true in this case that $p>2$.
This finishes the proof of Lemma \ref{l5.1}.
\end{proof}

Moreover, to show Theorem \ref{t1.3}, we need the following uniform boundedness of
the family $\{\sqrt{t}\nabla e^{-tL_D}\}_{t>0}$ of operators on $L^p(\boz)$, whose proof is
similar to that of \cite[Proposition 22]{at98}; we omit the details here.

\begin{lemma}\label{l5.2}
Let $n\ge2$, $\boz\subset\rn$ be a bounded {\rm NTA} domain, the real-valued, bounded,
and measurable matrix $A$ satisfy \eqref{eq1.3}, and $L_D$ be as in \eqref{eq1.4}.
Then there exists an $\epz_0\in(0,\fz)$, depending only on $n$ and $\mu_0$, such that,
for any given $p\in(2-\epz_0,2+\epz_0)$, the family $\{\sqrt{t}\nabla
e^{-tL_D}\}_{t>0}$ of operators is uniformly bounded on $L^p(\boz)$.
\end{lemma}

Now, we prove Theorem \ref{t1.3} by using Lemmas \ref{l3.6}, \ref{l5.1}, and \ref{l5.2}.

\begin{proof}[Proof of Theorem \ref{t1.3}]
We first show (i). By Lemma \ref{l3.6}, we find that the conclusion of (i)
holds true when $q\in(2,p_0)$. We assume that (ii) holds true, and use the conclusion of (ii)
to finish the proof of (i). Define a linear operator $\widetilde{\nabla L^{-1}_D}$ on $\rn$ as follows.
Let $p\in[1,n)$ and $q\in[\frac{n}{n-1},\fz)$ satisfy $\frac{1}{p}-\frac{1}{q}=\frac{1}{n}$.
For any $f\in L^p(\rn)$ when $p\in(1,n)$, or $f\in H^1(\rn)$, and any $x\in\rn$, let
\begin{equation*}
\widetilde{\nabla L^{-1}_D}(f)(x):=
\begin{cases}
\nabla L^{-1}_D(f|_\boz)(x)\ \ &\text{when}\ \ x\in\boz,\\
0 \ \ &\text{when}\ \ x\in\boz^\complement.
\end{cases}
\end{equation*}
Then, from Lemma \ref{l3.6}, it follows that, when $q\in(2,p_0)$,
\begin{equation}\label{eq5.7}
\lf\|\widetilde{\nabla L^{-1}_D}(f)\r\|_{L^q(\rn;\rn)}=\lf\|\nabla L^{-1}_D(f|_\boz)\r\|_{L^q(\boz;\rn)}
\ls\|f|_\boz\|_{L^p(\boz)}\ls\|f\|_{L^p(\rn)}.
\end{equation}
Thus, the operator $\widetilde{\nabla L^{-1}_D}$ is bounded from $L^p(\rn)$ to $L^q(\rn;\rn)$.
Moreover, by (ii), we conclude that
\begin{equation}\label{eq5.8}
\lf\|\widetilde{\nabla L^{-1}_D}(f)\r\|_{L^{\frac{n}{n-1}}(\rn;\rn)}=\lf\|\nabla L^{-1}_D
(f|_\boz)\r\|_{L^\frac{n}{n-1}(\boz;\rn)}\ls\|f|_\boz\|_{H^1_r(\boz)}\ls\|f\|_{H^1(\rn)},
\end{equation}
which further implies that $\widetilde{\nabla L^{-1}_D}$ is bounded from $H^1(\rn)$ to $L^{\frac{n}{n-1}}(\rn;\rn)$.
Recall that, for any $q\in(1,\fz)$, the Hardy space $H^q(\rn)$ is just the Lebesgue space $L^q(\rn)$.
From this, \eqref{eq5.7}, \eqref{eq5.8}, and the complex interpolation theory of Hardy spaces on $\rn$
(see, for instance, \cite[Theorem 8.1 and (9.3)]{kmm07}), we deduce that $\widetilde{\nabla L^{-1}_D}$
is bounded from $L^p(\rn)$ to $L^q(\rn;\rn)$ for any given $p\in(1,\frac{np_0}{n+p_0})$ and $q\in(\frac{n}{n-1},p_0)$
satisfying $\frac{1}{p}-\frac{1}{q}=\frac{1}{n}$, which, combined with the definition of $\widetilde{\nabla L^{-1}_D}$,
implies that the operator $\nabla L^{-1}_D$ is bounded from $L^p(\boz)$ to $L^q(\boz;\rn)$ for any given $p\in(1,\frac{np_0}{n+p_0})$
and $q\in(\frac{n}{n-1},p_0)$ satisfying $\frac{1}{p}-\frac{1}{q}=\frac{1}{n}$. This finishes the proof of (i).

Now, we prove (ii). Let $p\in(\frac{n}{n+1},1]$ and $q\in(1,\frac{n}{n-1}]$ satisfy $\frac{1}{p}-\frac{1}{q}=\frac{1}{n}$,
and $f\in H^p_{L_D}(\boz)\cap L^2(\boz)$. Take $1<p_1<2<q_1<\min\{p_0,2+\epz_0\}$ such that
$\frac{1}{q_1}-\frac{1}{p_1}=\frac{1}{q}-\frac{1}{p}$, where $p_0$ and $\epz_0$ are, respectively,
as in Lemmas \ref{l3.6} and \ref{l5.2}. Let $\epz\in(\frac{n}{p},\fz)$ and $M\in\nn\cap(\max\{\frac{n}{2p},
\frac{1+\epz}{2}\},\fz)$. Then there exist $\{\lz_j\}_{j=1}^\fz\subset\cc$ and a sequence $\{\az_j\}_{j=1}^\fz$ of $(p,\,q_1,\,M,\,\epz)_{L_D}$-molecules associated, respectively, with the balls $\{B_{\boz,\,j}\}_{j=1}^\fz$
such that
\begin{equation}\label{eq5.9}
f=\sum_{j=1}^\fz\lz_j\az_j
\end{equation}
in $L^2(\boz)$, and
\begin{equation}\label{eq5.10}
\|f\|_{H^p_{L_D}(\boz)}\sim\lf(\sum_{j=0}^\fz|\lz_j|^p\r)^{1/p},
\end{equation}
where, for any $j\in\nn$, $B_{\boz,\,j}:=B_j\cap\boz$, and $B_j:=B(x_j,r_j)$ with $x_j\in\boz$ and
$r_j\in(0,\diam(\boz))$ is a ball of $\rn$. To finish the proof of (ii), it suffices to prove that,
for any $(p,\,q_1,\,M,\,\epz)_{L_D}$-molecule $\az$ associated with the ball $B_\boz:=B\cap\boz$,
\begin{equation}\label{eq5.11}
\lf\|\nabla L^{-1}_D(\az)\r\|_{L^q(\boz;\rn)}\ls1.
\end{equation}
Indeed, if \eqref{eq5.11} holds true, then, by \eqref{eq5.9}, \eqref{eq5.10}, \eqref{eq5.11},
and $q>1\ge p$, we find that
\begin{align*}
\|\nabla u\|_{L^q(\boz;\rn)}&=\lf\|\nabla L^{-1}_D(f)\r\|_{L^q(\boz;\rn)}\le\sum_{j=1}^\fz|\lz_j|
\lf\|\nabla L^{-1}_D(\az_j)\r\|_{L^q(\boz;\rn)}\\
&\ls\sum_{j=0}^\fz|\lz_j|\ls\lf(\sum_{j=0}^\fz|\lz_j|^p\r)^{1/p}\sim\|f\|_{H^p_{L_D}(\boz)},
\end{align*}
which implies that (ii) holds true.

Next, we prove \eqref{eq5.11}. From Lemma \ref{l3.6}, it follows that $\nabla L^{-1}_D$ is bounded from
$L^{p_1}(\boz)$ to $L^{q_1}(\boz;\rn)$, which, together with the H\"older inequality and $\frac{1}{p_1}-
\frac{1}{q_1}=\frac{1}{p}-\frac{1}{q}$, implies that
\begin{align}\label{eq5.12}
\lf\|\nabla L^{-1}_D(\az)\r\|_{L^q(8B_\boz;\rn)}&\le\lf\|\nabla L^{-1}_D(\az)\r\|_{L^{q_1}
(8B_\boz;\rn)}\lf|8B_\boz\r|^{\frac{1}{q}-\frac{1}{q_1}}\\ \nonumber
&\ls\|\az\|_{L^{p_1}(8B_\boz)}|8B_\boz|^{\frac{1}{p}-\frac{1}{p_1}}\ls1.
\end{align}
For any $j\in\nn$ with $j\ge3$, let $\wz{S}_j(B_\boz):=(2^{j+3}B_\boz)\backslash(2^{j-3}B_\boz)$ and
$E_j(B_\boz):=\boz\backslash\wz{S}_j(B_\boz)$. Then, by the equation
$$\nabla L^{-1}_D(\az)=\int_0^\fz\nabla e^{-tL_D}(\az)\,dt
$$
and the Minkowski inequality, we conclude that
\begin{align}\label{eq5.13}
\lf\|\nabla L^{-1}_D(\az)\r\|_{L^q(S_j(B_\boz);\rn)}&\le\int_0^{r_B^2}\lf\|\nabla e^{-tL_D}(\az)
\r\|_{L^{q}(S_j(B_\boz);\rn)}\,dt+\int_{r_B^2}^\fz\cdots\\ \nonumber
&=:{\rm I}+{\rm II}.
\end{align}
For the term ${\rm I}$, we have
\begin{align}\label{eq5.14}
{\rm I}&=\int_0^{r_B^2}\lf\|\nabla e^{-tL_D}\lf(\az\mathbf{1}_{\wz{S}_j(B_\boz)}\r)\r\|_{L^{q}(S_j(B_\boz);\rn)}\,dt\\ \nonumber
&\quad+\int_0^{r_B^2}\lf\|\nabla e^{-tL_D}\lf(\az\mathbf{1}_{E_j(B_\boz)}\r)\r\|_{L^{q}(S_j(B_\boz);\rn)}\,dt\\ \nonumber
&=:{\rm I}_1+{\rm I}_2.
\end{align}
From Lemma \ref{l5.2}, the H\"older inequality, and $\frac{1}{p}-\frac{1}{q}=\frac{1}{n}$, we deduce that
\begin{align}\label{eq5.15}
{\rm I}_1&\ls\int_0^{r_B^2}\lf\|\nabla e^{-tL_D}\lf(\az\mathbf{1}_{\wz{S}_j(B_\boz)}\r)\r\|_{L^{q_1}(S_j(B_\boz);\rn)}
\lf|2^jB_\boz\r|^{\frac{1}{q}-\frac{1}{q_1}}\,dt\\ \nonumber
&\ls2^{-j\epz}\lf|2^jB_\boz\r|^{\frac{1}{q_1}}\lf|B_\boz\r|^{-\frac{1}{p}}\lf|2^jB_\boz\r|^{\frac{1}{q}-\frac{1}{q_1}}
\int_0^{r_B^2}t^{-1/2}\,dt\\ \nonumber
&\ls2^{-j(\epz-\frac{n}{q})}|B_\boz|^{-\frac{1}{n}}r_B\ls2^{-j(\epz-\frac{n}{q})}.
\end{align}
Moreover, by Lemma \ref{l5.1}, the Minkowski inequality, and $\frac{1}{p}-\frac{1}{q}=\frac{1}{n}$, we conclude that
\begin{align}\label{eq5.16}
{\rm I}_2&\ls\int_0^{r_B^2}\lf\|\nabla e^{-tL_D}\lf(\az\mathbf{1}_{E_j(B_\boz)}\r)\r\|_{L^{q_1}(S_j(B_\boz);\rn)}
\lf|2^jB_\boz\r|^{\frac{1}{q}-\frac{1}{q_1}}\,dt\\ \nonumber
&\ls\lf|2^jB_\boz\r|^{\frac{1}{q}-\frac{1}{q_1}}\int_0^{r_B^2}\int_{E_j(B_\boz)}\lf\|\nabla K_{t}^{L_D}
(\cdot,y)\r\|_{L^{q_1}(S_j(B_\boz);\rn)}|\az(y)|\,dydt\\ \nonumber
&\ls\lf|2^jB_\boz\r|^{\frac{1}{q}-\frac{1}{q_1}}\int_0^{r_B^2}(2^jr_B)^{\frac{n}{q_1}-1}t^{-\frac{n}{2}}
e^{-\frac{(2^jr_B)^2}{ct}}|B_\boz|^{1-\frac{1}{p}}\,dt\ls2^{-j(\epz-\frac{n}{q})}.
\end{align}
For the term ${\rm II}$, we have
\begin{align}\label{eq5.17}
{\rm II}&=\int_{r_B^2}^\fz\frac{1}{t^M}\lf\|\nabla(tL_D)^Me^{-tL_D}\lf(\lf(L^{-M}\az\r)
\mathbf{1}_{\wz{S}_j(B_\boz)}\r)\r\|_{L^{q}(S_j(B_\boz);\rn)}\,dt\\ \nonumber
&\quad+\int_{r_B^2}^\fz\frac{1}{t^M}\lf\|\nabla(tL_D)^Me^{-tL_D}\lf(\lf(L^{-M}_D\az\r)
\mathbf{1}_{E_j(B_\boz)}\r)\r\|_{L^{q}(S_j(B_\boz);\rn)}\,dt\\ \nonumber
&=:{\rm II}_1+{\rm II}_2.
\end{align}
From the facts that, for any $t\in(0,\fz)$,
$$\sqrt{t}\nabla(tL_D)^Me^{-tL_D}=
2^{M+\frac{1}{2}}\lf(\lf(\frac{t}{2}\r)^{1/2}\nabla e^{-\frac{t}{2}L_D}\r)
\lf(\frac{t}{2}L_D\r)^Me^{-\frac{t}{2}L_D},$$
$\{(\frac{t}{2}L_D)^Me^{-\frac{t}{2}L_D}\}_{t>0}$ is uniformly bounded on $L^s(\boz)$ for any
given $s\in[1,\fz)$ (see, for instance, \cite{o05}), and Lemma \ref{l5.2}, it follows that,
for any given $s\in(2-\epz_0,2+\epz_0)$, $\{\sqrt{t}\nabla(tL_D)^Me^{-tL_D}\}_{t>0}$ is uniformly bounded on $L^s(\boz)$.
By this and the H\"older inequality, we find that
\begin{align}\label{eq5.18}
{\rm II}_1&\ls\int_{r_B^2}^\fz\frac{1}{t^M}\lf\|\nabla (tL_D)^Me^{-tL_D}\lf(\lf(L^{-M}\az\r)
\mathbf{1}_{\wz{S}_j(B_\boz)}\r)\r\|_{L^{q_1}(S_j(B_\boz);\rn)}
\lf|2^jB_\boz\r|^{\frac{1}{q}-\frac{1}{q_1}}\,dt\\ \nonumber
&\ls2^{-j\epz}r_B^{2M}\lf|2^jB_\boz\r|^{\frac{1}{q_1}}|B_\boz|^{-\frac{1}{p}}
\lf|2^jB_\boz\r|^{\frac{1}{q}-\frac{1}{q_1}}\int_{r_B^2}^\fz t^{-(M+1/2)}\,dt\\ \nonumber
&\ls2^{-j(\epz-\frac{n}{q})}r^{2M}_B|B_\boz|^{-\frac{1}{n}}r_B^{-2M+1}\ls2^{-j(\epz-\frac{n}{q})}.
\end{align}
Moreover, from the H\"older inequality, the Minkowski inequality, and Lemma \ref{l5.1}, we deduce that
\begin{align}\label{eq5.19}
{\rm II}_2&\ls\int_{r_B^2}^\fz\frac{1}{t^M}\lf\|\nabla (tL_D)^Me^{-tL_D}\lf(\lf(L^{-M}_D\az\r)\mathbf{1}_{E_j(B_\boz)}\r)\r\|_{L^{q_1}(S_j(B_\boz);\rn)}
\lf|2^jB_\boz\r|^{\frac{1}{q}-\frac{1}{q_1}}\,dt\\ \nonumber
&\ls\lf|2^jB_\boz\r|^{\frac{1}{q}-\frac{1}{q_1}}\int_{r_B^2}^\fz\frac{1}{t^M}\int_{E_j(B_\boz)}
\lf\|\nabla K_{t,\,M}^{L_D}(\cdot,y)\r\|_{L^{q_1}(S_j(B_\boz);\rn)}\lf|L^{-M}_D\az(y)\r|\,dydt\\ \nonumber
&\ls\lf|2^jB_\boz\r|^{\frac{1}{q}-\frac{1}{q_1}}\int_{r_B^2}^\fz(2^jr_B)^{\frac{n}{q_1}-1}
t^{-(M+\frac{n}{2})}e^{-\frac{(2^jr_B)^2}{ct}}r_B^{2M}|B_\boz|^{1-\frac{1}{p}}\,dt\\ \nonumber
&\ls\lf|2^jB_\boz\r|^{\frac{1}{q}-\frac{1}{n}}r_B^{2M}|B_\boz|^{1-\frac{1}{p}}
\int_{r_B^2}^\fz t^{-(M+\frac{n}{2})}\lf[\frac{t}{(2^jr_B)^2}\r]^{M-1}\,dt\\ \nonumber
&\ls2^{-j(2M-1-\frac{n}{q})}\ls2^{-j(\epz-\frac{n}{q})}.
\end{align}
Thus, by \eqref{eq5.13} through \eqref{eq5.19}, we conclude that, for any $j\in\nn$ with $j\ge3$,
\begin{align}\label{eq5.20}
\lf\|\nabla L^{-1}_D(\az)\r\|_{L^q(S_j(B_\boz);\rn)}\ls2^{-j(\epz-\frac{n}{q})},
\end{align}
which, combined with \eqref{eq5.12} and $\epz>\frac{n}{p}>\frac{n}{q}$, further implies that
\begin{equation*}
\lf\|\nabla L^{-1}_D(\az)\r\|_{L^q(\boz;\rn)}\le\lf\|\nabla L^{-1}_D(\az)\r\|_{L^q(8B_\boz;\rn)}
+\sum_{j=3}^\fz\lf\|\nabla L^{-1}_D(\az)\r\|_{L^q(S_j(B_\boz);\rn)}\ls1.
\end{equation*}
Thus, \eqref{eq5.11} holds true. This finishes the proof of (ii).

Finally, we show (iii). Let $p\in(\frac{n}{n+2},\frac{n}{n+1}]$ and $q\in(\frac{n}{n+1},1]$
satisfy $\frac{1}{p}-\frac{1}{q}=\frac{1}{n}$, and $f\in H^p_{L_D}(\boz)\cap L^2(\boz)$.
Take $1<p_2<2<q_2<\min\{p_0,2+\epz_0\}$ such that
$\frac{1}{q_2}-\frac{1}{p_2}=\frac{1}{q}-\frac{1}{p}$, where $p_0$
and $\epz_0$ are, respectively, as in Lemmas \ref{l3.6} and \ref{l5.2}.
Let $\epz\in(\frac{n}{p},\fz)$ and $M\in\nn\cap(\max\{\frac{n}{2p},\frac{1+\epz}{2}\},\fz)$.
Then there exist $\{\lz_j\}_{j=1}^\fz\subset\cc$ and a sequence $\{\az_j\}_{j=1}^\fz$
of $(p,\,q_2,\,M,\,\epz)_{L_D}$-molecules associated, respectively,
with the balls $\{B_{\boz,\,j}\}_{j=1}^\fz$ such that \eqref{eq5.9} and \eqref{eq5.10} hold true.
To finish the proof of (iii), it suffices to show that, for any $(p,\,q_2,\,M,\,\epz)_{L_D}$-molecule $\az$,
the zero extension of $\nabla L^{-1}_D(\az)$ from $\boz$ to $\rn$,
denoted by $\wz{\nabla L^{-1}_D(\az)}$, is a harmless constant multiple of
a $(q,\,q_2,\,0,\,\epz)$-molecule associated with the ball $B$.
Indeed, if this claim holds true, then, by this, $p<q$, \eqref{eq5.9}, and \eqref{eq5.10}, we find that the zero extension
of $\nabla L_D^{-1}(f)$ from $\boz$ to $\rn$, denoted by $\wz{\nabla L_D^{-1}(f)}$, belongs to $H^q(\rn;\rn)$, and
\begin{align*}
\lf\|\wz{\nabla L_D^{-1}(f)}\r\|_{H^q(\rn;\rn)}\ls\lf(\sum_{j=0}^\fz|\lz_j|^q\r)^{1/q}\ls
\lf(\sum_{j=0}^\fz|\lz_j|^p\r)^{1/p}\sim\|f\|_{H^p_{L_D}(\boz)},
\end{align*}
which further implies that $\nabla u=\nabla L_D^{-1}(f)\in H^q_z(\rn;\rn)$ and
\begin{align*}
\|\nabla u\|_{H^q_z(\boz;\rn)}=\lf\|\nabla L_D^{-1}(f)\r\|_{H^q_z(\boz;\rn)}\ls\|f\|_{H^p_{L_D}(\boz)}.
\end{align*}
Thus, (iii) holds true.

Let $\az$ be a $(p,\,q_2,\,M,\,\epz)_{L_D}$-molecule. Now, we prove that $\wz{\nabla L^{-1}_D(\az)}$
is a harmless constant multiple of a $(q,\,q_2,\,0,\,\epz)$-molecule associated with the ball $B$.
Take a ball $B_0\subset\rn$ such that $\boz\subset B_0$, and
$\fai\in C^\fz_{\rm c}(\rn)$ satisfies that $\fai\equiv1$ on $B_0$.
From $\az\in L^{p_2}(\boz)$ and (i), we deduce that $L^{-1}_D(\az)\in W^{1,q_2}_0(\boz)$,
which, together with that $\fai\equiv1$ on $B_0$, further implies that
\begin{align}\label{eq5.21}
\int_{\rn}\wz{\nabla L^{-1}_D(\az)}(x)\,dx=\int_{\rn}\wz{\nabla L^{-1}_D(\az)}(x)\fai(x)\,dx
=-\int_\rn L^{-1}_D(\az)(x)\nabla\fai(x)\,dx=0.
\end{align}
Furthermore, by (i), we find that $\nabla L^{-1}_D$ is bounded from $L^{p_2}(\boz)$
to $L^{q_2}(\boz;\rn)$, which, combined with $\frac{1}{q_2}-\frac{1}{p_2}=\frac{1}{q}-\frac{1}{p}$,
implies that
\begin{align}\label{eq5.22}
\lf\|\wz{\nabla L^{-1}(\az)}\r\|_{L^{q_2}(4B;\rn)}&=
\lf\|\nabla L^{-1}(\az)\r\|_{L^{q_2}(4B_\boz;\rn)}\ls\|\az\|_{L^{p_2}(\boz)}\\ \nonumber
&\ls|B_\boz|^{\frac{1}{p_2}-\frac{1}{p}}\sim|B_\boz|^{\frac{1}{q_2}-\frac{1}{q}}
\ls|4B|^{\frac{1}{q_2}-\frac{1}{q}}.
\end{align}
Moreover, similarly to the proof of \eqref{eq5.20}, for any $j\in\nn$ with $j\ge3$, we have
\begin{align}\label{eq5.23}
\lf\|\wz{\nabla L^{-1}(\az)}\r\|_{L^{q_2}(S_j(B);\rn)}&=
\lf\|\nabla L^{-1}(\az)\r\|_{L^{q_2}(S_j(B_\boz);\rn)}\ls2^{-j\epz}
\lf|2^jB\r|^{\frac{1}{q_2}}|B|^{-\frac{1}{q}}.
\end{align}
Thus, from \eqref{eq5.21}, \eqref{eq5.22}, and \eqref{eq5.23}, it follows that
$\wz{\nabla L^{-1}(\az)}$ is a harmless constant multiple of a $(q,\,q_2,\,0,\,\epz)$-molecule.
This finishes the proof of (iii), and hence of Theorem \ref{t1.3}.
\end{proof}

\smallskip

\noindent{\bf Acknowledgements.}\quad
Sibei Yang would like to thank Professors Jun Cao and Jun Geng
for some helpful discussions on the topic of this article.

\bigskip

\noindent Sibei Yang

\medskip

\noindent School of Mathematics and Statistics, Gansu Key Laboratory of Applied Mathematics and
Complex Systems, Lanzhou University, Lanzhou 730000, People's Republic of China

\smallskip

\noindent{\it E-mail:} \texttt{yangsb@lzu.edu.cn} (S. Yang)

\bigskip

\noindent \noindent Dachun Yang (Corresponding author)

\medskip

\noindent Laboratory of Mathematics and Complex Systems (Ministry of Education of China),
School of Mathematical Sciences, Beijing Normal University, Beijing 100875, People's Republic of China

\smallskip

\noindent {\it E-mail}: \texttt{dcyang@bnu.edu.cn} (D. Yang)

\end{document}